\documentclass[12pt]{article}
\usepackage{amssymb}
\usepackage{amsthm}
\usepackage{amsmath}
\usepackage{bbm}
\usepackage{amscd}
\usepackage{stmaryrd}
\usepackage{array}
\usepackage{url}
\usepackage{hyperref}
\usepackage[latin2]{inputenc}
\usepackage{t1enc}
\usepackage{enumerate}
\usepackage[mathscr]{eucal}
\usepackage[british]{babel}
\usepackage[dvips]{graphicx}
\usepackage[dvips]{epsfig}
\usepackage{epsf}
\usepackage{color}
\usepackage{indentfirst}
\usepackage[all]{xy}
\usepackage{tikz}
	\usetikzlibrary{arrows,positioning}

\newcommand{\Ker}{\operatorname{Ker}}

\newcommand{\Zp}{\mathbb{Z}_p}

\newcommand{\val}{\operatorname{val}}
\newcommand{\diag}{\operatorname{diag}}

\newcommand{\Tr}{\operatorname{Tr}}

\newcommand{\Hom}{\operatorname{Hom}}

\newcommand{\End}{\operatorname{End}}
\newcommand{\Ind}{\operatorname{Ind}}

\newcommand{\GL}{\operatorname{GL}}

\newcommand{\res}{\operatorname{res}}
\newcommand{\id}{\operatorname{id}}

\newcommand{\pr}{\operatorname{pr}}

\newcommand{\Coker}{\operatorname{Coker}}
\newcommand{\Sub}{\operatorname{Sub}}

\newcommand{\bg}{(\hspace{-0.06cm}(}
\newcommand{\jg}{)\hspace{-0.06cm})}

\newcommand{\bs}{[\hspace{-0.04cm}[}
\newcommand{\js}{]\hspace{-0.04cm}]}
\newtheorem{thm}{Theorem}[section]

\newtheorem{pro}[thm]{Proposition}
\newtheorem{lem}[thm]{Lemma}
\newtheorem{cor}[thm]{Corollary}
\newtheorem{df}[thm]{Definition}
\newtheorem{ex}[thm]{Example}

\newtheorem{que}{Question}

\theoremstyle{definition}
\newtheorem*{rem}{Remark}
\newtheorem*{rems}{Remarks}

\begin{document}
\date{\today}
\title{Links between generalized Montr\'eal-functors}
\author{M\'arton Erd\'elyi \and Gergely Z\'abr\'adi
\footnote{Both authors wish to thank the Alfr\'ed R\'enyi Institute of Mathematics, Hungarian Academy of Sciences for its hospitality where this work was written. The second author was partially supported by a Hungarian OTKA Research grant K-100291 and by the J\'anos Bolyai Scholarship of the Hungarian Academy of Sciences.}}

\maketitle

\begin{abstract}
Let $o$ be the ring of integers in a finite extension $K/\mathbb{Q}_p$ and $G=\mathbf{G}(\mathbb{Q}_p)$ be the $\mathbb{Q}_p$-points of a $\mathbb{Q}_p$-split reductive group $\mathbf{G}$ defined over $\Zp$ with connected centre and split Borel $\mathbf{B}=\mathbf{TN}$. We show that Breuil's \cite{B} pseudocompact $(\varphi,\Gamma)$-module $D^\vee_{\xi}(\pi)$ attached to a smooth $o$-torsion representation $\pi$ of $B=\mathbf{B}(\mathbb{Q}_p)$ is isomorphic to the pseudocompact completion of the basechange $\mathcal{O_E}\otimes_{\Lambda(N_0),\ell}\widetilde{D_{SV}}(\pi)$ to Fontaine's ring (via a Whittaker functional $\ell\colon N_0=\mathbf{N}(\Zp)\to \Zp$) of the \'etale hull $\widetilde{D_{SV}}(\pi)$ of $D_{SV}(\pi)$ defined by Schneider and Vigneras \cite{SVig}.
Moreover, we construct a $G$-equivariant map from the Pontryagin dual $\pi^\vee$ to the global sections $\mathfrak{Y}(G/B)$ of the $G$-equivariant sheaf $\mathfrak{Y}$ on $G/B$ attached to a noncommutative multivariable version $D^\vee_{\xi,\ell,\infty}(\pi)$ of Breuil's $D^\vee_{\xi}(\pi)$ whenever $\pi$ comes as the restriction to $B$ of a smooth, admissible representation of $G$ of finite length.
\end{abstract}

\tableofcontents

\section{Introduction}

\subsection{Notations}

Let $G=\mathbf{G}(\mathbb{Q}_p)$ be the $\mathbb{Q}_p$-points of a $\mathbb{Q}_p$-split connected reductive group $\mathbf{G}$ defined over $\Zp$ with connected centre and a fixed split Borel subgroup $\mathbf{B}=\mathbf{TN}$. Put $B:=\mathbf{B}(\mathbb{Q}_p)$, $T:=\mathbf{T}(\mathbb{Q}_p)$, and $N:=\mathbf{N}(\mathbb{Q}_p)$. We denote by $\Phi_+$ the  set  of  roots of $T$ in $N$, by $\Delta\subset \Phi_+$ the set of simple roots, and by $u_\alpha :\mathbb G_a \to  N_\alpha$, for $\alpha \in \Phi_+$,  a  $\mathbb Q_p$-homomorphism onto the root subgroup $N_\alpha$ of $N$ such that $tu_\alpha (x) t^{-1}=  u_\alpha(\alpha (t) x)$ for $x\in \mathbb Q_p$ and $t\in T(\mathbb Q_p)$, and   $N_0=\prod_{\alpha\in \Phi_+} u_\alpha (\mathbb Z_p)$  is a subgroup of $N(\mathbb Q_p) $. We put $N_{\alpha,0}:=u_{\alpha}(\mathbb{Z}_p)$ for the image of $u_{\alpha}$ on $\mathbb{Z}_p$.
We denote by $T_{+} $ the monoid of dominant elements  $t$ in $T(\mathbb Q_p)$ such that $ \val_p(\alpha(t))\geq 0$ for all $\alpha \in \Phi_+$,   by $T_0\subset T_+$ the maximal subgroup,  by $T_{++}$ the subset of strictly dominant elements, i.e. $\val_p(\alpha(t))> 0$ for all $\alpha \in \Phi_+$, and we put $B_+=N_0T_+, B_0=N_0T_0$.  The natural conjugation action of  $T_+$ on $N_0$ extends to an action on the Iwasawa $o$-algebra $\Lambda (N_0)=o\bs N_0\js$. For $t\in T_+$ we denote this action of $t$ on $\Lambda(N_0)$ by $\varphi_t$. The map $\varphi_t\colon \Lambda(N_0)\to\Lambda(N_0)$ is an injective ring homomorphism with a distinguished left inverse $\psi_t\colon \Lambda(N_0)\to\Lambda(N_0)$ satisfying $\psi_t\circ\varphi_t=\id_{\Lambda(N_0)}$ and $\psi_t(u\varphi_t(\lambda))=\psi_t(\varphi_t(\lambda)u)=0$ for all $u\in N_0\setminus tN_0t^{-1}$ and $\lambda\in\Lambda(N_0)$.

Each simple root $\alpha$ gives a $\mathbb Q_p$-homomorphism $x_\alpha: N\to \mathbb G_a$ with   section  $u_\alpha$. We denote  by $\ell_{\alpha}:N_0\to \mathbb Z_p$, resp. $\iota_{\alpha}\colon  \mathbb Z_p\to N_0$, the restriction of $x_\alpha$, resp.\ $u_{\alpha}$, to $N_0$, resp.\ $\mathbb Z_p$.

Since the centre of $G$ is assumed to be connected, there exists a cocharacter $\xi\colon \mathbb{Q}_p^\times\to T$ such that $\alpha\circ\xi$ is the identity on $\mathbb{Q}_p^\times$ for each $\alpha\in\Delta$. We put $\Gamma:=\xi(\mathbb{Z}_p^\times)\leq T$ and often denote the action of $s:=\xi(p)$ by $\varphi=\varphi_s$.

By a smooth $o$-torsion representation $\pi$ of $G$ (resp.\ of $B=\mathbf{B}(\mathbb{Q}_p)$) we mean a torsion $o$-module $\pi$ together with a smooth (ie.\ stabilizers are open) and linear action of the group $G$ (resp.\ of $B$). 

For example,  $\mathbf{G}=\GL_n$,  $B$  is  the subgroup of upper triangular matrices,  $N$    consists of the strictly upper triangular matrices ($1$ on the diagonal),
$T$  is the diagonal subgroup,  $N_0=\mathbf{N}(\mathbb Z_p)$, the simple roots are $\alpha_1, \ldots, \alpha_{n-1}$ where  $\alpha_i(\diag(t_1,\ldots, t_n))= t_i t_{i+1}^{-1}$, $x_{ \alpha_i}$ sends a matrix to its $(i,i+1)$-coefficient, $u_{ \alpha_i}(\cdot) $ is the strictly upper triangular matrix, with $(i,i+1)$-coefficient $\cdot$ and $0$ everywhere else.

Let $\ell\colon N_0\to \mathbb{Z}_p$ (for now) any surjective group homomorphism and denote by $H_0\lhd N_0$ the kernel of $\ell$. The ring $ \Lambda_{\ell}(N_{0})$, denoted by $ \Lambda_{H_0}(N_{0})$ in \cite{SVig}, is a generalisation of the ring  $\mathcal O_{\mathcal E}$, which corresponds to  $\Lambda_{\id} (N_{0}^{(2)})$ where $N_0^{(2)}$ is the $\mathbb{Z}_p$-points of the unipotent radical of a split Borel subgroup in $\GL_2$. We refer the reader to \cite{SVig} for the proofs of some of the following claims.

The maximal ideal $\mathcal{M}(H_0)$ of  the completed group $o$-algebra $ \Lambda(H_0)=o\bs H_0\js$ is generated by $\varpi$ and by the kernel of the augmentation map $o\bs H_0\js\to o$.

The ring $ \Lambda_{\ell}(N_{0})$  is the $\mathcal{M}(H_0)$-adic completion of the localisation of  $ \Lambda (N_{0})$ with respect to the Ore subset $S_{\ell} (N_{0})$
of elements which are not in the ideal $\mathcal{M}(H_0)  \Lambda (N_{0})$.  The   ring $\Lambda ( N_{0})$ can be viewed as the ring  $\Lambda (H_0)[[X]]$ of skew Taylor series over $\Lambda (H_0)$ in the variable $X=[u]-1$  where  $u \in N_{0}$ and $\ell(u)$ is  a topological generator of $\ell (N_{0})=\mathbb{Z}_p$.  Then   $ \Lambda_{\ell}(N_{0})$ is viewed as
the ring of infinite skew Laurent series $\sum_{n \in \mathbb Z}a_{n}X^{n}$ over $\Lambda (H_0)$ in the variable $X$ with $\lim_{n\to -\infty}a_{n}=0$ for the compact topology of $\Lambda (H_0)$. For a different characterization of this ring in terms of a projective limit $\Lambda_\ell(N_0)\cong\varprojlim_{n,k}\Lambda(N_0/H_k)[1/X]/\varpi^n$ for $H_k\lhd N_0$ normal subgroups contained and open in $H_0$ satisfying $\bigcap_{k\geq 0}H_k=\{1\}$ see also \cite{Z14}.

For a finite index subgroup $\mathcal{G}_2$ in a group $\mathcal{G}_1$ we denote by $J(\mathcal{G}_1/\mathcal{G}_2)\subset \mathcal{G}_1$ a (fixed) set of representatives of the left cosets in $\mathcal{G}_1/\mathcal{G}_2$.

\subsection{General overview}\label{schvig}

By now the $p$-adic Langlands correspondence for $\GL_2(\mathbb{Q}_p)$ is very well understood through the work of Colmez \cite{Mira}, \cite{C} and others (see \cite{B1} for an overview). To review Colmez's work let $K/\mathbb{Q}_p$ be a finite extension with ring of integers $o$, uniformizer $\varpi$ and residue field $k$. The starting point is Fontaine's \cite{F} theorem that the category of $o$-torsion Galois representations of $\mathbb{Q}_p$ is equivalent to the category of torsion $(\varphi,\Gamma)$-modules over $\mathcal{O_E}=\varprojlim_h o/\varpi^h\bg X\jg $. One of Colmez's breakthroughs was that he managed to relate $p$-adic (and mod $p$) representations of $\GL_2(\mathbb{Q}_p)$ to $(\varphi,\Gamma)$-modules, too.
The so-called ``Montr\'eal-functor'' associates to a smooth $o$-torsion representation $\pi$ of the standard Borel subgroup $B_2(\mathbb{Q}_p)$ of $\GL_2(\mathbb{Q}_p)$ a torsion $(\varphi,\Gamma)$-module over $\mathcal{O_E}$. There are two different approaches to generalize this functor to reductive groups $G$ other than $\GL_2(\mathbb{Q}_p)$. We briefly recall these ``generalized Montr\'eal functors'' here.

The approach by Schneider and Vigneras \cite{SVig} starts with the set $\mathcal{B}_+(\pi)$ of generating $B_+$-subrepresentations $W\leq \pi$. The Pontryagin dual $W^\vee=\Hom_o(W,K/o)$ of each $W$ admits a natural action of the inverse monoid $B_+^{-1}$. Moreover, the action of $N_0\leq B_+^{-1}$ on $W^\vee$ extends to an action of the Iwasawa algebra $\Lambda(N_0)=o\bs N_0\js$. For $W_1,W_2\in\mathcal{B}_+(\pi)$ we also have $W_1\cap W_2\in \mathcal{B}_+(\pi)$ (Lemma 2.2 in \cite{SVig}) therefore we may take the inductive limit $D_{SV}(\pi):=\varinjlim_{W\in \mathcal{B}_+(\pi)}W^\vee$. In general, $D_{SV}(\pi)$ does not have good properties: for instance it may not admit a canonical right inverse of the $T_+$-action making $D_{SV}(\pi)$ an \'etale $T_+$-module over $\Lambda(N_0)$.
However, by taking a resolution of $\pi$ by compactly induced representations of $B$, one may consider the derived functors $D^i_{SV}$ of $D_{SV}$ for $i\geq 0$ producing \'etale $T_+$-modules $D^i_{SV}(\pi)$ over $\Lambda(N_0)$. Note that the functor $D_{SV}$ is neither left- nor right exact, but exact in the middle. The fundamental open question of \cite{SVig} whether the topological localizations $\Lambda_\ell(N_0)\otimes_{\Lambda(N_0)}D^i_{SV}(\pi)$ are finitely generated over $\Lambda_\ell(N_0)$ in case when $\pi$ comes as a restriction of a smooth admissible representation of $G$ of finite length. One can pass to usual $1$-variable \'etale $(\varphi,\Gamma)$-modules---still not necessarily finitely generated---over $\mathcal{O_E}$ via the map $\ell\colon \Lambda_\ell(N_0)\to\mathcal{O_E}$ which step is an equivalence of categories for finitely generated \'etale $(\varphi,\Gamma)$-modules (Thm.\ 8.20 in \cite{SVZ}).

More recently, Breuil \cite{B} managed to find a different approach, producing a pseudocompact (ie.\ projective limit of finitely generated) $(\varphi,\Gamma)$-module $D^\vee_\xi(\pi)$ over $\mathcal{O_E}$ when $\pi$ is killed by a power $\varpi^h$ of the uniformizer $\varpi$.  In \cite{B} (and also in \cite{SVig}) $\ell$ is a \emph{generic} Whittaker functional, namely $\ell$ is chosen to be the composite map $$\ell\colon N_0\to N_0/(N_0\cap [N,N])\cong\prod_{\alpha\in \Delta}N_{\alpha,0}\overset{\sum\limits_{\alpha\in \Delta}u_\alpha^{-1}}{\longrightarrow}\Zp\ .$$ Breuil passes right away to the space of $H_0$-invariants $\pi^{H_0}$ of $\pi$ where $H_0$ is the kernel of the group homomorphism $\ell\colon N_0\to\Zp$. By the assumption that $\pi$ is smooth, the invariant subspace $\pi^{H_0}$ has the structure of a module over the Iwasawa algebra $\Lambda(N_0/H_0)/\varpi^h\cong o/\varpi^h\bs X\js$.
Moreover, it admits a semilinear action of $F$ which is the Hecke action of $s:=\xi(p)$: For any $m\in \pi^{H_0}$ we define $$F(m):=\Tr_{H_0/sH_0s^{-1}}(sm)=\sum_{u\in J(H_0/sH_0s^{-1})}usm\ .$$
So $\pi^{H_0}$ is a module over the skew polynomial ring $\Lambda(N_0/H_0)/\varpi^h[F]$ (defined by the identity $FX=(sXs^{-1})F=((X+1)^p-1)F$). We consider those $(i)$ finitely generated $\Lambda(N_0/H_0)/\varpi^h[F]$-submodules $M\subset \pi^{H_0}$ that are $(ii)$ invariant under the action of $\Gamma$ and are $(iii)$ \emph{admissible} as a $\Lambda(N_0/H_0)/\varpi^h$-module, ie.\ the Pontryagin dual $M^{\vee}=\Hom_{o}(M,o/\varpi^h)$ is finitely generated over $\Lambda(N_0/H_0)/\varpi^h$.
Note that this admissibility condition $(iii)$ is equivalent to the usual admissibility condition in smooth representation theory, ie.\ that for any (or equivalently for a single) open subgroup $N'\leq N_0/H_0$ the fixed points $M^{N'}$ form a finitely generated module over $o$. We denote by $\mathcal{M}(\pi^{H_0})$ the---via inclusion partially ordered---set of those submodules $M\leq \pi^{H_0}$ satisfying $(i),(ii),(iii)$. Note that whenever $M_1,M_2$ are in $\mathcal{M}(\pi^{H_0})$ then so is $M_1+M_2$. It is shown in \cite{C} (see also \cite{E1} and Lemma 2.6 in \cite{B}) that for $M\in\mathcal{M}(\pi^{H_0})$ the localized Pontryagin dual $M^\vee[1/X]$ naturally admits a structure of an \'etale $(\varphi,\Gamma)$-module over $o/\varpi^h\bg X\jg$. Therefore Breuil \cite{B} defines
\begin{equation*}
D_\xi^{\vee}(\pi):=\varprojlim_{M\in\mathcal{M}(\pi^{H_0})}M^{\vee}[1/X]\ .
\end{equation*}
By construction this is a projective limit of usual $(\varphi,\Gamma)$-modules. Moreover, $D^\vee_\xi$ is right exact and compatible with parabolic induction \cite{B}. It can be characterized by the following universal property: For any (finitely generated) \'etale $(\varphi,\Gamma)$-module over $o/\varpi^h\bg X\jg\cong o/\varpi^h\bs \Zp\js[([1]-1)^{-1}]$ (here $[1]$ is the image of the topological generator of $\Zp$ in the Iwasawa algebra $o/\varpi^h\bs \Zp\js$) we may consider continuous $\Lambda(N_0)$-homomorphisms $\pi^\vee\to D$ via the map $\ell\colon N_0\to \Zp$ (in the weak topology of $D$ and the compact topology of $\pi^\vee$). These all factor through $(\pi^\vee)_{H_0}\cong (\pi^{H_0})^\vee$.
So we may require these maps be $\psi_s$- and $\Gamma$-equivariant where $\Gamma=\xi(\Zp\setminus\{0\})$ acts naturally on $(\pi^{H_0})^\vee$ and $\psi_s\colon (\pi^{H_0})^\vee\to (\pi^{H_0})^\vee$ is the dual of the Hecke-action $F\colon \pi^{H_0}\to\pi^{H_0}$ of $s$ on $\pi^{H_0}$. Any such continuous $\psi_s$- and $\Gamma$-equivariant map $f$ factors uniquely through $D^\vee_\xi(\pi)$. However, it is not known in general whether $D^\vee_\xi(\pi)$ is nonzero for smooth irreducible representations $\pi$ of $G$ (restricted to $B$).

The way Colmez goes back to representations of $\GL_2(\mathbb{Q}_p)$ requires the following construction. From any $(\varphi,\Gamma)$-module over $\mathcal{E}=\mathcal{O_E}[1/p]$ and character $\delta\colon\mathbb{Q}_p^\times\to o^\times$ Colmez constructs a $\GL_2(\mathbb{Q}_p)$-equivariant sheaf $\mathfrak{Y}\colon U\mapsto D\boxtimes_\delta U$ ($U\subseteq \mathbb{P}^1$ open) of $K$-vectorspaces on the projective space $\mathbb{P}^1(\mathbb{Q}_p)\cong\GL_2(\mathbb{Q}_p)/B_2(\mathbb{Q}_p)$. This sheaf has the following properties: $(i)$ the centre of $\GL_2(\mathbb{Q}_p)$ acts via $\delta$ on $D\boxtimes_\delta \mathbb{P}^1$; $(ii)$ we have $D\boxtimes_\delta \mathbb{Z}_p\cong D$ as a module over the monoid $\begin{pmatrix}\mathbb{Z}_p\setminus\{0\}&\mathbb{Z}_p\\0&1\end{pmatrix}$ (where we regard $\mathbb{Z}_p$ as an open subspace in $\mathbb{P}^1=\mathbb{Q}_p\cup\{\infty\}$).
Moreover, whenever $D$ is $2$-dimensional and $\delta$ is the character corresponding to the Galois representation of $\bigwedge^2D$ via local class field theory then the $G$-representation of global sections $D\boxtimes_\delta\mathbb{P}^1$ admits a short exact sequence
\begin{equation*}
0\to \Pi(\check{D})^\vee\to D\boxtimes\mathbb{P}^1\to \Pi(D)\to 0
\end{equation*}
where $\Pi(\cdot)$ denotes the $p$-adic Langlands correspondence for $\GL_2(\mathbb{Q}_p)$ and $\check{D}=\Hom(D,\mathcal{E})$ is the dual $(\varphi,\Gamma)$-module.

In \cite{SVZ} the functor $D\mapsto\mathfrak{Y}$ is generalized to arbitrary $\mathbb{Q}_p$-split reductive groups $G$ with connected centre. Assume that $\ell=\ell_\alpha: N_0\to N_{\alpha,0}\cong\mathbb{Z}_p$ is the projection onto the root subgroup corresponding to a fixed simple root $\alpha\in\Delta$. Then we have an action of the monoid $T_+$ on the ring $\Lambda_\ell(N_0)$ as we have $tH_0t^{-1}\leq H_0$ for any $t\in T_+$. Let $D$ be an \'etale $(\varphi,\Gamma)$-module finitely generated over $\mathcal{O_E}$ and choose a character $\delta\colon\Ker(\alpha)\to o^\times$. Then we may let the monoid $\xi(\mathbb{Z}_p\setminus\{0\})\Ker(\alpha)\leq T$ (containing $T_+$) act on $D$ via the character $\delta$ of $\Ker(\alpha)$ and via the natural action of $\Zp\setminus\{0\}\cong\varphi^{\mathbb{N}_0}\times\Gamma$ on $D$. This way we also obtain a $T_+$-action on $\Lambda_\ell(N_0)\otimes_{u_\alpha}D$ making $\Lambda_\ell(N_0)\otimes_{u_\alpha}D$ an \'etale $T_+$-module over $\Lambda_\ell(N_0)$.
In \cite{SVZ} a $G$-equivariant sheaf $\mathfrak{Y}$ on $G/B$ is attached to $D$ such that its sections on $\mathcal{C}_0:=N_0w_0B/B\subset G/B$ is $B_+$-equivariantly isomorphic to the \'etale $T_+$-module $(\Lambda_\ell(N_0)\otimes_{u_\alpha}D)^{bd}$ over $\Lambda(N_0)$ consisting of bounded elements in $\Lambda_\ell(N_0)\otimes_{u_\alpha}D$ (for a more detailed overview see section \ref{sheaf}).

\subsection{Summary of our results}

Our first result is the construction of a noncommutative multivariable version of $D^\vee_\xi(\pi)$. Let $\pi$ be a smooth $o$-torsion representation of $B$ such that $\varpi^h\pi=0$. The idea here is to take the invariants $\pi^{H_k}$ for a family of open normal subgroups $H_k\leq H_0$ with $\bigcap_{k\geq 0}H_k=\{1\}$. Now $\Gamma$ and the quotient group $N_0/H_k$ act on $\pi^{H_k}$ (we choose $H_k$ so that it is normalized by both $\Gamma$ and $N_0$). Further, we have a Hecke-action of $s$ given by $F_k:=\Tr_{H_k/sH_ks^{-1}}\circ(s\cdot)$. As in \cite{B} we consider the set $\mathcal{M}_k(\pi^{H_k})$ of finitely generated $\Lambda(N_0/H_k)[F_k]$-submodules of $\pi^{H_k}$ that are stable under the action of $\Gamma$ and admissible as a representation of $N_0/H_k$. In section \ref{multvarbr} we show that for any $M_k\in \mathcal{M}_k(\pi^{H_k})$ there is an \'etale $(\varphi,\Gamma)$-module structure on $M_k^\vee[1/X]$ over the ring $\Lambda(N_0/H_k)/\varpi^h[1/X]$.
So the projective limit $$D^\vee_{\xi,\ell,\infty}(\pi):=\varprojlim_{k\geq0}\varprojlim_{M_k\in\mathcal{M}_k(\pi^{H_k})}M_k^\vee[1/X]$$ is an \'etale $(\varphi,\Gamma)$-module over $\Lambda_\ell(N_0)/\varpi^h=\varprojlim_k \Lambda(N_0/H_k)/\varpi^h[1/X]$. More-over, we also give a natural isomorphism $D^\vee_{\xi,\ell,\infty}(\pi)_{H_0}\cong D^\vee_\xi(\pi)$ showing that $D^\vee_{\xi,\ell,\infty}(\pi)$ corresponds to $D^\vee_\xi(\pi)$ via (the projective limit of) the equivalence of categories in Thm.\ 8.20 in \cite{SVZ}. Moreover, the natural map $\pi^\vee\to D^\vee_{\xi,\ell}(\pi)$ factors through the projection map $D^\vee_{\xi,\ell,\infty}(\pi)\twoheadrightarrow D^\vee_{\xi,\ell}(\pi)=D^\vee_{\xi,\ell,\infty}(\pi)_{H_0}$. Note that this shows that $D^\vee_{\xi,\ell,\infty}(\pi)$ is naturally attached to $\pi$---not just simply via the equivalence of categories (loc.\ cit.)---in the sense that any $\psi$- and $\Gamma$-equivariant map from $\pi^\vee$ to an \'etale $(\varphi,\Gamma)$-module over $o/\varpi^h\bg X\jg$ factors uniquely through the corresponding multivariable $(\varphi,\Gamma)$-module. This fact is used crucially in the subsequent sections of this paper.

In section \ref{transf} we develop these ideas further and show that the natural map $\pi^\vee\to D^\vee_{\xi,\ell,\infty}(\pi)$ factors through the map $\pi^\vee\to D_{SV}(\pi)$. In fact, we show (Prop.\ \ref{1otimestildeprinj}) that $D^\vee_{\xi,\ell,\infty}(\pi)$ has the following universal property: Any continuous $\psi_s$- and $\Gamma$-equivariant map $f\colon D_{SV}(\pi)\to D$ into a finitely generated \'etale $(\varphi,\Gamma)$-module $D$ over $\Lambda_\ell(N_0)$ factors uniquely through $\pr=\pr_{\pi}\colon D_{SV}(\pi)\to D^\vee_{\xi,\ell,\infty}(\pi)$. The association $\pi\mapsto \pr_{\pi}$ is a natural transformation between the functors $D_{SV}$ and $D^\vee_{\xi,\ell,\infty}$. One application is that Breuil's functor $D^\vee_\xi$ vanishes on compactly induced representations of $B$ (see Corollary \ref{compactlyinduced}).

In order to be able to compute $D^\vee_{\xi,\ell,\infty}(\pi)$ (hence also $D^\vee_\xi(\pi)$) from $D_{SV}(\pi)$ we introduce the notion of the \emph{\'etale hull} of a $\Lambda(N_0)$-module with a $\psi$-action of $T_+$ (or of a submonoid $T_\ast\leq T_+$). Here a $\Lambda(N_0)$-module $D$ with a $\psi$-action of $T_+$ is the analogue of a $(\psi,\Gamma)$-module over $o\bs X\js$ in this multivariable noncommutative setting. The \'etale hull $\widetilde{D}$ of $D$ (together with a canonical map $\iota\colon D\to \widetilde{D}$) is characterized by the universal property that any $\psi$-equivariant map $f\colon D\to D'$ into an \'etale $T_+$-module $D'$ over $\Lambda(N_0)$ factors uniquely through $\iota$. It can be constructed as a direct limit $\varinjlim_{t\in T_+}\varphi_t^* D$ where $\varphi_t^*D=\Lambda(N_0)\otimes_{\varphi_t,\Lambda(N_0)}D$ (Prop.\ \ref{etalehull}).
We show (Thm.\ \ref{pscompdsv} and the remark thereafter) that the pseudocompact completion of $\Lambda_\ell(N_0)\otimes_{\Lambda(N_0)}\widetilde{D_{SV}}(\pi)$ is canonically isomorphic to $D^\vee_{\xi,\ell,\infty}(\pi)$ as they have the same universal property. 

In order to go back to representations of $G$ we need an \'etale action of $T_+$ on $D^\vee_{\xi,\ell,\infty}(\pi)$, not just of $\xi(\Zp\setminus\{0\})$. This is only possible if $tH_0t^{-1}\leq H_0$ for all $t\in T_+$ which is not the case for generic $\ell$. So in section \ref{nongeneric} we equip $D^\vee_{\xi,\ell,\infty}(\pi)$ with an \'etale action of $T_+$ (extending that of $\xi(\Zp\setminus\{0\})\leq T_+$) in case $\ell=\ell_\alpha$ is the projection of $N_0$ onto a root subgroup $N_{\alpha,0}\cong\Zp$ for some simple root $\alpha$ in $\Delta$. Moreover, we show (Prop.\ \ref{prpsiT+}) that the map $\pr\colon D_{SV}(\pi)\to D^\vee_{\xi,\ell,\infty}(\pi)$ is $\psi$-equivariant for this extended action, too.
Note that $D^\vee_{\xi,\ell,\infty}(\pi)$ may not be the projective limit of finitely generated \'etale $T_+$-modules over $\Lambda_\ell(N_0)$ as we do not necessarily have an action of $T_+$ on $M^\vee_\infty[1/X]$ for $M\in \mathcal{M}(\pi^{H_0})$, only on the projective limit. So the construction of a $G$-equivariant sheaf on $G/B$ with sections on $\mathcal{C}_0=N_0w_0B/B\subset G/B$ isomorphic to a dense $B_+$-stable $\Lambda(N_0)$-submodule $D^\vee_{\xi,\ell,\infty}(\pi)^{bd}$ of $D^\vee_{\xi,\ell,\infty}(\pi)$ is not immediate from the work \cite{SVZ} as only the case of finitely generated modules over $\Lambda_\ell(N_0)$ is treated in there.
However, as we point out in section \ref{sheaf} the most natural definition of bounded elements in $D^\vee_{\xi,\ell,\infty}(\pi)$ works: The $\Lambda(N_0)$-submodule $D^\vee_{\xi,\ell,\infty}(\pi)^{bd}$ is defined as the union of $\psi$-invariant compact $\Lambda(N_0)$-submodules of $D^\vee_{\xi,\ell,\infty}(\pi)$. This section is devoted to showing that the image of $\widetilde{\pr}\colon \widetilde{D_{SV}}(\pi)\to D^\vee_{\xi,\ell,\infty}(\pi)$ is contained in $D^\vee_{\xi,\ell,\infty}(\pi)^{bd}$ (Cor.\ \ref{prbd}) and that the constructions of \cite{SVZ} can be carried over to this situation (Prop.\ \ref{sheafGB}). We denote the resulting $G$-equivariant sheaf on $G/B$ by $\mathfrak{Y}=\mathfrak{Y}_{\alpha,\pi}$.

Now consider the functors $(\cdot)^\vee\colon\pi\mapsto \pi^\vee$ and the composite $$\mathfrak{Y}_{\alpha,\cdot}(G/B)\colon\pi\mapsto D^\vee_{\xi,\ell,\infty}(\pi)\mapsto \mathfrak{Y}_{\alpha,\pi}(G/B)$$ both sending smooth, admissible $o/\varpi^h$-representations of $G$ of finite length to topological representations of $G$ over $o/\varpi^h$. The main result of our paper (Thm.\ \ref{main}) is a natural transformation $\beta_{G/B}$ from $(\cdot)^\vee$ to $\mathfrak{Y}_{\alpha,\cdot}$. This generalizes Thm.\ IV.4.7 in \cite{C}. The proof of this relies on the observation that the maps $\mathcal{H}_g\colon D^\vee_{\xi,\ell,\infty}(\pi)^{bd}\to D^\vee_{\xi,\ell,\infty}(\pi)^{bd}$ in fact come from the $G$-action on $\pi^\vee$. More precisely, for any $g\in G$ and $W\in \mathcal{B}_+(\pi)$ we have maps
\begin{equation*}
(g\cdot)\colon (g^{-1}W\cap W)^\vee\to (W\cap gW)^\vee
\end{equation*}
where both $(g^{-1}W\cap W)^\vee$ and $(W\cap gW)^\vee$ are naturally quotients of $W^\vee$. We show in (the proof of) Prop.\ \ref{betag} that these maps fit into a commutative diagram
\begin{center}
\begin{tikzpicture}[xscale=1.4, yscale=2]
\node (Dbd) at (0,0) {$D^\vee_{\xi,\ell,\infty}(\pi)^{bd}$};
\node (resg1CcapC) [right=of Dbd] {$\res_{g^{-1}\mathcal{C}_0\cap\mathcal{C}_0}^{\mathcal{C}_0}(D^\vee_{\xi,\ell,\infty}(\pi)^{bd})$};
\node (resCcapgC) [right=of resg1CcapC] {$\res_{\mathcal{C}_0\cap g\mathcal{C}_0}^{\mathcal{C}_0}(D^\vee_{\xi,\ell,\infty}(\pi)^{bd})$};
\node (Wvee) [above=of Dbd] {$W^\vee$};
\node (g1WcapW) [above=of resg1CcapC] {$(g^{-1}W\cap W)^\vee$};
\node (WcapgW) [above=of resCcapgC] {$(W\cap gW)^\vee$};
\draw[->,font=\scriptsize]
(Wvee) edge (g1WcapW)
(g1WcapW) edge node[below] {$g\cdot$} (WcapgW)
(Wvee) edge node[auto] {$\pr_W$} (Dbd)
(g1WcapW) edge (resg1CcapC)
(WcapgW) edge (resCcapgC)
(Dbd) edge (resg1CcapC)
(resg1CcapC) edge node[auto] {$g\cdot$} (resCcapgC)
(Wvee) edge[out=10,in=170] (WcapgW)
(Dbd) edge[out=350,in=190] (resCcapgC);
\end{tikzpicture}
\end{center}
allowing us to construct the map $\beta_{G/B}$. The proof of Thm.\ \ref{main} is similar to that of Thm.\ IV.4.7 in \cite{C}. However, unlike that proof we do not need the full machinery of ``standard presentations'' in Ch.\ III.1 of \cite{C} which is not available at the moment for groups other than $\GL_2(\mathbb{Q}_p)$.

\section{Comparison of Breuil's functor with that of Schneider and Vigneras}

\subsection{A $\Lambda_\ell(N_0)$-variant of Breuil's functor}\label{multvarbr}

Our first goal is to associate a $(\varphi,\Gamma)$-module over $\Lambda_\ell(N_0)$ (not just over $\mathcal{O_E}$) to a smooth $o$-torsion representation $\pi$ of $G$ in the spirit of \cite{B} that corresponds to $D^{\vee}_\xi(\pi)$ via the equivalence of categories of \cite{SVZ} between $(\varphi,\Gamma)$-modules over $\mathcal{O_E}$ and over $\Lambda_\ell(N_0)$. 

Let $H_k$ be the normal subgroup of $N_0$ generated by $s^kH_0s^{-k}$, ie.\ we put $$H_k=\langle n_0s^kH_0s^{-k}n_0^{-1}\mid n_0\in N_0\rangle\ .$$ $H_k$ is an open subgroup of $H_0$ normal in $N_0$ and we have $\bigcap_{k\geq 0} H_k=\{1\}$. Denote by $F_k$ the operator $\Tr_{H_k/sH_ks^{-1}}\circ (s\cdot)$ on $\pi$ and consider the skew polynomial ring $\Lambda(N_0/H_k)/\varpi^h[F_k]$ where $F_k\lambda=(s\lambda s^{-1})F_k$ for any $\lambda\in \Lambda(N_0/H_k)/\varpi^h$. The set of finitely generated $\Lambda(N_0/H_k)[F_k]$-submodules of $\pi^{H_k}$ that are stable under the action of $\Gamma$ and admissible as a representation of $N_0/H_k$ is denoted by $\mathcal{M}_k(\pi^{H_k})$.

\begin{lem}\label{formula}
We have $F=F_0$ and $F_k\circ \Tr_{H_k/s^kH_0s^{-k}}\circ(s^k \cdot)=\Tr_{H_k/s^kH_0s^{-k}}\circ(s^k \cdot)\circ F_0$ as maps on $\pi^{H_0}$.
\end{lem}
\begin{proof}
We compute
\begin{align*}
F_k\circ \Tr_{H_k/s^kH_0s^{-k}}\circ(s^k \cdot)=\Tr_{H_k/sH_ks^{-1}}\circ(s\cdot)\circ \Tr_{H_k/s^kH_0s^{-k}}\circ(s^k \cdot)=\\
\Tr_{H_k/sH_ks^{-1}}\circ \Tr_{sH_ks^{-1}/s^{k+1}H_0s^{-k-1}}\circ(s^{k+1} \cdot)=\\
\Tr_{H_k/s^{k+1}H_0s^{-k-1}}\circ(s^{k+1} \cdot)=\\
\Tr_{H_k/s^kH_0s^{-k}}\circ \Tr_{s^kH_0s^{-k}/s^{k+1}H_0s^{-k-1}}\circ(s^{k+1} \cdot)=\\
\Tr_{H_k/s^kH_0s^{-k}}\circ(s^k\cdot)\circ \Tr_{H_0/sH_0s^{-1}}\circ(s \cdot)=\\
\Tr_{H_k/s^kH_0s^{-k}}\circ(s^k \cdot)\circ F_0\ .\ \qed
\end{align*}
\end{proof}

Note that if $M\in\mathcal{M}(\pi^{H_0})$ then $ \Tr_{H_k/s^kH_0s^{-k}}\circ(s^kM)$ is a $s^kN_0s^{-k}H_k$-subrepresentation of $\pi^{H_k}$. So in view of the above Lemma we define $M_k$ to be the $N_0$-subrepresentation of $\pi^{H_k}$ generated by $ \Tr_{H_k/s^kH_0s^{-k}}\circ(s^kM)$, ie.\ $M_k:=N_0 \Tr_{H_k/s^kH_0s^{-k}}\circ(s^kM)$. By Lemma \ref{formula} $M_k$ is a $\Lambda(N_0/H_k)/\varpi^h[F_k]$-submodule of $\pi^{H_k}$.

\begin{lem}\label{M_k}
For any $M\in\mathcal{M}(\pi^{H_0})$ the $N_0$-subrepresentation $M_k$ lies in $\mathcal{M}_k(\pi^{H_k})$.
\end{lem}
\begin{proof}
Let $\{m_1,\dots,m_r\}$ be a set of generators of $M$ as a $\Lambda(N_0/H_0)/\varpi^h[F]$-module. We claim that the elements $\Tr_{H_k/s^kH_0s^{-k}}(s^k m_i)$ ($i=1,\dots,r$) generate $M_k$ as a module over $\Lambda(N_0/H_k)/\varpi^h[F_k]$. Since both $H_k$ and $s^kH_0s^{-k}$ are normalized by $s^kN_0s^{-k}$, for any $u\in N_0$ we have 
\begin{equation}\label{s^kN_0s^{-k}linear}
\Tr_{H_k/s^kH_0s^{-k}}\circ (s^kus^{-k}\cdot)=(s^kus^{-k}\cdot)\circ\Tr_{H_k/s^kH_0s^{-k}}\ . 
\end{equation}
Therefore by continuity we also have 
\begin{equation*}
\Tr_{H_k/s^kH_0s^{-k}}\circ (s^k\lambda s^{-k}\cdot)=(s^k\lambda s^{-k}\cdot)\circ\Tr_{H_k/s^kH_0s^{-k}} 
\end{equation*}
for any $\lambda\in\Lambda(N_0/H_0)/\varpi^h$. Now writing any $m\in M$ as $m=\sum_{j=1}^r\lambda_jF^{i_j}m_j$ we compute 
\begin{align*}
\Tr_{H_k/s^kH_0s^{-k}}\circ (s^k\sum_{j=1}^r\lambda_jF^{i_j}m_j)=\sum_{j=1}^r(s^k\lambda s^{-k})F_k^{i_j}\Tr_{H_k/s^kH_0s^{-k}} (s^km_j)\in\\
\in\sum_{j=1}^r\Lambda(N_0/H_k)/\varpi^h[F_k]\Tr_{H_k/s^kH_0s^{-k}}(s^km_j) \ .
\end{align*}

For the stability under the action of $\Gamma$ note that $\Gamma$ normalizes both $H_k$ and $s^kH_0s^{-k}$ and the elements in $\Gamma$ commute with $s$.

Since $M$ is admissible as an $N_0$-representation, $s^kM$ is admissible as a representation of $s^kN_0s^{-k}$. Further by \eqref{s^kN_0s^{-k}linear} the map $\Tr_{H_k/s^kH_0s^{-k}}$ is $s^kN_0s^{-k}$-equivariant therefore its image is also admissible. Finally, $M_k$ can be written as a finite sum
\begin{equation*}
\sum_{u\in J(N_0/s^kN_0s^{-k}H_k)}u\Tr_{H_k/s^kH_0s^{-k}}(s^kM)
\end{equation*}
of admissible representations of $s^kN_0s^{-k}$ therefore the statement.
\qed\end{proof}

\begin{lem}\label{tr}
Fix a simple root $\alpha\in\Delta$ such that $\ell(N_{\alpha,0})=\mathbb{Z}_p$. Then for any $M\in\mathcal{M}(\pi^{H_0})$ the kernel of the trace map
\begin{equation}\label{trmap}
\Tr_{H_0/H_k}\colon Y_k:=\sum_{u\in J(N_{\alpha,0}/s^kN_{\alpha,0}s^{-k})}u\Tr_{H_k/s^kH_0s^{-k}}(s^kM)\to N_0F^k(M)
\end{equation}
is finitely generated over $o$. In particular, the length of $Y_k^\vee[1/X]$ as a module over $o/\varpi^h\bg X\jg$ equals the length of $M^\vee[1/X]$.
\end{lem}
\begin{proof}
Since any $u\in N_{\alpha,0}\leq N_0$ normalizes both $H_0$ and $H_k$ and we have $N_{\alpha,0}H_0=N_0$ by the assumption that $\ell(N_{\alpha,0})=\mathbb{Z}_p$, the image of the map \eqref{trmap} is indeed $N_0F^k(M)$. Moreover, by the proof of Lemma 2.6 in \cite{B} the quotient $M/N_0F^k(M)$ is finitely generated over $o$. Therefore we have $M^\vee[1/X]\cong (N_0F^k(M))^\vee[1/X]$ as a module over $o/\varpi^h\bg X\jg$. In particular, their length are equal: $$l:=\mathrm{length}_{o/\varpi^h\bg X\jg}M^\vee[1/X]=\mathrm{length}_{o/\varpi^h\bg X\jg} (N_0F^k(M))^\vee[1/X]\ .$$ We compute
\begin{align*}
l=\mathrm{length}_{o/\varpi^h\bg X\jg}M^\vee[1/X]=\mathrm{length}_{o/\varpi^h\bg \varphi^k(X)\jg}(s^kM)^\vee[1/X]\geq \\
\geq \mathrm{length}_{o/\varpi^h\bg \varphi^k(X)\jg}(\Tr_{H_k/s^kH_0s^{-k}}(s^kM))^\vee[1/X]=\\
=\mathrm{length}_{o/\varpi^h\bg X\jg}(o/\varpi^h\bs X\js\otimes_{o/\varpi^h\bs \varphi^k(X)\js}\Tr_{H_k/s^kH_0s^{-k}}(s^kM))^\vee[1/X]\geq\\
\geq \mathrm{length}_{o/\varpi^h\bg X\jg}Y_k^\vee[1/X]\ .
\end{align*}
By the existence of a surjective map \eqref{trmap} we must have equality in the above inequality everywhere. Therefore we have $\Ker(\Tr_{H_0/H_k})^\vee[1/X]=0$, which shows that $\Ker(\Tr_{H_0/H_k})$ is finitely generated over $o$, because $M$ is admissible, and so is $\Ker(\Tr_{H_0/H_k})\leq M$.
\qed\end{proof}

The kernel of the natural homomorphism $$\Lambda(N_0/H_k)/\varpi^h\to \Lambda(N_0/H_0)/\varpi\cong k\bs X\js$$ is a nilpotent prime ideal in the ring $\Lambda(N_0/H_k)/\varpi^h$. We denote the localization at this ideal by $\Lambda(N_0/H_k)/\varpi^h[1/X]$. For the justification of this notation note that any element in $\Lambda(N_0/H_k)/\varpi^h[1/X]$ can uniquely be written as a formal Laurent-series $\sum_{n\gg-\infty} a_nX^n$ with coefficients $a_n$ in the finite group ring $o/\varpi^h[H_0/H_k]$. Here $X$---by an abuse of notation---denotes the element $[u_0]-1$ for an element $u_0\in N_{\alpha,0}\leq N_0$ with $\ell(u_0)=1\in\mathbb{Z}_p$. The ring $\Lambda(N_0/H_k)/\varpi^h[1/X]$ admits a conjugation action of the group $\Gamma$ that commutes with the operator $\varphi$ defined by $\varphi(\lambda):=s\lambda s^{-1}$ (for $\lambda\in\Lambda(N_0/H_k)/\varpi^h[1/X]$).
A $(\varphi,\Gamma)$-module over $\Lambda(N_0/H_k)/\varpi^h[1/X]$ is a finitely generated module over $\Lambda(N_0/H_k)/\varpi^h[1/X]$ together with a semilinear commuting action of $\varphi$ and $\Gamma$. Note that $\varphi$ is no longer injective on the ring $\Lambda(N_0/H_k)/\varpi^h[1/X]$ for $k\geq 1$, in particular it is not flat either. However, we still call a $(\varphi,\Gamma)$-module $D_k$ over $\Lambda(N_0/H_k)/\varpi^h[1/X]$ \'etale if the natural map
\begin{equation*}
1\otimes \varphi\colon \Lambda(N_0/H_k)/\varpi^h[1/X]\otimes_{\varphi,\Lambda(N_0/H_k)/\varpi^h[1/X]}D_k\to D_k
\end{equation*}
is an isomorphism of $\Lambda(N_0/H_k)/\varpi^h[1/X]$-modules. For any $M\in\mathcal{M}(\pi^{H_0})$ we put $$M_k^\vee[1/X]:=\Lambda(N_0/H_k)/\varpi^h[1/X]\otimes_{\Lambda(N_0/H_k)/\varpi^h}M_k^\vee$$ where $(\cdot)^\vee$ denotes the Pontryagin dual $\Hom_o(\cdot,K/o)$. 

The group $N_0/H_k$ acts by conjugation on the finite $H_0/H_k\lhd N_0/H_k$. Therefore the kernel of this action has finite index. In particular, there exists a positive integer $r$ such that $s^rN_{\alpha,0}s^{-r}\leq N_0/H_k$ commutes with $H_0/H_k$. Therefore the group ring $o/\varpi^h\bg \varphi^r(X)\jg [H_0/H_k]$ is contained as a subring in $\Lambda(N_0/H_k)/\varpi^h[1/X]$.

\begin{lem}\label{induced}
As modules over the group ring $o/\varpi^h\bg \varphi^r(X)\jg [H_0/H_k]$ we have an isomorphism
\begin{equation*}
M_k^\vee[1/X]\to o/\varpi^h\bg \varphi^r(X)\jg [H_0/H_k]\otimes_{o/\varpi^h\bg \varphi^r(X)\jg}Y_k^\vee[1/X]\ .
\end{equation*}
In particular, $M_k^\vee[1/X]$ is induced as a representation of the finite group $H_0/H_k$, so the reduced (Tate-) cohomology groups $\tilde{H}^i(H',M_k^\vee[1/X])$ vanish for all subgroups $H'\leq H_0/H_k$ and $i\in\mathbb{Z}$.
\end{lem}
\begin{proof}
By the definition of $M_k$ we have a surjective $o/\varpi^h\bs \varphi^r(X)\js[H_0/H_k]$-linear map
\begin{equation*}
f\colon o/\varpi^h\bs \varphi^r(X)\js[H_0/H_k]\otimes_{o/\varpi^h\bs \varphi^r(X)\js}Y_k\to M_k
\end{equation*}
sending $\lambda\otimes y$ to $\lambda y$ for $\lambda\in o/\varpi^h\bs \varphi^r(X)\js[H_0/H_k]$ and $y\in Y_k$. By taking the Pontryagin dual of $f$ and inverting $X$ we obtain an injective $o/\varpi^h\bg \varphi^r(X)\jg[H_0/H_k]$-homomorphism
\begin{align*}
f^\vee[1/X]\colon M_k^\vee[1/X]\to (o/\varpi^h\bs \varphi^r(X)\js[H_0/H_k]\otimes_{o/\varpi^h\bs \varphi^r(X)\js}Y_k)^\vee[1/X]\cong\\
\cong o/\varpi^h\bg \varphi^r(X)\jg[H_0/H_k]\otimes_{o/\varpi^h\bg \varphi^r(X)\jg}(Y_k^\vee[1/X])\ .
\end{align*}
On the other hand, by construction the action of the group $H_0/H_k$ on the domain of $f$ is via the action on the first term which is a regular left-translation action. Therefore the $H_0/H_k$-invariants can be computed as the image of the trace map: 
$$(o/\varpi^h\bs \varphi^r(X)\js[H_0/H_k]\otimes_{o/\varpi^h\bs \varphi^r(X)\js}Y_k)^{H_0/H_k}=(\sum_{h\in H_0/H_k}h)\otimes Y_k\ .$$
The composite of $f$ with the bijection $$(\sum_{h\in H_0/H_k}h)\otimes\id_{Y_k}\colon Y_k\overset{\sim}{\to} (\sum_{h\in H_0/H_k}h)\otimes Y_k$$ is the trace map on $Y_k$ whose kernel is finitely generated over $o$ by Lemma \ref{tr}. In particular, the kernel of the restriction of $f$ to the $H_0/H_k$-invariants is finitely generated over $o$. Dually, we find that $f^\vee[1/X]$ becomes surjective after taking $H_0/H_k$-coinvariants. Since $M_k^\vee[1/X]$ is a finite dimensional representation of the finite $p$-group $H_0/H_k$ over the local artinian ring $o/\varpi^h\bg X\jg$ with residual characteristic $p$, the map $f^\vee[1/X]$ is in fact an isomorphism as its cokernel has trivial $H_0/H_k$-coinvariants.
\qed\end{proof}

Denote by $H_{k,-}/H_k$ the kernel of the group homomorphism $$s(\cdot)s^{-1}\colon N_0/H_k\to N_0/H_k\ .$$ It is a normal subgroup contained in the finite subgroup $H_0/H_k\leq N_0/H_k$ since $s(\cdot)s^{-1}$ is the multiplication by $p$ map on $N_0/H_0\cong \mathbb{Z}_p$ which is injective. If $k$ is big enough so that $H_k$ is contained in $sH_0s^{-1}$ then we have $H_{k,-}=s^{-1}H_ks$, otherwise we always have $H_{k,-}=H_0\cap s^{-1}H_ks$. The ring homomorphism $$\varphi\colon \Lambda(N_0/H_k)/\varpi^h\to\Lambda(N_0/H_k)/\varpi^h$$ factors through the quotient map $\Lambda(N_0/H_k)/\varpi^h\twoheadrightarrow\Lambda(N_0/H_{k,-})/\varpi^h$. We denote by $\tilde{\varphi}$ the induced ring homomorphism $$\tilde{\varphi}\colon \Lambda(N_0/H_{k,-})/\varpi^h\to \Lambda(N_0/H_k)/\varpi^h\ .$$ Note that $\tilde{\varphi}$ is injective and makes $\Lambda(N_0/H_k)/\varpi^h$ a free module of rank 
\begin{align*}
\nu:=|\Coker(s(\cdot)s^{-1}\colon N_0/H_k\to N_0/H_k)|=\\
=p|\Coker(s(\cdot)s^{-1}\colon H_0/H_k\to H_0/H_k)|=\\
=p|\Ker(s(\cdot)s^{-1}\colon H_0/H_k\to H_0/H_k)|=p|H_{k,-}/H_k|
\end{align*}
over $\Lambda(N_0/H_{k,-})/\varpi^h$ since the kernel and cokernel of an endomorphism of a finite group have the same cardinality.

\begin{lem}\label{dualtensor}
We have a series of isomorphisms of $\Lambda(N_0/H_k)/\varpi^h[1/X]$-mod-ules
\begin{align*}
\Tr^{-1}=\Tr^{-1}_{H_{k,-}/H_k}\colon(\Lambda(N_0/H_k)/\varpi^h\otimes_{\varphi,\Lambda(N_0/H_k)/\varpi^h}M_k)^\vee[1/X]\overset{(1)}{\to}\\
\overset{(1)}{\to}\Hom_{\Lambda(N_0/H_k),\varphi}(\Lambda(N_0/H_k),M_k^\vee[1/X]) \overset{(2)}{\to}\\
\overset{(2)}{\to}\Hom_{\Lambda(N_0/H_{k,-}),\tilde{\varphi}}(\Lambda(N_0/H_k),(M_k^\vee[1/X])^{H_{k,-}})\overset{(3)}{\to}\\
\overset{(3)}{\to}\Lambda(N_0/H_k)\otimes_{\Lambda(N_0/H_{k,-}),\tilde{\varphi}}M_k^\vee[1/X]^{H_{k,-}}\overset{(4)}{\to}\\
\overset{(4)}{\to}\Lambda(N_0/H_k)\otimes_{\Lambda(N_0/H_{k,-}),\tilde{\varphi}}(M_k^\vee[1/X])_{H_{k,-}}\overset{(5)}{\to}\\
\overset{(5)}{\to}\Lambda(N_0/H_k)/\varpi^h\otimes_{\Lambda(N_0/H_k)/\varpi^h,\varphi}M_k^\vee[1/X]\ .
\end{align*}
\end{lem}
\begin{proof}
$(1)$ follows from the adjoint property of $\otimes$ and $\Hom$. The second isomorphism follows from noting that the action of the ring $\Lambda(N_0/H_k)$ over itself via $\varphi$ factors through the quotient $\Lambda(N_0/H_{k,-})$ therefore $H_{k,-}$ acts trivially on $\Lambda(N_0/H_k)$ via this map. So any module-homomorphism $\Lambda(N_0/H_k)\to M_k^\vee[1/X]$ lands in the $H_{k,-}$-invariant part $M_k^\vee[1/X]^{H_{k,-}}$ of $M_k^\vee[1/X]$. The third isomorphism follows from the fact that $\Lambda(N_0/H_k)$ is a free module over $\Lambda(N_0/H_{k,-})$ via $\tilde{\varphi}$. The fourth isomorphism is given by (the inverse of) the trace map $\Tr_{H_{k,-}/H_k}\colon (M_k^\vee[1/X])_{H_{k,-}}\to M_k^\vee[1/X]^{H_{k,-}}$ which is an isomorphism by Lemma \ref{induced}. The last isomorphism follows from the isomorphism $(M_k^\vee[1/X])_{H_{k,-}}\cong \Lambda(N_0/H_{k,-})\otimes_{\Lambda(N_0/H_k)}M_k^\vee[1/X]$.
\qed\end{proof}

\begin{rem}
Here $\varphi$ always acted only on the ring $\Lambda(N_0/H_k)$, hence denoting $\varphi_t$ the action $n\mapsto tnt^{-1}$ for a fixed $t\in T_+$ and choosing $k$ large enough such that $tH_0t^{-1}\geq H_k$ we get analogously an isomorphism
\begin{align*}
\Tr_{t^{-1}H_kt/H_k}^{-1}\colon(\Lambda(N_0/H_k)/\varpi^h\otimes_{\varphi_t,\Lambda(N_0/H_k)/\varpi^h}M_k)^\vee[1/X]\to\\
\to\Lambda(N_0/H_k)/\varpi^h\otimes_{\Lambda(N_0/H_k)/\varpi^h,\varphi_t}M_k^\vee[1/X]\ .
\end{align*}
One of the key points of Lemma \ref{induced} is that the trace map on $M_k^\vee[1/X]$ induces a bijection between $M_k^\vee[1/X]_{H_{k,-}}$ and $M_k^\vee[1/X]^{H_{k,-}}$ as noted in the isomorphism $(4)$ above. We shall use this fact later on.
\end{rem}

We denote the composite of the five isomorphisms in Lemma \ref{dualtensor} by $\Tr^{-1}$ emphasising that all but $(4)$ are tautologies. Our main result in this section is the following generalization of Lemma 2.6 in \cite{B}.

\begin{pro}\label{M_ketale}
The map 
\begin{eqnarray}\label{1F_kvee}
\Tr^{-1}\circ(1\otimes F_k)^\vee[1/X]\colon\\
M_k^\vee[1/X]\to \Lambda(N_0/H_k)/\varpi^h[1/X]\otimes_{\varphi,\Lambda(N_0/H_k)/\varpi^h[1/X]}M_k^\vee[1/X]\notag
\end{eqnarray}
is an isomorphism of $\Lambda(N_0/H_k)/\varpi^h[1/X]$-modules. Therefore the natural action of $\Gamma$ and the operator 
\begin{eqnarray*}
\varphi \colon  M_k^\vee[1/X]&\to& M_k^\vee[1/X]\\
f&\mapsto &(\Tr^{-1}\circ(1\otimes F_k)^\vee[1/X])^{-1}(1\otimes f)
\end{eqnarray*}
make $M_k^\vee[1/X]$ into an \'etale $(\varphi,\Gamma)$-module over the ring $\Lambda(N_0/H_k)/\varpi^h[1/X]$.
\end{pro}
\begin{proof}
Since $M_k$ is finitely generated over $\Lambda(N_0/H_k)/\varpi^h[F_k]$ by Lemma \ref{M_k}, the cokernel $C$ of the map
\begin{equation}\label{1otimesFkonMk}
1\otimes F_k\colon \Lambda(N_0/H_k)/\varpi^h\otimes_{\varphi,\Lambda(N_0/H_k)/\varpi^h}M_k\to M_k
\end{equation}
is finitely generated as a module over $\Lambda(N_0/H_k)/\varpi^h$. Further, it is admissible as a representation of $N_0$ (again by Lemma \ref{M_k}), therefore $C$ is finitely generated over $o$. In particular, we have $C^\vee[1/X]=0$ showing that \eqref{1F_kvee} is injective.

For the surjectivity put $Y_k:=\sum_{u\in J(N_{\alpha,0}/s^kN_{\alpha,0}s^{-k})}u\Tr_{H_k/s^kH_0s^{-k}}(s^kM)$. This is an $o/\varpi^h\bs X\js$-submodule of $M_k$. 
By Lemma \ref{tr} we have 
\begin{align*}
\mathrm{length}_{o/\varpi^h\bg \varphi^r(X)\jg}(Y_k^\vee[1/X])=\\
=|N_{\alpha,0}:s^rN_{\alpha,0}s^{-r}|\mathrm{length}_{o/\varpi^h\bg X\jg}(Y_k^\vee[1/X])=p^rl\ .
\end{align*}
By Lemma \ref{induced} we obtain 
\begin{align*}
\mathrm{length}_{o/\varpi^h\bg \varphi^r(X)\jg}M_k^\vee[1/X]=\\
=|H_0:H_k|\cdot\mathrm{length}_{o/\varpi^h\bg \varphi^r(X)\jg}Y_k^\vee[1/X]=|H_0:H_k|p^rl\ .
\end{align*}
Consider the ring homomorphism 
\begin{equation}\label{phi}
\varphi\colon \Lambda(N_0/H_k)/\varpi^h[1/X]\to \Lambda(N_0/H_k)/\varpi^h[1/X]\ .
\end{equation}
Its image is the subring $\Lambda(sN_0s^{-1}H_k/H_k)/\varpi^h[1/\varphi(X)]$ over which the ring $\Lambda(N_0/H_k)/\varpi^h[1/X]$ is a free module of rank $\nu=|N_0:sN_0s^{-1}H_k|=p|H_{k,-}:H_k|$. So we obtain
\begin{align*}
p\mathrm{length}_{o\bg \varphi^r(X)\jg}\Lambda(N_0/H_k)/\varpi^h[1/X]\otimes_{\varphi,\Lambda(N_0/H_k)/\varpi^h[1/X]}M_k^\vee[1/X]=\\
=\mathrm{length}_{o\bg \varphi^{r+1}(X)\jg}\Lambda(N_0/H_k)/\varpi^h[1/X]\otimes_{\varphi,\Lambda(N_0/H_k)/\varpi^h[1/X]}M_k^\vee[1/X]=\\
=\nu\mathrm{length}_{o\bg \varphi^{r+1}(X)\jg}\Lambda(sN_0s^{-1}H_k/H_k)/\varpi^h[1/\varphi(X)]\\
\otimes_{\varphi,\Lambda(N_0/H_k)/\varpi^h[1/X]}M_k^\vee[1/X]\overset{(*)}{=}\\
=\nu\mathrm{length}_{o\bg \varphi^{r}(X)\jg}M_k^\vee[1/X]_{H_{k,-}}=\\
=\nu\mathrm{length}_{o\bg \varphi^{r}(X)\jg}(o/\varpi^h[H_0/H_{k,-}]\otimes_{o/\varpi^h}Y_k^\vee[1/X])=\\
=\nu|H_0:H_{k,-}|p^rl=p|H_0:H_{k}|p^rl=p\mathrm{length}_{o/\varpi^h\bg \varphi^r(X)\jg}M_k^\vee[1/X]\ .
\end{align*}
Here the equality $(*)$ follows from the fact that the map $\varphi$ induces an isomorphism between $\Lambda(N_0/H_{k,-})/\varpi^h[1/X]$ and $\Lambda(sN_0s^{-1}H_k/H_k)/\varpi^h[1/\varphi(X)]$ sending the subring $o\bg \varphi^r(X)\jg$ isomorphically onto $o\bg \varphi^{r+1}(X)\jg$.

This shows that \eqref{1F_kvee} is an isomorphism as it is injective and the two sides have equal length as modules over the artinian ring $o/\varpi^h\bg X\jg$.
\qed\end{proof}

\begin{rem}
We also obtain in particular that the map \eqref{1otimesFkonMk} has finite kernel and cokernel. Hence there exists a finite $\Lambda(N_0/H_k)/\varpi^h$-submodule $M_{k,\ast}$ of $M_k$ such that the kernel of $1\otimes F_k$ is contained in the image of $ \Lambda(N_0/H_k)/\varpi^h\otimes_{\varphi}M_{k,\ast}$ in $\Lambda(N_0/H_k)/\varpi^h\otimes_{\varphi}M_k$. We denote by $M_k^\ast$ the image of $1\otimes F_k$.
\end{rem}

Note that for $k=0$ we have $M_0=M$. Let now $0\leq j\leq k$ be two integers. By Lemma \ref{induced} the space of $H_j$-invariants of $M_k$ is equal to $\Tr_{H_j/H_k}(M_k)$ upto finitely generated modules over $o$. On the other hand, we compute
\begin{align*}
N_0F_j^{k-j}(M_j)=N_0\Tr_{H_j/s^{k-j}H_js^{j-k}}\circ (s^{k-j}\cdot)\circ\Tr_{H_j/s^jH_0s^{-j}}(s^jM)=\\
=N_0\Tr_{H_j/s^kH_0s^{-k}}(s^kM)=N_0\Tr_{H_j/H_k}\circ\Tr_{H_k/s^kH_0s^{-k}}(s^kM)=\\
=\Tr_{H_j/H_k}(N_0\Tr_{H_k/s^kH_0s^{-k}}(s^kM))=\Tr_{H_j/H_k}(M_k)
\end{align*}
since both $H_k$ and $H_j$ are normal in $N_0$ whence we have $(u\cdot)\circ\Tr_{H_j/H_k}=\Tr_{H_j/H_k}\circ(u\cdot)$ for all $u\in N_0$. So taking $H_j/H_k$-coinvariants of $M_k^\vee[1/X]$, we have a natural identification
\begin{align}
M_k^\vee[1/X]_{H_j/H_k}\cong (M_k^{H_j/H_k})^\vee[1/X]\cong\notag\\
\cong (\Tr_{H_j/H_k}(M_k))^\vee[1/X]=(N_0F_j^{k-j}(M_j))^\vee[1/X]\cong M_j^\vee[1/X]\label{coinvariants}
\end{align}
induced by the inclusion $N_0F_j^{k-j}(M_j)\subseteq M_k^{H_j}\subseteq M_k$. The last identification follows from the fact that $M_j/N_0F_j^{k-j}(M_j)$ is finitely generated over $o$ as noted in the beginning of the proof of Proposition \ref{M_ketale} applied to $j$ instead of $k$.

\begin{lem}\label{trjk}
We have $\Tr_{H_j/H_k}\circ F_k=F_j\circ \Tr_{H_j/H_k}$.
\end{lem}
\begin{proof}
We compute
\begin{align*}
\Tr_{H_j/H_k}\circ F_k=\Tr_{H_j/H_k}\circ \Tr_{H_k/sH_ks^{-1}}\circ (s\cdot)=\\
\Tr_{H_j/sH_ks^{-1}}\circ(s\cdot)=\Tr_{H_j/sH_js^{-1}}\circ\Tr_{sH_js^{-1}/sH_ks^{-1}}(s\cdot)=\\
\Tr_{H_j/sH_js^{-1}}\circ(s\cdot)\Tr_{H_j/H_k}=F_j\circ\Tr_{H_j/H_k}\ .\ \qed
\end{align*}
\end{proof}
\begin{pro}\label{limit}
The identification \eqref{coinvariants} is $\varphi$ and $\Gamma$-equivariant.
\end{pro}
\begin{proof}
For fixed $j$ it suffices to treat the case when $k$ is large enough so that we have $H_{k,-}=s^{-1}H_ks$. Indeed, for fixed $j$ and $k$ we may choose a larger integer $k'>k$ with $H_{k',-}=s^{-1}H_{k'}s$ and the $\varphi$- and $\Gamma$ equivariance of the identifications $M_k^\vee[1/X]\cong M_{k'}^\vee[1/X]_{H_{k'}/H_k}$ and $M_j^\vee[1/X]\cong M_{k'}^\vee[1/X]_{H_{k'}/H_j}$ will imply that of $$M_j^\vee[1/X]\cong M_{k'}^\vee[1/X]_{H_{k'}/H_j}= (M_{k'}^\vee[1/X]_{H_{k'}/H_j})_{H_k/H_j}\cong M_{k}^\vee[1/X]_{H_{k}/H_j}\ .$$ So from now on we assume $H_k\leq sH_0s^{-1}\leq sN_0s^{-1}$. As $\Gamma$ acts both on $M_k$ and $M_j$ by multiplication coming from the action of $\Gamma$ on $\pi$, the map \eqref{coinvariants} is clearly $\Gamma$-equivariant. In order to avoid confusion we are going to denote the map $\varphi$ on $M_k^\vee[1/X]$ (resp.\ on $M_j^\vee[1/X]$) temporarily by $\varphi_k$ (resp.\ by $\varphi_j$). Let $f$ be in $M_k^\vee$ such that its restriction to $M_{k,\ast}$ is zero (see the Remark after Prop.\ \ref{M_ketale}).
We regard $f$ as an element in $(M_k^*/M_{k,*})^\vee\leq (M_k^*)^\vee$. We are going to compute $\varphi_k(f)$ and $\varphi_j(f_{\mid\Tr_{H_j/H_k}(M_k^*)})$ explicitly and find that the restriction of $\varphi_k(f)$ to $\Tr_{H_j/H_k}(M_k^*)$ is equal to $\varphi_j(f_{\mid\Tr_{H_j/H_k}(M_k^*)})$. Note that we have an isomorphism $M_k^\vee[1/X]\cong {M_k^*}^\vee[1/X]\cong (M_k^*/M_{k,*})^\vee[1/X]$ (resp.\ $M_j^\vee[1/X]\cong \Tr_{H_j/H_k}(M_k^*)^\vee[1/X]$) obtained from the Remark after Prop.\ \ref{M_ketale}. 

Let $m\in M_k^\ast\leq M_k$ be in the form 
\begin{equation*}
m=\sum_{u\in J((N_0/H_k)/s(N_0/H_k)s^{-1})}u F_k(m_u)
\end{equation*}
with elements $m_u\in M_k$ for $u\in J((N_0/H_k)/s(N_0/H_k)s^{-1})$. By the remark after Proposition \ref{M_ketale} $M_k^\ast$ is a finite index submodule of $M_k$. Note that the elements $m_u$ are unique upto $M_{k,\ast}+\Ker(F_k)$. Therefore $\varphi_k(f)\in (M_k^\ast)^\vee$ is well-defined by our assumption that $f_{\mid M_{k,\ast}}=0$ noting that the kernel of $F_k$ equals the kernel of $\Tr_{H_{k,-}/H_k}$ since the multiplication by $s$ is injective and we have $F_k=s\circ\Tr_{H_{k,-}/H_k}$. So we compute
\begin{align}
\varphi_k(f)(m)=((1\otimes F_k)^\vee)^{-1}(\Tr_{H_{k,-}/H_k}(1\otimes f))(m)=\notag\\
=((1\otimes F_k)^\vee)^{-1}(1\otimes \Tr_{H_{k,-}/H_k}(f))(\sum_{u\in J((N_0/H_k)/s(N_0/H_k)s^{-1})}u F_k(m_u))=\notag\\
=((1\otimes F_k)^\vee)^{-1}(1\otimes \Tr_{H_{k,-}/H_k}(f))(\sum_{u}1\otimes F_k(u\otimes m_u))=\notag\\
=(1\otimes \Tr_{H_{k,-}/H_k}(f))(\sum_{u\in J((N_0/H_k)/s(N_0/H_k)s^{-1})}(u\otimes m_u))=\notag\\
=\Tr_{H_{k,-}/H_k}(f)(F_k^{-1}(u_0F_k(m_{u_0})))=f(\Tr_{H_{k,-}/H_k}((s^{-1}u_0s)m_{u_0}))\label{varphik}
\end{align}
where $u_0$ is the single element in $J(N_0/sN_0s^{-1})$ corresponding to the coset of $1$. The other terms in the above sum vanish as $1\otimes \Tr_{H_{k,-}/H_k}(f)$ is supported on $1\otimes M_k$ by definition. In order to simplify notation put $f_*$ for the restriction of $f$ to $\Tr_{H_j/H_k}(M_k)$ and
$$U:=J(N_0/sN_0s^{-1})\cap H_jsN_0s^{-1}\ .$$ Note that we have $0=\varphi_j(f_*)(uF_j(m'))$ for all $m'\in M_j$ and $$u\in J(N_0/sN_0s^{-1})\setminus U\ .$$ Therefore using Lemma \ref{trjk} we obtain
\begin{align}
\varphi_j(f_{*})(\Tr_{H_j/H_k}m)=\varphi_j(f_{*})(\Tr_{H_j/H_k}\sum_{u\in J(N_0/sN_0s^{-1})}u F_k(m_u))=\notag\\
=\varphi_j(f_{*})(\sum_{u\in J(N_0/sN_0s^{-1})}u F_j\circ\Tr_{H_j/H_k}(m_u))=\notag\\
=\sum_{u\in U}f(\Tr_{H_{j,-}/H_j}(s^{-1}\overline{u}s\Tr_{H_j/H_k}(m_u)))=\notag\\
=\sum_{u\in U}f(s^{-1}\overline{u}s\Tr_{H_{j,-}/H_k}(m_u))\label{phij}
\end{align}
where for each $u\in U$ we choose a fixed $\overline{u}$ in $sN_0s^{-1}\cap H_ju$. Note that $f(s^{-1}\overline{u}s\Tr_{H_{j,-}/H_k}(m_u))$ does not depend on this choice: If $\overline{u_1}\in sN_0s^{-1}\cap H_ju$ is another choice then we have  $(\overline{u_1})^{-1}\overline{u}\in sN_0s^{-1}\cap H_j$ whence $s^{-1}(\overline{u_1})^{-1}\overline{u}s$ lies in $H_{j,-}=N_0\cap s^ {-1}H_js$ so we have 
\begin{align*}
s^{-1}\overline{u}s\Tr_{H_{j,-}/H_k}(m_u)=s^{-1}\overline{u_1}ss^{-1}(\overline{u_1})^{-1}\overline{u}s\Tr_{H_{j,-}/H_k}(m_u)=\\
=s^{-1}\overline{u_1}s\Tr_{H_{j,-}/H_k}(m_u)\ .
\end{align*}
Moreover, the equation \eqref{phij} also shows that $\varphi_j(f_*)$ is a well-defined element in $(\Tr_{H_j/H_k}(M_k^\ast))^\vee$. On the other hand, for the restriction of $\varphi_k(f)$ to $\Tr_{H_j/H_k}(M_k)$ we compute
\begin{align*}
\varphi_k(f)(\Tr_{H_j/H_k}m)=\varphi_k(f)(\sum_{w\in J(H_j/H_k)}w\sum_{u\in J(N_0/sN_0s^{-1})}u F_k(m_u))=\\
=\sum_{w\in J(H_j/H_k)}\sum_{u\in J(N_0/sN_0s^{-1})}\varphi_k(f)(wu F_k(m_u))=\\
=\sum_{\substack{u\in U\\ w\in J(H_j/H_k)\cap (sN_0s^{-1}u^{-1})}}f(\Tr_{H_{k,-}/H_k}((s^{-1}wus)m_u))=\\
=f(\sum_{v:=s^{-1}wu{\overline{u}}^{-1}s\in J(H_{j,-}/H_{k,-})}\Tr_{H_{k,-}/H_k}\sum_{u\in U}vs^{-1}\overline{u}sm_u)=\\
=\sum_{u\in U}f(s^{-1}\overline{u}s\Tr_{H_{j,-}/H_k}(m_u))
\end{align*}
that equals $\varphi_j(f_{*})(\Tr_{H_j/H_k}m)$ by \eqref{phij}. Finally, let now $f\in M_k^\vee$ be arbitrary. Since $M_{k,\ast}$ is finite, there exists an integer $r\geq 0$ such that $X^rf$ vanishes on $M_{k,\ast}$. By the above discussion we have $\varphi_k(X^rf)(\Tr_{H_j/H_k}m)=\varphi_j(X^rf_*)(\Tr_{H_j/H_k}m)$. The statement follows noting that $\varphi(X^r)$ is invertible in the ring $\Lambda(N_0/H_j)/\varpi^h[1/X]$.
\qed\end{proof}

So we may take the projective limit $M_\infty^\vee[1/X]:=\varprojlim_k M_k^\vee[1/X]$ with respect to these quotient maps. The resulting object is an \'etale $(\varphi,\Gamma)$-module over the ring $$\varprojlim_{k} \Lambda(N_0/H_k)/\varpi^h[1/X]\cong \Lambda_\ell(N_0)/\varpi^h\ .$$ Moreover, by taking the projective limit of \eqref{coinvariants} with respect to $k$ we obtain a $\varphi$- and $\Gamma$-equivariant isomorphism $(M_\infty^\vee[1/X])_{H_j}\cong M_j^\vee[1/X]$. So we just proved
\begin{cor}\label{MinftyM}
For any object $M\in\mathcal{M}(\pi^{H_0})$ the $(\varphi,\Gamma)$-module $M^\vee[1/X]$ over $o/\varpi^h\bg X\jg$ corresponds to $M_\infty^\vee[1/X]$ via the equivalence of categories in Theorem 8.20 in \cite{SVZ}.
\end{cor}

Note that whenever $M\subset M'$ are two objects in $\mathcal{M}(\pi^{H_0})$ then we have a natural surjective map ${M'}_\infty^\vee[1/X]\twoheadrightarrow M_\infty^\vee[1/X]$. So in view of the above corollary we define
\begin{equation*}
D^\vee_{\xi,\ell,\infty}(\pi):=\varprojlim_{k\geq 0,M\in\mathcal{M}(\pi^{H_0})}M_k^\vee[1/X]=\varprojlim_{M\in\mathcal{M}(\pi^{H_0})}M_\infty^\vee[1/X]\ .
\end{equation*}

We call two elements $M,M'\in\mathcal{M}(\pi^{H_0})$ equivalent ($M\sim M'$) if the inclusions $M\subseteq M+M'$ and $M'\subseteq M+M'$ induce isomorphisms $M^\vee[1/X]\cong (M+M')^\vee[1/X]\cong {M'}^\vee[1/X]$. This is equivalent to the condition that $M$ equals $M'$ upto finitely generated $o$-modules. In particular, this is an equivalence relation on the set $\mathcal{M}(\pi^{H_0})$. Similarly, we say that $M_k,M'_k\in\mathcal{M}_k(\pi^{H_k})$ are equivalent if the inclusions $M_k\subseteq M_k+M'_k$ and $M_k'\subseteq M_k+M'_k$ induce isomorphisms $M_k^\vee[1/X]\cong (M_k+M'_k)^\vee[1/X]\cong {M'_k}^\vee[1/X]$.
\begin{pro}
The maps
\begin{eqnarray*}
M&\mapsto& N_0\Tr_{H_k/s^kH_0s^{-k}}(s^kM)\\
\Tr_{H_0/H_k}(M_k) & \mapsfrom & M_k
\end{eqnarray*}
induce a bijection between the sets $\mathcal{M}(\pi^{H_0})/\sim$ and $\mathcal{M}_k(\pi^{H_k})/\sim$. In particaular, we have 
\begin{equation*}
D^\vee_{\xi,\ell,\infty}(\pi)=\varprojlim_{k\geq 0}\varprojlim_{M_k\in\mathcal{M}_k(\pi^{H_k})}M_k^\vee[1/X]\ .
\end{equation*}
\end{pro}
\begin{proof}
We have $\Tr_{H_0/H_k}(N_0\Tr_{H_k/s^kH_0s^{-k}}(s^kM))=N_0\Tr_{H_0/s^kH_0s^{-k}}(s^kM)=N_0F^k(M)$ which is equivalent to $M$. Conversely, $$N_0\Tr_{H_k/s^kH_0s^{-k}}(s^k\Tr_{H_0/H_k}(M_k))=N_0\Tr_{H_k/s^k{H_k}s^{-k}}(s^kM_k)=N_0F_k^k(M_k)$$ is equivalent to $M_k$ as it is the image of the map $$1\otimes F_k^k\colon \Lambda(N_0/H_k)/\varpi^h\otimes_{\varphi^k,\Lambda(N_0/H_k)/\varpi^h}\to M_k$$ having finite cokernel.
\qed\end{proof}

We equip the pseudocompact $\Lambda_\ell(N_0)$-module $D^\vee_{\xi,\ell,\infty}(\pi)$ with the weak topology, ie.\ with the projective limit  topology of the weak topologies of $M^\vee_{\infty}[1/X]$. (The weak topology on $\Lambda_\ell(N_0)$ is defined in section 8 of \cite{SVig}.)  Recall that the sets
\begin{equation}\label{OMll'}
O(M,l,l'):=f_{M,l}^{-1}(\Lambda(N_0/H_l)\otimes_{u_\alpha}X^{l'}M^\vee[1/X]^{++})
\end{equation}
for $l,l'\geq 0$ and $M\in \mathcal{M}(\pi^{H_0})$ form a system of neighbourhoods of $0$ in the weak topology of $D^\vee_{\xi,\ell,\infty}(\pi)$. Here $f_{M,l}$ is the natural projection map $f_{M,l}\colon D^\vee_{\xi,\ell,\infty}(\pi)\twoheadrightarrow M_l^\vee[1/X]$ and $M^\vee[1/X]^{++}$ denotes the set of elements $d\in M^\vee[1/X]$ with $\varphi^n(d)\to 0$ in the weak topology of $M^\vee[1/X]$ as $n\to \infty$.

\subsection{A natural transformation from $D_{SV}$ to $D^\vee_{\xi,\ell,\infty}$}\label{transf}

In order to avoid confusion we denote by $D_{SV}(\pi)$ the $\Lambda(N_0)$-module with an action of $B_+^{-1}$ associated to the smooth $o$-torsion representation $\pi$ defined as $D(\pi)$ in \cite{SVig} (note that in \cite{SVig} the notation $V$ is used for the $o$-torsion representation that we denote by $\pi$). For a brief review of this functor see section \ref{schvig}.

\begin{lem}\label{minw}
Let $W$ be in $\mathcal{B}_+(\pi)$ and $M\in\mathcal{M}(\pi^{H_0})$. There exists a positive integer $k_0>0$ such that for all $k\geq k_0$ we have $s^kM\subseteq W$. In particular, both $M_k=N_0\Tr_{H_k/s^kH_0s^{-k}}(s^kM)$ and $N_0F^k(M)$ are contained in $W$ for all $k\geq k_0$.
\end{lem}
\begin{proof}
By the assumption that $M$ is finitely generated over $\Lambda(N_0/H_0)/\varpi^h[F]$ and $W$ is a $B_+$-subrepresentation it suffices to find an integer $s^{k_0}$ such that we have $s^{k_0}m_i$ lies in $W$ for all the generators $m_1,\dots,m_r$ of $M$. This, however, follows from Lemma 2.1 in \cite{SVig} noting that the powers of $s$ are cofinal in $T_+$.
\qed\end{proof}
In particular, we have a homomorphism $W^\vee\twoheadrightarrow M_k^\vee$ of $\Lambda(N_0)$-modules induced by this inclusion. We compose this with the localisation map $M_k^\vee\to M_k^\vee[1/X]$ and take projective limits with respect to $k$ in order to obtain a $\Lambda(N_0)$-homomorphism
\begin{equation*}
\pr_{W,M}\colon W^\vee\to M_\infty^\vee[1/X]\ .
\end{equation*}

\begin{lem}\label{prpsigamma}
The map $\pr_{W,M}$ is $\psi_s$- and $\Gamma$-equivariant.
\end{lem}
\begin{proof}
The $\Gamma$-equivariance is clear as it is given by the multiplication by elements of $\Gamma$ on both sides. For the $\psi_s$-equivariance let $k>0$ be large enough so that $H_k$ is contained in $sH_0s^{-1}\leq sN_0s^{-1}$ (ie.\ $H_{k,-}=s^{-1}H_ks$) and $M_k$ is contained in $W$. Let $f$ be in $W^\vee=\Hom_o(W,o/\varpi^h)$ such that $f_{\mid N_0sM_{k,\ast}}=0$. By definition we have $\psi_s(f)(w)=f(sw)$ for any $w\in W$. Denote the restriction of $f$ to $M_k$ by $f_{\mid M_k}$ and choose an element $m\in M_k^\ast\leq M_k$ written in the form $$m=\sum_{u\in J(N_0/sN_0s^{-1})}uF_k(m_u)=\sum_{u\in J(N_0/sN_0s^{-1})}us\Tr_{H_{k,-}/H_k}(m_u)\ .$$ Then we compute
\begin{align*}
f_{\mid M_k}(m)=\sum_{u\in J(N_0/sN_0s^{-1})}f(us\Tr_{H_{k,-}/H_k}(m_u))=\\
=\sum_{u\in J(N_0/sN_0s^{-1})}(u^{-1}f)(s\Tr_{H_{k,-}/H_k}(m_u))=\\
=\sum_{u\in J(N_0/sN_0s^{-1})}\psi_s(u^{-1}f)(\Tr_{H_{k,-}/H_k}(m_u))=\\
\overset{\eqref{varphik}}{=}\sum_{u\in J(N_0/sN_0s^{-1})}\varphi(\psi_s(u^{-1}f)_{\mid M_k})(F_k(m_u))=\\
=\sum_{u\in J(N_0/sN_0s^{-1})}u\varphi(\psi_s(u^{-1}f)_{\mid M_k})(uF_k(m_u))=\\
=\sum_{u\in J(N_0/sN_0s^{-1})}u\varphi(\psi_s(u^{-1}f)_{\mid M_k})(m)
\end{align*}
as for distinct $u,v\in J(N_0/sN_0s^{-1})$ we have $u\varphi(f_0)(vF_k(m_v))=0$ for any $f_0\in (M_k^\ast)^\vee$. So by inverting $X$ and taking projective limits with respect to $k$ we obtain
\begin{equation*}
\pr_{W,M}(f)=\sum_{u\in J(N_0/sN_0s^{-1})}u\varphi(\pr_{W,M}(\psi_s(u^{-1}f)))
\end{equation*}
as we have $(M_k^\ast)^\vee[1/X]\cong M_k^\vee[1/X]$. However, since $M_\infty^\vee[1/X]$ is an \'etale $(\varphi,\Gamma)$-module over $\Lambda_\ell(N_0)/\varpi^h$ we have a unique decomposition of $\pr_{W,M}(f)$ as
\begin{equation*}
\pr_{W,M}(f)=\sum_{u\in J(N_0/sN_0s^{-1})}u\varphi(\psi(u^{-1}\pr_{W,M}(f)))
\end{equation*}
so we must have $\psi(\pr_{W,M}(f))=\pr_{W,M}(\psi_s(f))$. For general $f\in W^\vee$ note that $N_0sM_{k,*}$ is killed by $\varphi(X^r)$ for $r\geq 0$ big enough, so we have $X^r\psi(\pr_{W,M}(f))=\psi(\pr_{W,M}(\varphi(X^r)f))=\pr_{W,M}(\psi_s(\varphi(X^r)f))=X^r\pr_{W,M}(\psi_s(f))$. The statement follows since $X^r$ is invertible in $\Lambda_\ell(N_0)$.
\qed\end{proof}

By taking the projective limit with respect to $M\in\mathcal{M}(\pi^{H_0})$ and the injective limit with respect to $W\in\mathcal{B}_+(\pi)$ we obtain a $\psi_s$- and $\Gamma$-equivariant $\Lambda(N_0)$-homomorphism 
\begin{equation*}
\pr:=\varinjlim_{W}\varprojlim_{M}\pr_{W,M}\colon D_{SV}(\pi)\to D^\vee_{\xi,\ell,\infty}(\pi)\ .
\end{equation*}

\begin{rems}
\begin{enumerate}
\item Taking Pontryagin dual of the inclusion $M_k\leq \pi$ for all $M\in \mathcal{M}(\pi^{H_k})$ and $k\geq 0$ we obtain a composite map $\pi^\vee\twoheadrightarrow M_k^\vee\to M_k^\vee[1/X]$. These are compatible with the projective limit construction therefore induce natural maps $\pi^\vee\to D^\vee_{\xi}(\pi)$ and $\pi^\vee\to D^\vee_{\xi,\ell,\infty}(\pi)$. Both of these maps factor through the map $\pi^\vee\twoheadrightarrow D_{SV}(\pi)$ by Lemma \ref{minw}.
\item The natural topology on $D_{SV}$ obtained as the quotient topology from the compact topology on $\pi^\vee$ via the surjective map $\pi^\vee\twoheadrightarrow D_{SV}(\pi)$ is compact, but may not be Hausdorff in general. However, if $\mathcal{B}_+(\pi)$ contains a minimal element (as in the case of the principal series \cite{Er}) then it is also Hausdorff. However, the map $\pr$ factors through the maximal Hausdorff quotient of $D_{SV}(\pi)$, namely $\overline{D}_{SV}(\pi):=(\bigcap_{W\in \mathcal{B}_+(\pi)} W)^\vee$. Indeed, $\pr$ is continuous and $D^\vee_{\xi,\ell,\infty}(\pi)$ is Hausdorff, so the kernel of $\pr$ is closed in $D_{SV}(\pi)$ (and contains $0$).
\item Assume that $h=1$, ie.\ $\pi$ is a smooth representation in characteristic $p$. Then $D^\vee_{\xi,\ell,\infty}(\pi)$ has no nonzero $\Lambda(N_0)/\varpi$-torsion. Hence the $\Lambda(N_0)/\varpi$-torsion part of $D_{SV}(\pi)$ is contained in the kernel of $\pr$.
\item If $D_{SV}(\pi)$ has finite rank and its torsion free part is \'etale over $\Lambda(N_0)$ then $\Lambda_\ell(N_0)\otimes_{\Lambda(N_0)}D_{SV}(\pi)$ is also \'etale and of finite rank $r$ over $\Lambda_\ell(N_0)$. Moreover, the map $\Lambda_\ell(N_0)\otimes_{\Lambda(N_0)}\pr: \Lambda_\ell(N_0)\otimes_{\Lambda(N_0)}D_{SV}(\pi)\to D_{\xi,\ell,\infty}(\pi)$ has dense image by Lemma \ref{minw}. Thus $D_{\xi,\ell,\infty}^\vee(\pi)$ has rank at most $r$ over $\Lambda_\ell(N_0)$. In particular, for $\pi$ being the principal series $D_{SV}(\pi)$ has rank $1$ and its torsion free part is \'etale over $\Lambda(N_0)$ (\cite{Er}), hence we obtained that $D^\vee_{\xi,\ell,\infty}(\pi)$ has rank $1$ over $\Lambda_\ell(N_0)$ (cf.\ Example 7.6 of \cite{B}).
\end{enumerate}
\end{rems}

One can show the above Remark 2 algebraically, too. Let $M\in \mathcal{M}(\pi^{H_0})$ be arbitrary. Then the map $1\otimes\id_{M^\vee}\colon M^\vee\to M^\vee[1/X]$ has finite kernel, so the image $(1\otimes \id_{M^\vee})(M^\vee)$ is isomorphic to $M_0^\vee$ for some finite index submodule $M_0\leq M$. Moreover, $M_0^\vee$ is a $\psi$- and $\Gamma$-invariant treillis in $D:=M^\vee[1/X]=M_0^\vee[1/X]$. Therefore the map $(1\otimes F)^\vee$ is injective on $M_0^\vee$ since it is injective after inverting $X$ and $M_0^\vee$ has no $X$-torsion. This means that $1\otimes F\colon o/\varpi^h\bs X\js\otimes_{o/\varpi^h\bs X\js,\varphi}M_0\to M_0$ is surjective, ie.\ we have $M_0=N_0F^k(M_0)$ for all $k\geq 0$. However, for any $W\in\mathcal{B}_+(\pi)$ and $k$ large enough (depending a priori on $W$) we have $N_0F^k(M_0)\subseteq W$, so we deduce $M_0\subset \cap_{W\in \mathcal{B}_+} W$.

\begin{cor}\label{compactlyinduced}
If $\pi=\Ind_{B_0}^B\pi_0$ is a compactly induced representation of $B$ for some smooth $o/\varpi^h$-representation $\pi_0$ of $B_0$ then we have $D^\vee_\xi(\pi)=0$. In particular, $D^\vee_\xi$ is not exact on the category of smooth $o/\varpi^h$-representations of $B$. (However, it may still be exact on a smaller subcategory with additional finiteness conditions.)
\end{cor}
\begin{proof}
By the $2$nd remark above the map $\pi^\vee\to D^\vee_\xi(\pi)$ factors through the maximal Hausdorff quotient $\overline{D}_{SV}(\pi)$ of $D_{SV}(\pi)$. By Lemma 3.2 in \cite{SVig}, we have $\overline{D}_{SV}(\pi)=(\bigcap_{\sigma} W_\sigma)^\vee$ where the $B_+$-subrepresentations $W_\sigma$ are indexed by order-preserving maps $\sigma\colon T_+/T_0\to \Sub(\pi_0)$ where $\Sub(\pi_0)$ is the partially order set of $B_0$-subrepresentations of $\pi_0$. The explicit description of the $B_+$-subrepresentations $W_\sigma$ (there denoted by $M_\sigma$) before Lemma 3.2 in \cite{SVig} shows that we have in fact $\bigcap_\sigma W_\sigma=\{0\}$ whence the natural map $\pi^\vee\to D^\vee_\xi(\pi)$ is zero. However, by the construction of this map this can only be zero if $D^\vee_\xi(\pi)=0$.

Since the principal series arises as a quotient of a compactly induced representation, the exactness of $D^\vee_\xi$ would imply the vanishing of $D^\vee_\xi$ on the principal series, too---which is not the case by Ex.\ 7.6 in \cite{B}.
\qed\end{proof}

\begin{pro}\label{1otimestildeprinj}
Let $D$ be an \'etale $(\varphi,\Gamma)$-module over $\Lambda_\ell(N_0)/\varpi^h$, and $f:D_{SV}(\pi)\to D$ be a continuous $\psi_s$ and $\Gamma$-equivariant $\Lambda(N_0)$-homomorphism. Then $f$ factors uniquely through $\pr$, ie.\ there exists a unique $\psi$- and $\Gamma$-equi-variant $\Lambda(N_0)$-homomorphism $\hat{f}:D^\vee_{\xi,\ell,\infty}(\pi)\to D$ such that $f=\hat{f}\circ\pr$.
\end{pro}

\begin{proof}
For the uniqueness of $\hat{f}$ note that Lemma \ref{minw} implies the density of the image of $\Lambda_\ell(N_0)\otimes D_{SV}(\pi)$ in $D^\vee_{\xi,\ell,\infty}(\pi)$ as its composite with the projection onto $M_k^\vee[1/X]$ is surjective for $k$ large enough and $M\in\mathcal{M}(\pi^{H_0})$ arbitrary. Therefore if $\hat{f}'$ is another lift then $\hat{f}-\hat{f}'$ vanishes on a dense subset whence it is zero by continuity.

At first we construct a homomorphism $\hat{f}_{H_0}:D^\vee_\xi=(D^\vee_{\xi,\ell,\infty})_{H_0}\to D_{H_0}$ such that the following diagram commutes:
\begin{equation*}
\xymatrix{
D_{SV}(\pi)\ar[rd]_f\ar[r]^{\pr} & D^\vee_{\xi,\ell,\infty}(\pi)\ar[r]^{(\cdot)_{H_0}} & D^\vee_\xi(\pi)\ar[d]^{\hat{f}_{H_0}}\\
 & D\ar[r]_{(\cdot)_{H_0}} & D_{H_0}
}
\end{equation*}

Consider the composite map $f':\pi^\vee\to D_{SV}(\pi)\stackrel{f}{\to}D\to D_{H_0}$. Note that $f'$ is continuous and $D_{H_0}$ is Hausdorff, so $\Ker(f')$ is closed in $\pi^\vee$. Therefore $M_0=(\pi^\vee/\Ker(f'))^\vee$ is naturally a subspace in $\pi$. We claim that $M_0$ lies in $\mathcal{M}(\pi^{H_0})$. Indeed, $M_0^\vee$ is a quotient of $\pi^\vee_{H_0}$, hence $M_0\leq\pi^{H_0}$ and it is $\Gamma$-invariant since $f'$ is $\Gamma$-equivariant. $M_0$ is admissible because it is discrete, hence $M_0^\vee$ is compact, equivalently finitely generated over $o/\varpi^h\bs X\js$, because $M_0^\vee$ can be identified with a $o/\varpi^h\bs X\js$-submodule of $D_{H_0}$ which is finitely generated over $o/\varpi^h\bg X\jg$. The last thing to verify is that $M$ is finitely generated over $o/\varpi^h\bs X\js[F]$, which follows from the following

\begin{lem}
Let $D$ be an \'etale $(\varphi,\Gamma)$-module over $o/\varpi^h\bg X\jg$ and $D_0\subset D$ be a $\psi$ and $\Gamma$-invariant compact (or, equivalently, finitely generated) $o/\varpi^h\bs X\js$ submodule. Then $D_0^\vee$ is finitely generated as a module over $o/\varpi^h\bs X\js[F]$ where for any $m\in D_0^\vee=\Hom_o(D_0,o/\varpi^h)$ we put $F(m)(f):=m(\psi(f))$ (for all $f\in D_0$).
\end{lem}
\begin{proof}
As the extension of finitely generated modules over a ring is again finitely generated, we may assume without loss of generality that $h=1$ and $D$ is irreducible, ie.\ $D$ has no nontrivial \'etale $(\varphi,\Gamma)$-submodule over $o/\varpi\bg X\jg$.

If $D_0=\{0\}$ then there is nothing to prove. Otherwise $D_0$ contains the smallest $\psi$ and $\Gamma$ stable $o\bs X\js $-submodule $D^\natural$ of $D$. So let $0\neq m\in D_0^\vee$ be arbitrary such that the restriction of $m$ to $D^\natural$ is nonzero and consider the $o/\varpi\bs X\js[F]$-submodule $M:=o/\varpi\bs X\js[F]m$ of $D_0^\vee$ generated by $m$. We claim that $M$ is not finitely generated over $o$. Suppose for contradiction that the elements $F^rm$ are not linearly independent over $o/\varpi$. Then we have a polynomial $P(x)=\sum_{i=0}^n a_ix^i\in o/\varpi[x]$ such that $0=P(F)m(f)=m(\sum a_i\psi^i(f))=m(P(\psi)f)$ for any $f\in D^\natural\subset D_0$.
However, $P(\psi)\colon D^\natural\to D^\natural$ is surjective by Prop.\ II.5.15.\ in \cite{Mira}, so we obtain $m_{\mid D^\natural}=0$ which is a contradiction. In particular, we obtain that $M^\vee[1/X]\neq 0$. However, note that $M^\vee[1/X]$ has the structure of an \'etale $(\varphi,\Gamma)$-module over $o/\varpi\bg X\jg$ by Lemma 2.6 in \cite{B}. Indeed, $M$ is admissible, $\Gamma$-invariant, and finitely generated over $o/\varpi\bs X\js [F]$ by construction. Moreover, we have a natural surjective homomorphism $D=D_0[1/X]=(D_0^\vee)^\vee[1/X]\to M^\vee[1/X]$ which is an isomorphism as $D$ is assumed to be irreducible. Therefore we have $(D_0^\vee/M)^\vee[1/X]=0$ showing that $D_0^\vee/M$ is finitely generated over $o$. In particular, both $M$ and $D_0^\vee/M$ are finitely generated over $o/\varpi\bs X\js [F]$ therefore so is $D_0^\vee$.
\qed\end{proof}

Now $D_0=M_0^\vee$ is a $\psi$- and $\Gamma$-invariant $o/\varpi^ h\bs X\js$-submodule of $D$ therefore we have an injection $f_0\colon M_0^ \vee[1/X]\hookrightarrow D$ of \'etale $(\varphi,\Gamma)$-modules. The map $\hat{f}_{H_0}\colon D^\vee_{\xi}\to D_{H_0}$ is the composite map $D^ \vee_{\xi}\twoheadrightarrow M_0^\vee[1/X]\hookrightarrow D$. It is well defined and makes the above diagram commutative, because the map $$\pi^\vee\to D_{SV}(\pi)\stackrel{\pr}{\to} D^\vee_{\xi,\ell,\infty}(\pi)\stackrel{(\cdot)_{H_0}}{\to}D^\vee_\xi(\pi)\to M_0^\vee[1/X]$$
is the same as $\pi^\vee\to M_0^\vee\to M_0^\vee[1/X]$.

Finally, by Corollary \ref{MinftyM} $M^\vee[1/X]$ (resp.\ $D_{H_0}$) corresponds to $M_\infty^\vee[1/X]$ (resp.\ to $D$) via the equivalence of categories in Theorem 8.20 in \cite{SVZ} therefore $f_0$ can uniquely be lifted to a $\varphi$- and $\Gamma$-equivariant $\Lambda_\ell(N_0)$-homomorphism $f_\infty\colon M_\infty^\vee[1/X]\hookrightarrow D$. The map $\hat{f}$ is defined as the composite $D^\vee_{\xi,\ell,\infty}\twoheadrightarrow M_\infty^\vee[1/X]\hookrightarrow D$. Now the image of $f-\hat{f}\circ\pr$ is a $\psi_s$-invariant $\Lambda(N_0)$-submodule in $(H_0-1)D$ therefore it is zero by Lemma 8.17 and the proof of Lemma 8.18 in \cite{SVZ}. Indeed, for any $x\in D_{SV}(\pi)$ and $k\geq 0$ we may write $(f-\hat{f}\circ\pr)(x)$ in the form $\sum_{u\in J(N_0/s^kN_0s^{-k})}u\varphi^k((f-\hat{f}\circ\pr)(\psi^k(u^{-1}x)))$ that lies in $(H_k-1)D$.
\qed\end{proof}

\subsection{\'Etale hull}

In this section we construct the \'etale hull of $D_{SV}(\pi)$: an \'etale $T_+$-module $\widetilde{D_{SV}}(\pi)$ over $\Lambda(N_0)$ with an injection $\iota:D_{SV}(\pi)\to\widetilde{D_{SV}}(\pi)$ with the following universal property: For any \'etale $(\varphi,\Gamma)$-module $D'$ over $\Lambda(N_0)$, and $\psi_s$ and $\Gamma$-equivariant map $f:D_{SV}(\pi)\to D'$, $f$ factors through $\widetilde{D_{SV}}(\pi)$, ie.\ there exists a unique $\psi$- and $\Gamma$-equivariant $\Lambda(N_0)$-homomorphism $\widetilde{f}:\widetilde{D_{SV}}(\pi)\to D'$ making the diagram
\begin{equation*}
\xymatrix{
D_{SV}(\pi)\ar[r]^{\iota}\ar[d]_f&\widetilde{D_{SV}}(\pi)\ar[dl]^{\widetilde{f}}\\
D'&
}
\end{equation*}
commutative. Moreover, if we assume further that $D'$ is an \'etale $T_+$-module over $\Lambda(N_0)$ and the map $f$ is $\psi_t$-equivariant for all $t\in T_+$ then the map $\widetilde{f}$ is $T_+$-equivariant.

\begin{df}
Let $D$ be a $\Lambda(N_0)$-module and $T_\ast\leq T_+$ be a submonoid. Assume moreover that the monoid $T_\ast$ (or in the case of $\psi$-actions the inverse monoid $T_\ast^{-1}$) acts $o$-linearly on $D$, as well.

We call the action of $T_\ast$ a $\varphi$-action (relative to the $\Lambda(N_0)$-action) and denote the action of $t$ by $d\mapsto\varphi_t(d)$, if for any $\lambda\in\Lambda(N_0)$, $t\in T_\ast$ and $d\in D$ we have $\varphi_t(\lambda d)=\varphi_t(\lambda)\varphi_t(d)$. Moreover, we say that the $\varphi$-action is \emph{injective} if for all $t\in T_\ast$ the map $\varphi_t$ is injective. The $\varphi$-action of $T_\ast$ is \emph{nondegenerate} if for all $t\in T_\ast$ we have
\begin{equation*}
D=\sum_{u\in J(N_0/tN_0t^{-1})}\mathrm{Im}(u\circ\varphi_t)=\sum_{u\in J(N_0/tN_0t^{-1})}u(\varphi_t(D))\ .
\end{equation*}

We call the action of $T_\ast^{-1}$ a $\psi$-action of $T_\ast$ (relative to the $\Lambda(N_0)$-action) and denote the action of $t^{-1}\in T_\ast^{-1}$ by $d\mapsto\psi_t(d)$, if for any $\lambda\in\Lambda(N_0)$, $t\in T_\ast$ and $d\in D$ we have $\psi_t(\varphi_t(\lambda) d)=\lambda\psi_t(d)$. Moreover, we say that the $\psi$-action of $T_\ast$ is \emph{surjective} if for all $t\in T_\ast$ the map $\psi_t$ is surjective. The $\psi$-action of $T_\ast$ is \emph{nondegenerate} if for all $t\in T_\ast$ we have
\begin{equation*}
\{0\}=\bigcap_{u\in J(N_0/tN_0t^{-1})}\mathrm{Ker}(\psi_t\circ u^{-1})\ .
\end{equation*}
The nondegeneracy is equivalent to the condition that for any $t\in T_\ast$ $\Ker(\psi_t)$ does not contain any nonzero $\Lambda(N_0)$-submodule of $D$.

We say that a $\varphi$- and a $\psi$-action of $T_\ast$ are \emph{compatible} on $D$, if
\begin{enumerate}[$(\varphi\psi)$]
\item for any $t\in T_\ast$, $\lambda\in\Lambda(N_0)$, and $d\in D$ we have $\psi_t(\lambda\varphi_t(d))=\psi_t(\lambda)d$.
\end{enumerate}
Note that with $\lambda=1$ we also have $\psi_t\circ\varphi_t=\mathrm{id}_D$ for any $t\in T_\ast$ assuming $(\varphi\psi)$. 

We also consider $\varphi$- and $\psi$-actions of the monoid $\Zp\setminus\{0\}$ on $\Lambda(N_0)$-modules via the embedding $\xi\colon \Zp\setminus\{0\}\to T_+$. Modules with a $\varphi$-action (resp.\ $\psi$-action) of $\Zp\setminus\{0\}$ are called $(\varphi,\Gamma)$-modules (resp.\ $(\psi,\Gamma)$-modules). 
\end{df}

For example, the natural $\varphi$- and $\psi$-actions of $T_+$ on $\Lambda(N_0)$ are compatible.

\begin{rems}
\begin{enumerate}
\item Note that the $\psi$-action of the monoid $T_\ast$ is in fact an action of the inverse monoid $T_\ast^{-1}$. However, we assume $T_+$ to be commutative so it may also be viewed as an action of $T_\ast$.
\item Pontryagin duality provides an equivalence of categories between compact $\Lambda(N_0)$-modules with a continuous $\psi$-action of $T_\ast$ and discrete $\Lambda(N_0)$-modules with a continuous $\varphi$-action of $T_\ast$. The surjectivity of the $\psi$-action corresponds to the injectivity of $\varphi$-action. Moreover, the $\psi$-action is nondegenerate if and only if so is the corresponding $\varphi$-action on the Pontryagin dual.
\end{enumerate}
\end{rems}

If $D$ is a $\Lambda(N_0)$-module with a $\varphi$-action of $T_\ast$ then there exists a homomorphism 
\begin{equation}
\Lambda(N_0)\otimes_{\Lambda(N_0),\varphi_t}D\to D, \lambda\otimes d\mapsto\lambda\varphi_t(d)\label{varphimap}
\end{equation}
of $\Lambda(N_0)$-modules. We say that the $T_\ast$-action on $D$ is \emph{\'etale} if the above map is an isomorphism. The $\varphi$-action of $T_\ast$ on $D$ is \'etale if and only if it is injective and for any $t\in T_\ast$ we have
\begin{equation}\label{etalevarphi_t}
D=\bigoplus_{u\in J(N_0/tN_0t^{-1})}u\varphi_t(D)\ .
\end{equation}
Similarly, we call a $\Lambda(N_0)$-module together with a $\varphi$-action of the monoid $\Zp\setminus\{0\}$ an \'etale $(\varphi,\Gamma)$-module over $\Lambda(N_0)$ if the action of $\varphi=\varphi_s$ is \'etale.

If $D$ is an \'etale $T_\ast$-module over $\Lambda(N_0)$ then there exists a $\psi$-action of $T_\ast$ compatible with the \'etale $\varphi$-action (see \cite{SVig} Section 6).

Dually, if $D$ is a $\Lambda(N_0)$-module with a $\psi$-action of $T_\ast$ then there exists a map
\begin{eqnarray*}
\iota_t\colon D&\to&\Lambda(N_0)\otimes_{\Lambda(N_0),\varphi_t}D\\
d&\mapsto&\sum_{u\in J(N_0/tN_0t^{-1})}u\otimes\psi_t(u^{-1}d)\ .
\end{eqnarray*}

\begin{lem}\label{phiuvorthogonal}
Fix $t\in T_\ast$. For any $\lambda\in\Lambda(N_0)$ and $u,v\in N_0$ we put $\lambda_{u,v}:=\psi_t(u^{-1}\lambda v)$. For any fixed $v\in N_0$ we have 
\begin{align*}
\lambda v=\sum_{u\in J(N_0/tN_0t^{-1})}u\varphi_t(\lambda_{u,v})
\end{align*}
and for any fixed $u\in N_0$ we have
\begin{align*}
u^{-1}\lambda=\sum_{v\in J(N_0/tN_0t^{-1})}\varphi_t(\lambda_{u,v})v^{-1}\ .
\end{align*}
\end{lem}
\begin{proof}
The above formulae follow from the usual identities $$\sum_{u\in J(N_0/tN_0t^{-1})}u\varphi_t(\psi_t(u^{-1}\mu))=\mu=\sum_{v\in J(N_0/tN_0t^{-1})}\varphi_t(\psi_t(\mu v))v^{-1}$$ for $\mu\in\Lambda(N_0)$ as the inverses of elements of $J(N_0/tN_0t^{-1})$ form a set of representatives of the right cosets of $tN_0t^{-1}$.
\qed\end{proof}

\begin{lem}\label{psimap}
For any $t\in T_\ast$ the map $\iota_t$ is a homomorphism of $\Lambda(N_0)$-modules. It is injective for all $t\in T_\ast$ if and only if the $\psi$-action of $T_\ast$ on $D$ is nondegenerate.
\end{lem}
\begin{proof}
Using Lemma \ref{phiuvorthogonal} we compute
\begin{align*}
\iota_t(\lambda x)=\sum_{u\in J(N_0/tN_0t^{-1})}u\otimes\psi_t(u^{-1}\lambda x)=\\
=\sum_{u,v\in J(N_0/tN_0t^{-1})}u\otimes\psi_t(\varphi_t(\lambda_{u,v})v^{-1}x)=\\
=\sum_{u,v\in J(N_0/tN_0t^{-1})}u\otimes\lambda_{u,v}\psi_t(v^{-1}x)=\\
=\sum_{u,v\in J(N_0/tN_0t^{-1})}u\varphi_t(\lambda_{u,v})\otimes\psi_t(v^{-1}x)=\\
=\sum_{v\in J(N_0/tN_0t^{-1})}\lambda v\otimes\psi_t(v^{-1}x)=\lambda\iota_t(x)\ .
\end{align*}
The second statement follows from noting that $\Lambda(N_0)$ is a free right module over itself via the map $\varphi_t$ with free generators $u\in J(N_0/tN_0t^{-1})$.
\qed\end{proof}

\begin{lem}\label{psipushforward}
Let $D$ be a $\Lambda(N_0)$-module with a $\psi$-action of $T_\ast$ and $t\in T_\ast$. Then there exists a $\psi$-action of $T_\ast$ on $\varphi_t^*D:=\Lambda(N_0)\otimes_{\Lambda(N_0),\varphi_t}D$ making the homomorphism $\iota_t$ $\psi$-equivariant. Moreover, if we assume in addition that the $\psi$-action on $D$ is nondegenerate then so is the $\psi$-action on $\varphi_t^*D$.
\end{lem}
\begin{proof}
Let $t'\in T_\ast$ be arbitrary and define the action of $\psi_{t'}$ on $\varphi_t^*D$ by putting
\begin{equation*}
\psi_{t'}(\lambda\otimes d):=\sum_{u'\in J(N_0/t'N_0t'^{-1})}\psi_{t'}(\lambda\varphi_t(u'))\otimes\psi_{t'}(u'^{-1}d)\text{ for }\lambda\in \Lambda(N_0),d\in D\ ,
\end{equation*}
and extending $\psi_{t'}$ to $\varphi_t^*D$ $o$-linearly. Note that we have 
\begin{align*}
\psi_{t'}(\varphi_{t'}(\mu)\lambda\otimes d)=\\
=\sum_{u'\in J(N_0/t'N_0t'^{-1})}\psi_{t'}(\varphi_{t'}(\mu)\lambda\varphi_t(u'))\otimes\psi_{t'}(u'^{-1}d)=\mu\psi_{t'}(\lambda\otimes d)\ .
\end{align*} 
Moreover, the map $\psi_{t'}$ is well-defined since we have 
\begin{align*}
\psi_{t'}(\lambda\varphi_t(\mu)\otimes d)=\sum_{v'\in J(N_0/t'N_0t'^{-1})}\psi_{t'}(\lambda\varphi_t(\mu)\varphi_t(v'))\otimes\psi_{t'}(v'^{-1}d)=\\
=\sum_{v'\in J(N_0/t'N_0t'^{-1})}\psi_{t'}(\lambda\varphi_t(\mu v'))\otimes\psi_{t'}(v'^{-1}d)=\\
=\sum_{u',v'\in J(N_0/t'N_0t'^{-1})}\psi_{t'}(\lambda\varphi_t(u'\varphi_{t'}(\mu_{u',v'})))\otimes\psi_{t'}(v'^{-1}d)=\\
=\sum_{u',v'\in J(N_0/t'N_0t'^{-1})}\psi_{t'}(\lambda\varphi_t(u'))\varphi_{t}(\mu_{u',v'})\otimes\psi_{t'}(v'^{-1}d)=\\
=\sum_{u',v'\in J(N_0/t'N_0t'^{-1})}\psi_{t'}(\lambda\varphi_t(u'))\otimes\mu_{u',v'}\psi_{t'}(v'^{-1}d)=\\
=\sum_{u',v'\in J(N_0/t'N_0t'^{-1})}\psi_{t'}(\lambda\varphi_t(u'))\otimes\psi_{t'}(\varphi_{t'}(\mu_{u',v'})v'^{-1}d)=\\
=\sum_{u'\in J(N_0/t'N_0t'^{-1})}\psi_{t'}(\lambda\varphi_t(u'))\otimes\psi_{t'}(u'^{-1}\mu d)=\psi_{t'}(\lambda\otimes\mu d)\ ,
\end{align*}
using Lemma \ref{phiuvorthogonal} where $\mu_{u',v'}=\psi_{t'}(u'^{-1}\mu v')$. Introducing the notation $J':=J(N_0/t'N_0t'^{-1})$ and $J'':=J(N_0/t''N_0t''^{-1})$ we further compute
\begin{align*}
\psi_{t''}(\psi_{t'}(\lambda\otimes d))=\psi_{t''}(\sum_{u'\in J(N_0/t'N_0t'^{-1})}\psi_{t'}(\lambda\varphi_t(u'))\otimes\psi_{t'}(u'^{-1}d))=\\
=\sum_{u''\in J''}\sum_{u'\in J'}\psi_{t''}(\psi_{t'}(\lambda\varphi_t(u'))\varphi_t(u''))\otimes\psi_{t''}(u''^{-1}\psi_{t'}(u'^{-1}d))=\\
=\sum_{u''\in J''}\sum_{u'\in J'}\psi_{t''}(\psi_{t'}(\lambda\varphi_t(u'\varphi_{t'}(u''))))\otimes\psi_{t''}(\psi_{t'}(\varphi_{t'}(u'')^{-1}u'^{-1}d))=\\
=\psi_{t''t'}(\lambda\otimes d)
\end{align*}
showing that it is indeed a $\psi$-action of the monoid $T_\ast$. 

For the second statement of the Lemma we compute
\begin{align*}
\psi_{t'}(\iota_t(x))=\\
=\sum_{u'\in J(N_0/t'N_0t'^{-1})}\sum_{u\in J(N_0/tN_0t^{-1})}\psi_{t'}(u\varphi_t(u'))\otimes\psi_{t'}(u'^{-1}\psi_t(u^{-1}x))=\\
=\sum_{u'\in J(N_0/t'N_0t'^{-1})}\sum_{u\in J(N_0/tN_0t^{-1})}\psi_{t'}(u\varphi_t(u'))\otimes\psi_{t'}(\psi_t(\varphi_t(u')^{-1}u^{-1}x))\ .
\end{align*}
Note that in the above sum $u\varphi_t(u')$ runs through a set of representatives for the cosets $N_0/tt'N_0t'^{-1}t^{-1}$. Moreover, $v:=\psi_{t'}(u\varphi_t(u'))$ is nonzero if and only if $u\varphi_t(u')$ lies in $t'N_0t'^{-1}$ and the nonzero values of $v$ run through a set $J'(N_0/tN_0t^{-1})$ of representatives of the cosets $N_0/tN_0t^{-1}$. In case $v\neq 0$ we have $\varphi_{t'}(v)^{-1}=(u\varphi_t(u'))^{-1}=\varphi_t(u')^{-1}u^{-1}$. So we continue computing by replacing $\psi_{t'}(u\varphi_t(u'))$ by $v$ and omitting the terms with $v=0$
\begin{align*}
\psi_{t'}(\iota_t(x))=\\
=\sum_{u'\in J(N_0/t'N_0t'^{-1})}\sum_{u\in J(N_0/tN_0t^{-1})}\psi_{t'}(u\varphi_t(u'))\otimes\psi_{t'}(\psi_t(\varphi_t(u')^{-1}u^{-1}x))=\\
=\sum_{v\in J'(N_0/tN_0t^{-1})}v\otimes\psi_t(\psi_{t'}(\varphi_{t'}(v^{-1})x))=\\
=\sum_{v\in J'(N_0/tN_0t^{-1})}v\otimes\psi_t(v^{-1}\psi_{t'}(x))=\iota_t(\psi_{t'}(x))\ .
\end{align*}

Assume now that the $\psi$-action of $T_\ast$ on $D$ is nondegenerate. Any element in $x\in\varphi_t^*D$ can be uniquely written in the form $\sum_{u\in J(N_0/tN_0t^{-1})}u\otimes x_u$. Assume that for a fixed $t'\in T_\ast$ we have $\psi_{t'}(u_0'^{-1}x)=0$ for all $u_0'\in N_0$. Then we compute
\begin{align*}
0=\psi_{t'}(u_0'^{-1}x)=\\
=\sum_{u'\in J(N_0/t'N_0t'^{-1})}\sum_{u\in J(N_0/tN_0t^{-1})}\psi_{t'}(u_0'^{-1}u\varphi_t(u'))\otimes\psi_{t'}(u'^{-1}x_u)\ .
\end{align*}
Put $y=u_0'^{-1}u\varphi_t(u')$. For any fixed $u_0'$ the set $\{y\mid u\in J(N_0/tN_0t^{-1}),u'\in J(N_0/t'N_0t'^{-1})\}$ forms a set of representatives of $N_0/tt'N_0(tt')^{-1}$, and we have $\psi_{t'}(y)\neq0$ if and only if $y$ lies in $t'N_0t'^{-1}$ in which case we have $\psi_{t'}(y)=t'^{-1}yt'$. So the nonzero values of $\psi_{t'}(y)$ run through a set of representatives of $N_0/tN_0t^{-1}$. Since we have the direct sum decomposition $\varphi_t^*D=\bigoplus_{v\in J(N_0/tN_0t^{-1})}v\otimes D$ we obtain $\psi_{t'}(u'^{-1}x_u)=0$ for all $u'\in J(N_0/t'N_0t'^{-1})$ and $u\in J(N_0/tN_0t^{-1})$ such that $y=u_0'^{-1}u\varphi_t(u')$ is in $t'N_0t'^{-1}$. However, for any choice of $u'$ and $u$ there exists such a $u_0'$, so we deduce $x=0$. 
\qed\end{proof}

\begin{pro}
\label{psiprop}
Let $D$ be a $\Lambda(N_0)$-module with a $\psi$-action of $T_\ast$. The following are equivalent:
\begin{enumerate}
\item There exists a unique $\varphi$-action on $D$, which is compatible with $\psi$ and which makes $D$ an \'etale $T_\ast$-module.
\item The $\psi$-action is surjective and for any $t\in T_\ast$ we have
\begin{equation}\label{psidirect}
D=\bigoplus_{u_0\in J(N_0/tN_0t^{-1})}\bigcap_{\substack{u\in J(N_0/tN_0t^{-1}) \\ u\neq u_0}}\mathrm{Ker}(\psi_t\circ u^{-1})\ .
\end{equation}
In particular, the action of $\psi$ is nondegenerate.
\item The map $\iota_t$ is bijective for all $t\in T_\ast$.
\end{enumerate}
\end{pro}
\begin{proof}
$1\Longrightarrow 3$ In this case the map $\iota_t$ is the inverse of the isomorphism \eqref{varphimap} so it is bijective by the \'etale property.

$3\Longrightarrow 2$: The injectivity of $\iota_t$ shows the nondegeneracy of the $\psi$-action. Further if $1\otimes d=\iota_t(x)$ then we have $\psi_t(x)=d$ so the $\psi$-action is surjective. Moreover, $\iota_t^{-1}(u_0\otimes D)$ equals $\bigcap_{u_0\neq u\in J(N_0/tN_0t^{-1}) }\Ker(\psi_t\circ u^{-1})$ therefore $D$ can be written as a direct sum \eqref{psidirect}.

$2\Longrightarrow 1$:
In order to define the $\varphi$-action of $T_\ast$ on $D$ we fix $t\in T_\ast$. For any $d\in D$ we have to choose $\varphi_t(d)$ such that $\psi_t(\varphi_t(d))=d$. By the surjectivity of $\psi_t$ we can choose $x\in D$ such that $\psi_t(x)=d$. Using the assumption we can write $x=\sum_{u_0\in J(N_0/tN_0t^{-1})}x_{u_0}$, with
\begin{equation*}
x_{u_0}\in\bigcap_{\substack{u\in J(N_0/tN_0t^{-1}) \\ u\neq u_0}}\mathrm{Ker}(\psi_t\circ u^{-1})\ .
\end{equation*}
By the compatibility $(\varphi\psi)$ we should have
\begin{equation*}
\varphi_t(d)\in\bigcap_{\substack{u\in J(N_0/tN_0t^{-1}) \\ u\neq 1}}\mathrm{Ker}(\psi_t\circ u^{-1})
\end{equation*}
as we have $\psi_t(u)=0$ for all $u\in N_0\setminus tN_0t^{-1}$.

A convenient choice is $\varphi_t(d)=x_1$, and there exists exactly one such element in $D$: if $x'$ would be an other, then
\begin{equation*}
x_1-x'\in\bigcap_{u\in J(N_0/tN_0t^{-1})}\mathrm{Ker}(\psi_t\circ u^{-1})=\{0\}\ .
\end{equation*}
This shows the uniqueness of the $\varphi$-action. Further, $x_1=\varphi_t(d)=0$ would mean that $x$ lies in $\Ker(\psi_t)$ whence $d=\psi_t(x)=0$---therefore the injectivity. Similarly, by definition we also have $x_{u_0}=u_0\varphi_t\circ\psi_t(u_0^{-1}x)$ for all $u_0\in J(N_0/sN_0s^{-1})$. By the surjectivity of the $\psi$-action any element in $D$ can be written of the form $\psi_t(u_0^{-1}x)$ for any fixed $u_0\in J(N_0/tN_0t^{-1})$ so we obtain 
\begin{equation*}
u_0\varphi_t(D)=\bigcap_{u_0\neq u\in J(N_0/tN_0t^{-1})}\Ker(\psi_t\circ u^{-1})\ .
\end{equation*}
The \'etale property \eqref{etalevarphi_t} follows from this using our assumption 2. Moreover, this also shows $\psi_t(u\varphi_t(d))=0$ for all $u\in N_0\setminus tN_0t^{-1}$ which implies $(\varphi\psi)$ using that $\psi_t\circ\varphi_t=\id_D$ by construction. Finally, $\varphi_t(\lambda)\varphi_t(d)-\varphi_t(\lambda d)$ lies in the kernel of $\psi_t\circ u_0^{-1}$ for any $u_0\in J(N_0/tN_0t^{-1})$, $\lambda\in\Lambda(N_0)$ and $d\in D$, so it is zero.
\qed\end{proof}

From now on if we have an \'etale $T_\ast$-module over $\Lambda(N_0)$ we a priori equip it with the compatible $\psi$-action, and if we have a $\Lambda(N_0)$-module with a $\psi$-action, which satisfies the above property $2$, we equip it with the compatible $\varphi$-action, which makes it \'etale. The construction of the \'etale hull and its universal property is given in the following

\begin{pro}\label{etalehull}
For any $\Lambda(N_0)$-module $D$, with a $\psi$-action of $T_\ast$ there exists an \'etale $T_\ast$-module $\widetilde{D}$ over $\Lambda(N_0)$ and a $\psi$-equivariant $\Lambda(N_0)$-homomor-phism $\iota\colon D\to \widetilde{D}$ with the following universal property: For any $\psi$-equivariant $\Lambda(N_0)$-homomorphism $f:D\to D'$ into an \'etale $T_\ast$-module $D'$ we have a unique morphism $\widetilde{f}:\widetilde{D}\to D'$ of \'etale $T_\ast$-modules over $\Lambda(N_0)$ making the diagram
\begin{equation*}
\xymatrix{
D\ar[r]^{\iota}\ar[d]_f&\widetilde{D}\ar[dl]^{\widetilde{f}}\\
D'&
}
\end{equation*}
commutative. $\widetilde{D}$ is unique upto a unique isomorphism. If we assume the $\psi$-action on $D$ to be nondegenerate then $\iota$ is injective.
\end{pro}
\begin{proof}
We will construct $\widetilde{D}$ as the injective limit of $\varphi_t^*D$ for $t\in T_\ast$. Consider the following partial order on the set $T_\ast$: we put $t_1\leq t_2$ whenever we have $t_2t_1^{-1}\in T_\ast$. Note that by Lemma \ref{psipushforward} we obtain a $\psi$-equivariant isomorphism $\varphi_{t_2t_1^{-1}}^*\varphi_{t_1}^*D\cong \varphi_{t_2}^*D$ for any pair $t_1\leq t_2$ in $T_\ast$. In particular, we obtain a $\psi$-equivariant map $\iota_{t_1,t_2}\colon \varphi_{t_1}^*D\to \varphi_{t_2}^*D$. Applying this observation to $\varphi_{t_1}^*D$ for a sequence $t_1\leq t_2\leq t_3$ we see that the $\Lambda(N_0)$-modules $\varphi_t^*D$ ($t\in T_\ast$) with the $\psi$-action of $T_\ast$ form a direct system with respect to the connecting maps $\iota_{t_1,t_2}$. We put 
\begin{equation*}
\widetilde{D}:=\varinjlim_{t\in T_\ast}\varphi_t^*D
\end{equation*}
as a $\Lambda(N_0)$-module with a $\psi$-action of $T_\ast$. For any fixed $t'\in T_\ast$ we have
\begin{align*}
\varphi_{t'}^*\widetilde{D}=\Lambda(N_0)\otimes_{\Lambda(N_0),\varphi_{t'}}\varinjlim_{t\in T_\ast}\varphi_t^*D\cong\\
\cong \varinjlim_{t\in T_\ast}\Lambda(N_0)\otimes_{\Lambda(N_0),\varphi_{t'}}\varphi_{t}^*D\cong \varinjlim_{t't\in T_\ast}\varphi_{t't}^*D\cong\widetilde{D}
\end{align*}
showing that there exists a unique $\varphi$-action of $T_\ast$ on $\widetilde{D}$ making $\widetilde{D}$ an \'etale $T_\ast$-module over $\Lambda(N_0)$ by Proposition \ref{psiprop}. 

For the universal property, let $f:D\to D'$ be an $\psi$-equivariant map into an \'etale $T_\ast$-module $D'$ over $\Lambda(N_0)$. By construction of the map $\varphi_t$ on $\widetilde{D}$ ($t\in T_\ast$) we have $\varphi_t(\iota(x))=(1\otimes x)_t$ where $(1\otimes x)_t$ denotes the image of $1\otimes x\in \varphi_t^*D$ in $\widetilde{D}$. So we put
\begin{equation*}
\widetilde{f}((\lambda\otimes x)_t):=\lambda\varphi_t(f(x))\in D'
\end{equation*}
and extend it $o$-linearly to $\widetilde{D}$. Note right away that $\widetilde{f}$ is unique as it is $\varphi_t$-equivariant. The map $\widetilde{f}\colon \widetilde{D}\to D'$ is well-defined as we have
\begin{align*}
\widetilde{f}(\iota_{t,tt'}(1\otimes_t x))=\widetilde{f}(\sum_{u'\in N_0/t'N_0t'^{-1}}u'\otimes_{t'}\psi_{t'}(u'^{-1}\otimes_t x))=\\
=\sum_{u',v'\in N_0/t'N_0t'^{-1}}\widetilde{f}(u'\otimes_{t'}\psi_{t'}(u'^{-1}\varphi_t(v'))\otimes_t\psi_{t'}(v'^{-1}x))=\\
=\sum_{u',v'\in N_0/t'N_0t'^{-1}}\widetilde{f}(u'\varphi_{t'}\circ\psi_{t'}(u'^{-1}\varphi_t(v'))\otimes_{tt'}\psi_{t'}(v'^{-1}x))=\\
=\sum_{v'\in N_0/t'N_0t'^{-1}}\widetilde{f}(\varphi_t(v')\otimes_{tt'}\psi_{t'}(v'^{-1}x))=\\
=\sum_{v'\in N_0/t'N_0t'^{-1}}\varphi_t(v')\varphi_{tt'}(f(\psi_{t'}(v'^{-1}x)))=\\
=\sum_{v'\in N_0/t'N_0t'^{-1}}\varphi_t(v'\varphi_{t'}\circ\psi_{t'}(v'^{-1}f(x)))=\varphi_t(f(x))=\widetilde{f}(1\otimes_t x)
\end{align*}
noting that $\iota_{t,tt'}$ is a $\Lambda(N_0)$-homomorphism. Here the notation $\otimes_t$ indicates that the tensor product is via the map $\varphi_t$. By construction $\widetilde{f}$ is a homomorphism of \'etale $T_\ast$-modules over $\Lambda(N_0)$ satisfying $\widetilde{f}\circ\iota=f$.

The injectivity of $\iota$ in case the $\psi$-action on $D$ is nondegenerate follows from Lemmata \ref{psimap} and \ref{psipushforward}.
\qed\end{proof}

\begin{ex}
If $D$ itself is \'etale then we have $\widetilde{D}=D$.
\end{ex}

\begin{cor}
The functor $D\mapsto\widetilde{D}$ from the category of $\Lambda(N_0)$-modules with a $\psi$-action of $T_\ast$ to the category of \'etale $T_\ast$-modules over $\Lambda(N_0)$ is exact.
\end{cor}
\begin{proof}
$\Lambda(N_0)$ is a free $\varphi_t(\Lambda(N_0))$-module, so $\Lambda(N_0)\otimes_{\Lambda(N_0),\varphi_t}-$ is exact, and so is the direct limit functor.
\qed\end{proof}

\begin{cor}\label{tildeinjective}
Assume that $D$ is a $\Lambda(N_0)$-module with a nondegenerate $\psi$-action of $T_\ast$ and $f\colon D\to D'$ is an injective $\psi$-equivariant $\Lambda(N_0)$-homomor-phism into the \'etale $T_\ast$-module $D'$ over $\Lambda(N_0)$. Then $\widetilde{f}$ is also injective. 
\end{cor}
\begin{proof}
Since $D$ is nondegenerate we may identify $\varphi_t^*D$ with a $\Lambda(N_0)$-submodule of $\widetilde{D}$. Assume that $x=\sum_{u\in J(N_0/tN_0t^{-1})}u\otimes_t x_u\in \varphi_t^*D$ lies in the kernel of $\widetilde{f}$. Then $x_u=\psi_t(u^{-1}x)\in D\subseteq \varphi_t^* D\subseteq \widetilde{D}$ ($u\in J(N_0/tN_0t^{-1})$) also lies in the kernel of $\widetilde{f}$. However, we have $\widetilde{f}(x_u)=f(x_u)$ showing that $x_u=0$ for all $u\in J(N_0/tN_0t^{-1})$ as $f$ is injective. 
\qed\end{proof}

\begin{ex}
Let $D$ be a (classical) irreducible \'etale $(\varphi,\Gamma)$-module over $k\bg X\jg$ and $D_0\subset D$ a $\psi$- and $\Gamma$-invariant treillis in $D$. Then we have $\widetilde{D_0}\cong D$ unless $D$ is $1$-dimensional and $D_0=D^\natural$ in which case we have $\widetilde{D_0}=D_0$.
\end{ex}
\begin{proof}
If $D$ is $1$-dimensional then $D^\natural=D^+$ is an \'etale $(\varphi,\Gamma)$-module over $k\bs X\js$ (Prop.\ II.5.14 in \cite{Mira}) therefore it is equal to its \'etale hull. If $\dim D>1$ then we have $D^\natural=D^\#\subseteq D_0$ by Cor.\ II.5.12 and II.5.21 in \cite{Mira}. By Corollary \ref{tildeinjective} $\widetilde{D^\#}\subseteq \widetilde{D_0}$ injects into $D$ and it is $\varphi$- and $\psi$-invariant. Since $D^\#$ is not $\varphi$-invariant (Prop.\ II.5.14 in \cite{Mira}) and it is the maximal compact $o\bs X\js$-submodule of $D$ on which $\psi$ acts surjectively (Prop.\ II.4.2 in \cite{Mira}) we obtain that $\widetilde{D_0}$ is not compact. In particular, its $X$-divisible part is nonzero therefore equals $D$ as the $X$-divisible part of $\widetilde{D_0}$ is an \'etale $(\varphi,\Gamma)$-submodule of the irreducible $D$.
\qed\end{proof}

\begin{pro}
The $T_+^{-1}$ action on $D_{SV}(\pi)$ is a surjective nondegenerate $\psi$-action of $T_+$.
\end{pro}
\begin{proof}
Let $d\in D_{SV}(\pi)$ and $t\in T_+$. Since the action of both $t$ and $\Lambda(N_0)$ on $D_{SV}(\pi)$ comes from that on $\pi^\vee$ we have $t^{-1}\varphi_t(\lambda)d=t^{-1}t\lambda t^{-1}d=\lambda t^{-1}d$, so this is indeed a $\psi$-action. The surjectivity of each $\psi_t$ follows from the injectivity of the multiplication by $t$ on each $W\in\mathcal{B}_+(\pi)$ and the exactness of $\varinjlim$ and $(\cdot)^\vee$. Finally, if $W$ is in $\mathcal{B}_+(\pi)$ then so is $t^*W:=\sum_{u\in J(N_0/tN_0t^{-1})}utW$ for any $t\in T_+$. Take an element $d\in D_{SV}(\pi)$ lying in the kernel of $\psi_t(u^ {-1}\cdot)$ for all $u\in J(N_0/tN_0t^ {-1})$. Now $D_{SV}(\pi)$ is by definition the direct limit of $W^\vee$ for all $W\in \mathcal{B}_+(\pi)$, so $\psi_t(u^{-1}d)=0$ means that $t^{-1}u^{-1}d$ vanishes on some $W\in \mathcal{B}_+(\pi)$ (depending a priori on $u$). Since the set $J(N_0/tN_0t^{-1})$ is finite, we may even choose a common $W$ for all $u$ (taking the intersection and using Lemma 2.2 in \cite{SVig}). Then the restriction of $d$ to $t^*W$ is zero showing that $d$ is zero in $D_{SV}(\pi)$ therefore the nondegeneracy.
Alternatively, the nondegeneracy of the $\psi$-action also follows from the existence of a $\psi$-equivariant injective map $D_{SV}(\pi)\hookrightarrow D^0_{SV}(\pi)$ into an \'etale $T_+$-module $D^0_{SV}(\pi)$ (\cite{SVig} Proposition 3.5 and Remark 6.1).
\qed\end{proof}

\begin{que}
Let $D_{SV}^{(0)}(\pi)$ as in \cite{SVig}. We have that $D_{SV}^{(0)}(\pi)$ is an \'etale $T_\ast$-module over $\Lambda(N_0)$ (\cite{SVig} Proposition 3.5) and $f:D_{SV}(\pi)\hookrightarrow D_{SV}^{(0)}(\pi)$ is a $\psi$-equivariant map (\cite{SVig} Remark 6.1). By the universal property of the \'etale hull and Corollary \ref{tildeinjective} $\widetilde{D_{SV}}(\pi)$ also injects into $D_{SV}^{(0)}(\pi)$. Whether or not this injection is always an isomorphism is an open question. In case of the Steinberg representation this is true by Proposition 11 in \cite{Z}.
\end{que}

We call the submonoid $T'_\ast\leq T_\ast\leq T_+$ cofinal in $T_\ast$ if for any $t\in T_\ast$ there exists a $t'\in T'_\ast$ such that $t\leq t'$. For example $\xi(\Zp\setminus\{0\})$ is cofinal in $T_+$.

\begin{cor}\label{univphigamma}
Let $D$ be a $\Lambda(N_0)$-module with a $\psi$-action of $T_\ast$ and denote by $\widetilde{D}$ (resp.\ by $\widetilde{D}'$) the \'etale hull of $D$ for the $\psi$-action of $T_\ast$ (resp.\ of $T'_\ast$). Then we have a natural isomorphism $\widetilde{D}'\overset{\sim}{\rightarrow}\widetilde{D}$ of \'etale $T'_\ast$-modules over $\Lambda(N_0)$. More precisely, if $f\colon D\to D_1$ is a $\psi$-equivariant $\Lambda(N_0)$-homomorphism into an \'etale $T'_\ast$-module $D_1$ then $f$ factors uniquely through $\iota\colon D\to \widetilde{D}$.
\end{cor}
\begin{proof}
Since $T'_\ast\leq T_\ast$ is cofinal in $T_\ast$ we have $\varinjlim_{t'\in T'_\ast}\varphi_{t'}^*D\cong \varinjlim_{t\in T_\ast}\varphi_t^*D= \widetilde{D}$.
\qed\end{proof}

By Corollary \ref{univphigamma} there exists a homomorphism $\widetilde{\pr}:\widetilde{D_{SV}}(\pi)\to D^{\vee}_{\xi,\ell,\infty}(\pi)$ of \'etale $(\varphi,\Gamma)$-modules over $\Lambda(N_0)$ such that $\pr=\widetilde{\pr}\circ\iota$. Our main result in this section is the following

\begin{thm}\label{pscompdsv}
$D^\vee_{\xi,\ell,\infty}(\pi)$ is the pseudocompact completion of $\Lambda_\ell(N_0)\otimes_{\Lambda(N_0)}\widetilde{D_{SV}}(\pi)$ in the category of \'etale $(\varphi,\Gamma)$-modules over $\Lambda_\ell(N_0)$, ie.\ we have
\begin{equation*}
D^\vee_{\xi,\ell,\infty}(\pi)\cong \varprojlim_{D}D
\end{equation*}
where $D$ runs through the finitely generated \'etale $(\varphi,\Gamma)$-modules over $\Lambda_\ell(N_0)$ arising as a quotient of $\Lambda_\ell(N_0)\otimes_{\Lambda(N_0)}\widetilde{D_{SV}}(\pi)$ by a closed submodule. This holds in \emph{any} topology on $\Lambda_\ell(N_0)\otimes_{\Lambda(N_0)}\widetilde{D_{SV}}(\pi)$ making both the maps $1\otimes\iota\colon D_{SV}(\pi)\to \Lambda_\ell(N_0)\otimes_{\Lambda(N_0)}\widetilde{D_{SV}}(\pi)$, $d\mapsto 1\otimes\iota(d)$ and $1\otimes\widetilde{\pr}\colon \Lambda_\ell(N_0)\otimes_{\Lambda(N_0)}\widetilde{D_{SV}}(\pi)\to D^\vee_{\xi,\ell,\infty}(\pi)$ continuous.
\end{thm}
\begin{rem}
Since the map $\pr\colon D_{SV}(\pi)\to D^\vee_{\xi,\ell,\infty}(\pi)$ is continuous, there exists such a topology on $ \Lambda_\ell(N_0)\otimes_{\Lambda(N_0)}\widetilde{D_{SV}}(\pi)$. For instance we could take either the final topology of the map $D_{SV}(\pi)\to \Lambda_\ell(N_0)\otimes_{\Lambda(N_0)}\widetilde{D_{SV}}(\pi)$ or the initial topology of the map $\Lambda_\ell(N_0)\otimes_{\Lambda(N_0)}\widetilde{D_{SV}}(\pi)\to D^\vee_{\xi,\ell,\infty}(\pi)$.
\end{rem}
\begin{proof}
The homomorphism $\widetilde{\pr}$ factors through the map $1\otimes\id \colon \widetilde{D_{SV}}(\pi)\to \Lambda_\ell(N_0)\otimes_{\Lambda(N_0)}\widetilde{D_{SV}}(\pi)$ since $D^\vee_{\xi,\ell,\infty}(\pi)$ is a module over $\Lambda_\ell(N_0)$, so we obtain a homomorphism 
\begin{equation*}
1\otimes\widetilde{\pr}\colon \Lambda_\ell(N_0)\otimes_{\Lambda(N_0)}\widetilde{D_{SV}}(\pi)\to D^\vee_{\xi,\ell,\infty}(\pi)
\end{equation*}
of \'etale $(\varphi,\Gamma)$-modules over $\Lambda_\ell(N_0)$. At first we claim that $1\otimes\widetilde{\pr}$ has dense image. Let $M\in\mathcal{M}(\pi^{H_0})$ and $W\in\mathcal{B}_+(\pi)$ be arbitrary. Then by Lemma \ref{minw} the map $\pr_{W,M,k}\colon W^\vee\to M_k^\vee$ is surjective for $k\geq 0$ large enough. This shows that the natural map 
\begin{equation*}
1\otimes\pr_{W,M,k}\colon\Lambda_\ell(N_0)\otimes_{\Lambda(N_0)}W^\vee\to \Lambda_\ell(N_0)\otimes_{\Lambda(N_0)}M_k^\vee\cong M_k^\vee[1/X]
\end{equation*}
is surjective. However, $1\otimes\pr_{W,M,k}$ factors through $\Lambda_\ell(N_0)\otimes_{\Lambda(N_0)}D_{SV}(\pi)$ by the Remarks after Lemma \ref{prpsigamma}. In particular, the natural map
\begin{equation*}
1\otimes\pr_{M,k}\colon\Lambda_\ell(N_0)\otimes_{\Lambda(N_0)}D_{SV}(\pi)\to M_k^\vee[1/X]
\end{equation*}
is surjective for all $M\in\mathcal{M}(\pi^{H_0})$ and $k\geq 0$ large enough (whence in fact for all $k\geq 0$). This shows that the image of the map
\begin{equation*}
1\otimes\pr\colon \Lambda_\ell(N_0)\otimes_{\Lambda(N_0)}D_{SV}(\pi)\to D^\vee_{\xi,\ell,\infty}(\pi)
\end{equation*}
is dense whence so is the image of $1\otimes\widetilde{\pr}$. By the assumption that $1\otimes\widetilde{\pr}$ is continuous we obtain a surjective homomorphism
\begin{equation*}
\widehat{1\otimes\widetilde{\pr}}\colon\varprojlim_{D}D\to D^\vee_{\xi,\ell,\infty}(\pi)
\end{equation*}
of pseudocompact $(\varphi,\Gamma)$-modules over $\Lambda_\ell(N_0)$ where $D$ runs through the finitely generated \'etale $(\varphi,\Gamma)$-modules over $\Lambda_\ell(N_0)$ arising as a quotient of $\Lambda_\ell(N_0)\otimes_{\Lambda(N_0)}\widetilde{D_{SV}}(\pi)$. 

Let $0\neq (x_D)_D$ be in the kernel of $\widehat{1\otimes\widetilde{\pr}}$. Then there exists a finitely generated \'etale $(\varphi,\Gamma)$-module $D$ over $\Lambda_\ell(N_0)$ with a surjective continuous homomorphism $\Lambda_\ell(N_0)\otimes_{\Lambda(N_0)}\widetilde{D_{SV}}(\pi)\twoheadrightarrow D$ such that $x_D\neq 0$. By Proposition \ref{1otimestildeprinj} this map factors through $D^\vee_{\xi,\ell,\infty}(\pi)$ contradicting to the assumption $\widehat{1\otimes\widetilde{\pr}}((x_D)_D)=0$.
\qed\end{proof}

\begin{rem}
Breuil's functor $D^\vee_\xi$ can therefore be computed from $D_{SV}$ the following way: For a smooth $o/\varpi^h$-representation $\pi$ we have $D^\vee_\xi(\pi)\cong (\varprojlim_{D}D)_{H_0}\cong \varprojlim_D D_{H_0}$ where $D$ runs through the finitely generated \'etale $(\varphi,\Gamma)$-modules over $\Lambda_\ell(N_0)$ arising as a quotient of $\Lambda_\ell(N_0)\otimes_{\Lambda(N_0)}\widetilde{D_{SV}}(\pi)$ by a closed submodule.
\end{rem}

\section{Nongeneric $\ell$}  \label{nongeneric}

Assume from now on that $\ell=\ell_\alpha$ is a nongeneric Whittaker functional defined by the projection of $N_0$ onto $N_{\alpha,0}\cong\Zp$ for some simple root $\alpha\in\Delta$. 

\begin{rem}
In \cite{B} the Whittaker functional $\ell$ is assumed to be generic. However, even if $\ell$ is not generic, the functor $D^\vee_\xi$ (hence also $D^\vee_{\xi,\ell,\infty}$) is right exact even though the restriction of $D^\vee_\xi$ to the category $SP_{o/\varpi^h}$ may not be exact in general. 
\end{rem}

\subsection{Compatibility with parabolic induction}

Let $P=L_PN_P$ be a parabolic subgroup of $G$ containing $B$ with Levi component $L_P$ and unipotent radical $N_P$ and let $\pi_P$ be a smooth $o/\varpi^h$-representation of $L_P$ that we view as a representation of $P^-$ via the quotient map $P^-\twoheadrightarrow L_P$ where $P^-=L_PN_{P^-}$ is the parabolic subgroup opposite to $P$. Since $T$ is contained in $L_P$, we may consider the same cocharacter $\xi\colon \mathbb{Q}_p^\times\to T$ for the group $L_P$ instead of $G$. Further, we put $N_{L_P}:=N\cap L_P$ and $N_{L_P,0}:=N_0\cap L_P$.

As in \cite{B} denote by $W:=N_G(T)/T$ (resp.\ by $W_P:=(N_G(T)\cap L_P)/T$) the Weyl group of $G$ (resp.\ of $L_P$) and by $w_0\in W$ the element of maximal length. We have a canonical system $$K_P:=\{w\in W\mid w^{-1}(\Phi_P^+)\subseteq\Phi^+\}$$
of representatives (the Kostant representatives) of the right cosets $W_P\backslash W$ where $\Phi_P^+$ denotes the set of positivie roots of $L_P$ with respect to the Borel subgroup $L_P\cap B$. We have a generalized Bruhat decomposition
$$G=\coprod_{w\in K_P}P^-wB=\coprod_{w\in K_P}P^-wN\ .$$

Now let $\pi_P$ be a smooth representation of $L_P$ over $A$. We regard $\pi_P$ as a representation of $P^-$ via the quotient map $P^-\twoheadrightarrow L_P$. Then the parabolically induced representation $\Ind_{P^-}^G\pi_P$ admits \cite{V} (see also \cite{E} \S 4.3) a filtration by $B$-subrepresentations whose graded pieces are contained in
$$\mathcal{C}_w(\pi_P):=c-\Ind_{P^-}^{P^-wN}\pi_P$$
for $w\in K_P$ where $c-\Ind_{P^-}^*$ stands for the space of locally constant functions on $*\supseteq P^-$ with compact support modulo $P^-$. $B$ acts on $\mathcal{C}_w(\pi_P)$ by right translations. Moreover, the first graded piece equals $\mathcal{C}_1(\pi_P)$.

\begin{lem}\label{DxiCw}
Let $\pi'\leq \mathcal{C}_w(\pi_P)$ be any $B$-subrepresentation for some $w\in K_P\setminus\{1\}$. Then we have $D^\vee_\xi(\pi')=0$.
\end{lem}
\begin{proof}
By the right exactness of $D^\vee_\xi$ (Prop.\ 2.7$(ii)$ in \cite{B}) it suffices to treat the case $\pi'=\mathcal{C}_w(\pi_P)$. For this the same argument works as in Prop.\ 6.2 \cite{B} with the following modification:

The particular shape of $\ell$ is only used in Lemma 6.5 in \cite{B} (note that the subgroup $H_0=\Ker(\ell\colon N_0\to\Zp)$ is denoted by $N_1$ therein). For an element $w\neq 1$ in the Weyl group we have $(w^{-1}N_{P^-}w\cap N_0)\backslash N_0/H_0=\{1\}$ if and only if $H_0$ does not contain $w^{-1}N_{P^-}w\cap N_0$. Whenever $w^{-1}N_{P^-}w\cap N_0\not\subseteq H_0$, the statement of Lemma 6.5 in \cite{B} is true and there is nothing to prove. 

In case we have $\{1\}\neq  w^{-1}N_{P^-}w\cap N_0\subseteq H_0$, the statement of Lemma 6.5 is not true for $\ell=\ell_\alpha$. However, the argument using it in the proof of Prop.\ 6.2 can be replaced by the following: the operator $F$ acts on the space $\mathcal{C}((w^{-1}N_{P^-}w\cap N_0)\backslash N_0,\pi_P^w)^{H_0}$ nilpotently. Indeed, the trace map $\Tr_{H_0/sH_0s^{-1}}$
\begin{eqnarray*}
\mathcal{C}((w^{-1}N_{P^-}w\cap N_0)\backslash N_0,\pi_P^w)^{sH_0s^{-1}}\to \mathcal{C}((w^{-1}N_{P^-}w\cap N_0)\backslash N_0,\pi_P^w)^{H_0}
\end{eqnarray*}
is zero as each double coset $(w^{-1}N_{P^-}w\cap H_0)\backslash H_0/sH_0s^{-1}$ has size divisible by $p$ and any function in $\mathcal{C}((w^{-1}N_{P^-}w\cap N_0)\backslash N_0,\pi_P^w)^{sH_0s^{-1}}$ is constant on these double cosets. The statement follows from Prop.\ 2.7$(iii)$ in \cite{B}.
\qed\end{proof}

In order to extend Thm.\ 6.1 in \cite{B} (the compatibility with parabolic induction) to our situation ($\ell=\ell_\alpha$) we need to distinguish two cases: whether the root subgroup $N_\alpha$ is contained in $L_P$ or in $N_P$. Similarly to \cite{E} we define the $s^{\mathbb{Z}}N_{L_P}$-ordinary part $\mathrm{Ord}_{s^\mathbb{Z}N_{L_P}}(\pi_P)$ of a smooth representation $\pi_P$ of $L_P$ as follows. We equip $\pi_P^{N_{L_P,0}}$ with the Hecke action $F_P:=\Tr_{N_{L_P,0}/sN_{L_P,0}s^{-1}}\circ (s\cdot)$ of $s$ making $\pi_P^{N_{L_P,0}}$ a module over the polynomial ring $o/\varpi^h[F_P]$ and put $$\mathrm{Ord}_{s^\mathbb{Z}N_{L_P}}(\pi_P):=\Hom_{o/\varpi^h[F_p]}(o/\varpi^h[F_P,F_P^{-1}],\pi_P^{N_{L_P,0}})_{F_P-fin}$$ where $F_P-fin$ stands for those elements in the $\Hom$-space whose orbit under the action of $F_P$ is finite. By Lemmata 3.1.5 and 3.1.6 in \cite{E} we may identify $\mathrm{Ord}_{s^\mathbb{Z}N_{L_P}}(\pi_P)$ with an $o/\varpi^h[F_P]$-submodule in $\pi_P^{N_{L_P,0}}$ by sending a map $f\in \mathrm{Ord}_{s^\mathbb{Z}N_{L_P}}(\pi_P)$ to its value $f(1)\in \pi_P^{N_{L_P,0}}$ at $1\in o/\varpi^h[F_P,F_P^{-1}]$.

\begin{pro}\label{parindnongen}
Let $\pi_P$ be a smooth locally admissible representation of $L_P$ over $A$ which we view by inflation as a representation of $P^-$. We have an isomorphism 
\begin{equation*}
D^\vee_{\xi}\left(\Ind_{P^-}^G\pi_P\right)\cong \begin{cases}D^\vee_{\xi}(\pi_P)&\text{if }N_\alpha\subseteq L_P\\
o/\varpi^h\bg X\jg\widehat{\otimes}_{o/\varpi^h}\mathrm{Ord}_{s^{\mathbb{Z}}N_{L_P}}(\pi_P)^\vee&\text{if }N_\alpha\subseteq N_P
\end{cases}
\end{equation*}
as \'etale $(\varphi,\Gamma)$-modules. In particular, for $P=B$ we have $D^\vee_\xi(\Ind_{B^-}\pi_B)\cong o/\varpi^h\bg X\jg\widehat{\otimes}_{o/\varpi^h}\pi_B^\vee$, ie.\ the value of $D^\vee_\xi$ at the principal series is the same $(\varphi,\Gamma)$-module of rank $1$ regardless of the choice of $\ell$ (generic or not).
\end{pro}
\begin{proof}
By Lemma \ref{DxiCw} and the right exactness of $D_\xi^\vee$ (Prop.\ 2.7$(ii)$ in \cite{B}) it suffices to show that $D^\vee_\xi(\mathcal{C}_1(\pi_P))\cong D^\vee_{\xi}(\pi_P)$. Moreover, the proof of Prop.\ 6.7 in \cite{B} goes through without modification so we have an isomorphism $D^\vee_\xi(\mathcal{C}_1(\pi_P))\cong D^\vee((\Ind_{P^-\cap N_0}^{N_0}\pi_P)^{H_0})$. Hence we are reduced to computing $D^\vee((\Ind_{P^-\cap N_0}^{N_0}\pi_P)^{H_0})$ in terms of $\pi_P$. We further have an identification $$\Ind_{P^-\cap N_0}^{N_0}\pi_P\cong\mathcal{C}(N_{P,0},\pi_P)\cong \mathcal{C}(N_{P,0},o/\varpi^h)\otimes_{o/\varpi^h}\pi_P$$ by equation $(40)$ in \cite{B}. We need to distinguish two cases.

\emph{Case 1: $N_\alpha\subseteq L_P$.} In this case we have $N_{P,0}\subseteq H_0$. Hence we deduce $(\mathcal{C}(N_{P,0},o/\varpi^h)\otimes_{o/\varpi^h}\pi_P)^{H_0}=\pi_P^{H_0/N_{P,0}}=\pi_P^{H_{P,0}}$. So we have $$D^\vee_{\xi}\left(\Ind_{P^-}^G\pi_P\right)\cong D^\vee(\Ind_{P^-\cap N_0}^{N_0}\pi_P)^{H_0})\cong D^\vee(\pi_P^{H_{P,0}})\cong D^\vee_{\xi}(\pi_P)$$ in this case as claimed.

\emph{Case 2: $N_\alpha\subseteq N_P$.} In this case we have $N_{L_P,0}\subseteq H_0$ and $N_{P,0}/(N_{P,0}\cap H_0)\cong \Zp$. So we have an identification $$\mathcal{C}(N_{P,0},\pi_P)^{H_0}\cong \mathcal{C}(N_{P,0}/(N_{P,0}\cap H_0),\pi_P^{N_{L_P,0}})\cong \mathcal{C}(\Zp,\pi_P^{N_{L_P,0}})\ .$$
Here the Hecke action $F=F_G=\Tr_{H_0/sH_0s^{-1}}\circ (s\cdot)$ of $s$ on the right hand side is given by the formula $$F_G(f)(a)=\begin{cases}F_P(f(a/p))&\text{if }a\in p\Zp\\ 0&\text{if }a\in\Zp\setminus p\Zp\end{cases}\ ,$$ where $F_P=\Tr_{N_{L_P,0}/sN_{L_P,0}s^{-1}}\circ (s\cdot)$ denotes the Hecke action of $s$ on $\pi_P^{N_{L_P,0}}$.

Now let $M$ be a finitely generated $o/\varpi^h\bs X\js[F]$ submodule of $\mathcal{C}(\Zp,\pi_P^{N_{L_P,0}})$ that is stable under the action of $\Gamma$ and is admissible as a representation of $\Zp$. By possibly passing to a finite index submodule of $M$ we may assume without loss of generality that the natural map $M^\vee\to M^\vee[1/X]$ is injective whence the map $\id\otimes F\colon o/\varpi^h\bs X\js\otimes_{o/\varpi^h\bs X\js,F}M\to M$ is surjective. Let $f\in M$ be arbitrary. By continuity of $f$ there exists an integer $n\geq 0$ such that $f$ is constant on the cosets of $p^n\Zp$. Writing $f=\sum_{i=0}^{p^n-1}[i]\cdot F^n(f_i)$ (where $[i]\cdot$ denotes the multiplication by the group element $i\in\Zp$) by the surjectivity of $\id\otimes F$ we find that each $f_i$ is necessarily constant as a function on $\Zp$ satisfying $F_P^n(f_0(0))=f(0)$. Put $M_\ast:=\{f(0)\mid f\in M\}\subseteq \pi_P^{N_{L_P,0}}$. By the previous discussion $F_P$ acts surjectively on $M_\ast$ and is generated by the values of elements in $M^{\Zp}$ (ie.\ constant functions) as a module over $A[F_P]$. By the admissibility of $M$ we deduce that $M^{\Zp}$ hence $M_\ast$ is finite (or, equivalently, finitely generated over $o/\varpi^h$). We deduce that in fact we have $M=\mathcal{C}(\Zp,M_\ast)$, ie.\ $M^\vee\cong o/\varpi^h\bs X\js\otimes_{o/\varpi^h}M_\ast^\vee$. Conversely, whenever we have a $o/\varpi^h[F_P]$-submodule $M'\leq\pi_P^{N_{L_P,0}}$ that is finitely generated over $o/\varpi^h$ and on which $F_P$ acts surjectively (hence bijectively as the cardinality of $o/\varpi^h$ is finite) then for $M:=\mathcal{C}(\Zp,M')$ we have $M'=M_\ast$, $M\in \mathcal{M}(\mathcal{C}(\Zp,\pi_P^{N_{L_P,0}}))$, and $M^\vee\cong o/\varpi^h\bs X\js\otimes_{o/\varpi^h}(M')^\vee$ is $X$-torsion free. In particular, we compute 
\begin{align*}
D^\vee_\xi(\mathcal{C}_1(\pi_P))\cong\varprojlim_{M\in\mathcal{M}(\mathcal{C}(\Zp,\pi_P^{N_{L_P,0}}))}M^\vee[1/X]\cong\\
\cong\varprojlim_{\substack{M\in\mathcal{M}(\mathcal{C}(\Zp,\pi_P^{N_{L_P,0}})),\\ M^\vee\hookrightarrow M^\vee[1/X]}}o/\varpi^h\bg X\jg\otimes_{o/\varpi^h}M_\ast^\vee\cong\\
o/\varpi^h\bg X\jg\widehat{\otimes}_{o/\varpi^h}(\varinjlim_{\substack{M\in\mathcal{M}(\mathcal{C}(\Zp,\pi_P^{N_{L_P,0}})),\\ M^\vee\hookrightarrow M^\vee[1/X]}}M_\ast)^\vee =\\  
=o/\varpi^h\bg X\jg\widehat{\otimes}_{o/\varpi^h}\mathrm{Ord}_{s^\mathbb{Z}N_{L_P}}(\pi_P)^\vee
\end{align*}
as claimed.
\qed\end{proof}

\begin{cor}
Assume $L_P\cong \GL_2(\mathbb{Q}_p)\times T'$ where $T'$ is a torus and let $\pi_P\cong \pi_2\otimes_{k}\chi$ be the twist of a supercuspidal modulo $p$ representation $\pi_2$ of $\GL_2(\mathbb{Q}_p)$ by a character $\chi$ of the torus. Then we have $$\dim_{k\bg X\jg}D^\vee_{\xi}\left(\Ind_{P^-}^G\pi_P\right)=\begin{cases}0&\text{if }N_\alpha\not\subseteq L_P\\
2&\text{if }N_\alpha\subseteq L_P\end{cases}\ .$$
\end{cor}
\begin{proof}
Let the superscript $^{(2)}$ denote the analogous construction of the subgroups $B,T,N,T_0$ and element $s$ of $G$ in case $G=\GL_2(\mathbb{Q}_p)$. Note that the torus $T^{(2)}$ is generated by $s^{(2)}$ and $T_0^{(2)}$. So in this case we have an isomorphism $\mathrm{Ord}_{s^\mathbb{Z}N_{L_P}}(\pi_P)\cong(\mathrm{Ord}_{B^{(2)}}(\pi_2)\otimes\chi)_{\mid k[F_P]}=0$ by the adjunction formula of Emerton's ordinary parts (Thm.\ 4.4.6 in \cite{E}). In the other case we apply Thm.\ 0.10 in \cite{C}.
\qed\end{proof}

\subsection{The action of $T_+$}

Our goal in this section is to define a $\varphi$-action of $T_+$ on $D^{\vee}_{\xi,\ell,\infty}(\pi)$ or, equivalently, on $D^\vee_\xi(\pi)$ extending the action of $\xi(\Zp\setminus\{0\})\leq T_+$ and making $D^{\vee}_{\xi,\ell,\infty}(\pi)$ an \'etale $T_+$-module over $\Lambda_\ell(N_0)$. Let $t\in T_+$ be arbitrary. Note that by the choice of this $\ell$ we have $tH_0t^{-1}\subseteq H_0$. In particular, $T_+$ acts via conjugation on the ring $\Lambda(N_0/H_0)\cong o\bs X\js$; we denote the action of $t\in T_+$ by $\varphi_t$. This action is via the character $\alpha$ mapping $T_+$ onto $\Zp\setminus\{0\}$. In particular, $o\bs X\js$ is a free module of finite rank over itself via $\varphi_t$. Moreover, we define the Hecke action of $t\in T_+$ on $\pi^{H_0}$ by the formula $F_t(m):=\Tr_{H_0/tH_0t^{-1}}(tm)$ for any $m\in \pi^{H_0}$. For $t,t'\in T_+$ we have 
\begin{align*}
F_{t'}\circ F_t=\Tr_{H_0/t'H_0t'^{-1}}\circ (t'\cdot)\circ\Tr_{H_0/tH_0t^{-1}}\circ (t\cdot)=\\
=\Tr_{H_0/t'H_0t'^{-1}}\circ\Tr_{t'H_0t'^{-1}/t'tH_0t^{-1}t'^{-1}}\circ(t't\cdot)=F_{t't}\ . 
\end{align*}
For any $M\in \mathcal{M}(\pi^{H_0})$ we put $F_t^*M:=N_0F_t(M)$.

\begin{lem}
For any $M\in\mathcal{M}(\pi^{H_0})$ we have $F_t^*M\in\mathcal{M}(\pi^{H_0})$.
\end{lem}
\begin{proof}
We have 
\begin{align*}
F(F_t^*M)=F(N_0F_t(M))\subset N_0FF_t(M)=\\
=N_0F_{st}(M)=N_0F_t(F(M))\subseteq F_t^*M\ . 
\end{align*}
So $F_t^*M$ is a module over $\Lambda(N_0/H_0)/\varpi^h[F]$. Moreover, if $m_1,\dots m_r$ generates $M$, then the elements $F_t(m_i)$ ($1\leq i\leq r$) generate $F_t^*M$, so it is finitely generated. The admissibility is clear as $F_t^*M=\sum_{u\in J(N_0/tN_0t^{-1})}uF_t(M)$ is the sum of finitely many admissible submodules. Finally, $F_t^*M$ is stable under the action of $\Gamma$ as $F_t$ commutes with the action of $\Gamma$.
\qed\end{proof}

By the definition of $F_t^*M$ we have a surjective $o/\varpi^h\bs X\js$-homomorphism
\begin{equation*}
1\otimes F_t\colon o/\varpi^h\bs X\js \otimes_{o/\varpi^h\bs X\js,\varphi_t}M\twoheadrightarrow F_t^*M
\end{equation*}
which gives rise to an injective $o/\varpi^h\bg X\jg$-homomorphism
\begin{equation}
(1\otimes F_t)^\vee[1/X]\colon (F_t^*M)^\vee[1/X]\hookrightarrow o/\varpi^h\bg X\jg \otimes_{o/\varpi^h\bg X\jg,\varphi_t}M^\vee[1/X]\ . \label{1otimesFt}
\end{equation}
Moreover, there is a structure of an $o/\varpi^h\bs X\js[F]$-module on $$o/\varpi^h\bs X\js \otimes_{o/\varpi^h\bs X\js,\varphi_t}M$$ by putting $F(\lambda\otimes m):=\varphi_t(\lambda)\otimes F(m)$. Similarly, the group $\Gamma$ also acts on $o/\varpi^h\bs X\js \otimes_{o/\varpi^h\bs X\js,\varphi_t}M$ semilinearly. The map $1\otimes F_t$ is $F$ and $\Gamma$-equivariant as $F_t$, $F$, and the action of $\Gamma$ all commute. We deduce that $(1\otimes F_t)^\vee[1/X]$ is a $\varphi$- and $\Gamma$-equivariant map of \'etlae $(\varphi,\Gamma)$-modules.

Note that for any $t\in T_+$ there exists a positive integer $k\geq 0$ such that $t\leq s^k$, ie.\ $t':=t^{-1}s^k$ lies in $T_+$. So we have $F_t^*(F_{t'}^*M)=F_{s^k}^*M=N_0F^k(M)\subseteq M$. So we obtain an isomorphism $M^\vee[1/X]\cong (F_{s^k}^*M)^\vee[1/X]= (F_t^*(F_{t'}^*M))^\vee[1/X]$ as $M/N_0F^k(M)$ is finitely generated over $o$.

\begin{lem}
The map \eqref{1otimesFt} is an isomorphism of \'etale $(\varphi,\Gamma)$-modules for any $M\in \mathcal{M}(\pi^{H_0})$ and $t\in T_+$.
\end{lem}
\begin{proof}
The composite $(1\otimes F_{t'})^\vee[1/X]\circ (1\otimes F_t)^\vee[1/X]=(1\otimes F^k)^\vee[1/X]$ is an isomorphism by Lemma 2.6 in \cite{B}. So $(1\otimes F_t)^\vee[1/X]$ is also an isomorphism as both $(1\otimes F_t)^\vee[1/X]$ and $(1\otimes F_{t'})^\vee[1/X]$ are injective.
\qed\end{proof}

Now taking projective limits  we obtain an isomorphism of pseudocompact \'etale $(\varphi,\Gamma)$-modules
\begin{eqnarray*}
(1\otimes F_t)^\vee[1/X]\colon D^\vee_\xi(\pi)&\to& \varprojlim_{M\in\mathcal{M}(\pi^{H_0})}(o/\varpi^h\bg X\jg\otimes_{o/\varpi^h\bg X\jg,\varphi_t}M^\vee[1/X])\\
(m)_{(F_t^*M)^\vee[1/X]}&\mapsto & ((1\otimes F_t)^\vee[1/X](m))_{M^\vee[1/X]}\ .
\end{eqnarray*}
Moreover, since $o\bg X\jg$ is finite free over itself via $\varphi_t$, we have an identification
\begin{align*}
 \varprojlim_{M\in\mathcal{M}(\pi^{H_0})}(o/\varpi^h\bg X\jg\otimes_{o/\varpi^h\bg X\jg,\varphi_t}M^\vee[1/X])\cong\\
\cong o/\varpi^h\bg X\jg\otimes_{o/\varpi^h\bg X\jg,\varphi_t}D^\vee_\xi(\pi)\ .
\end{align*}

Using the maps $(1\otimes F_t)^\vee[1/X]$ we define a $\varphi$-action of $T_+$ on $D^\vee_\xi(\pi)$ by putting $\varphi_t(d):=((1\otimes F_t)^\vee[1/X])^{-1}(1\otimes d)$ for $d\in D^\vee_\xi(\pi)$.
\begin{pro}
The above action of $T_+$ extends the action of $\xi(\Zp\setminus\{0\})\leq T_+$ and makes $D^\vee_\xi(\pi)$ into an \'etale $T_+$-module over $o/\varpi^h\bs X\js$.
\end{pro}
\begin{proof}
By the definition of the $T_+$-action it is indeed an extension of the action of the monoid $\Zp\setminus\{0\}$. For $t,t'\in T_+$ we compute
\begin{align*}
\varphi_{t'}\circ\varphi_t(d)=((1\otimes F_{t'})^\vee[1/X])^{-1}\circ((1\otimes F_t)^\vee[1/X])^{-1}(1\otimes d)=\\
=((1\otimes F_{t})^\vee[1/X]\circ(1\otimes F_{t'})^\vee[1/X])^{-1}(1\otimes d)=\\
=((1\otimes F_{tt'})^\vee[1/X])^{-1}(1\otimes d)=\varphi_{tt'}(d)=\varphi_{t't}(d)\ .
\end{align*}
Further, we have 
\begin{align*}
\varphi_t(\lambda d)=((1\otimes F_t)^\vee[1/X])^{-1}(1\otimes \lambda d)=((1\otimes F_t)^\vee[1/X])^{-1}(\varphi_t(\lambda)\otimes d)=\\
=\varphi_t(\lambda)((1\otimes F_t)^\vee[1/X])^{-1}(1\otimes d)=\varphi_t(\lambda)\varphi_t(d)
\end{align*}
showing that this is indeed a $\varphi$-action of $T_+$. The \'etale property follows from the fact that $(1\otimes F_t)^\vee[1/X]$ is an isomorphism for each $t\in T_+$.
\qed\end{proof}

The inclusion $u_{\alpha}\colon \Zp\to N_{\alpha,0}\leq N_0$ induces an injective ring homomor-phism---still denoted by $u_\alpha$ by a certain abuse of notation---$u_\alpha\colon \widehat{o\bg X\jg}^{p} \hookrightarrow \Lambda_\ell(N_0)$ where $\widehat{o\bg X\jg}^{p}$ denotes the $p$-adic completion of the Laurent-series ring $o\bg X\jg$. For each $t\in T_+$ this gives rise to a commutative diagram
\begin{equation*}
\xymatrix{
\widehat{o\bg X\jg}^{p}\ar[r]^{u_\alpha}\ar[d]_{\varphi_t}&\Lambda_\ell(N_0)\ar[d]^{\varphi_t}\\
\widehat{o\bg X\jg}^{p}\ar[r]^{u_\alpha}&\Lambda_\ell(N_0)
}
\end{equation*}
with injective ring homomorphisms. On the other hand, by the equivalence of categories in Thm.\ 8.20 in \cite{SVZ} we have a $\varphi$- and $\Gamma$-equivariant identification $M_\infty^\vee[1/X]\cong \Lambda_\ell(N_0)\otimes_{\widehat{o\bg X\jg}^{p}, u_\alpha}M^\vee[1/X]$. Therefore tensoring the isomorphism \eqref{1otimesFt} with $\Lambda_\ell(N_0)$ via $u_\alpha$ we obtain an isomorphism
\begin{align}
(1\otimes F_t)_\infty^\vee[1/X] \colon (F_t^*M)_\infty^\vee[1/X]\cong \Lambda_\ell(N_0)\otimes_{u_\alpha}(F_t^*M)^\vee[1/X]\to\notag\\
\to \Lambda_\ell(N_0)\otimes_{u_\alpha}o/\varpi^h\bg X\jg \otimes_{o/\varpi^h\bg X\jg,\varphi_t}M^\vee[1/X]\cong\notag\\
\cong  \Lambda_\ell(N_0) \otimes_{\Lambda_\ell(N_0),\varphi_t}\Lambda_\ell(N_0)\otimes_{u_\alpha}M^\vee[1/X]\cong \Lambda_\ell(N_0) \otimes_{\Lambda_\ell(N_0),\varphi_t}M_\infty^\vee[1/X]\ .\label{Ftinfty}
\end{align}
Taking projective limits again we deduce an isomorphism
\begin{eqnarray*}
(1\otimes F_t)_\infty^\vee[1/X]\colon D^{\vee}_{\xi,\ell,\infty}(\pi)&\to& \Lambda_\ell(N_0)\otimes_{\Lambda_\ell(N_0),\varphi_t}D^{\vee}_{\xi,\ell,\infty}(\pi) \\
(m)_{(F_t^*M)_\infty^\vee[1/X]}&\mapsto & ((1\otimes F_t)_\infty^\vee[1/X](m))_{M_\infty^\vee[1/X]}
\end{eqnarray*}
for all $t\in T_+$ using the identification 
\begin{equation*}
 \varprojlim_{M\in\mathcal{M}(\pi^{H_0})}(\Lambda_\ell(N_0)\otimes_{\Lambda_\ell(N_0),\varphi_t}M_\infty^\vee[1/X])\cong \Lambda_\ell(N_0)\otimes_{\Lambda_\ell(N_0),\varphi_t}D^{\vee}_{\xi,\ell,\infty}(\pi)\ .
\end{equation*}
Using the maps $(1\otimes F_t)_\infty^\vee[1/X]$ we define a $\varphi$-action of $T_+$ on $D^{\vee}_{\xi,\ell,\infty}(\pi)$ by putting $\varphi_t(d):=((1\otimes F_t)_\infty^\vee[1/X])^{-1}(1\otimes d)$ for $d\in D^{\vee}_{\xi,\ell,\infty}(\pi)$.
\begin{cor}
The above action of $T_+$ extends the action of $\xi(\Zp\setminus\{0\})\leq T_+$ and makes $D^{\vee}_{\xi,\ell,\infty}(\pi)$ into an \'etale $T_+$-module over $\Lambda_\ell(N_0)$. The reduction map $D^{\vee}_{\xi,\ell,\infty}(\pi)\to D^\vee_\xi(\pi)$ is $T_+$-equivariant for the $\varphi$-action.
\end{cor}

We can view this $\varphi$-action of $T_+$ in a different way: Let us define $F_{t,k}:=\Tr_{H_k/tH_kt^{-1}}\circ(t\cdot)$. Then we have a map
\begin{equation}\label{Ftk}
1\otimes F_{t,k}\colon \Lambda(N_0/H_k)/\varpi^h\otimes_{\Lambda(N_0/H_k)/\varpi^h,\varphi_t}M_k\to F_{t,k}^*M_k:=N_0F_{t,k}(M_k)\ ,
\end{equation}
where we have $F_{t,k}^*M\in\mathcal{M}_k(\pi^{H_k})$. Let $k$ be large enough such that we have $tH_0t^{-1}\geq H_k$.  After taking Pontryagin duals, inverting $X$, taking projective limit and using the remark after Lemma \ref{dualtensor} we obtain a homomorphism of \'etale $(\varphi,\Gamma)$-modules
\begin{equation}
\varprojlim_{k}\Tr_{t^{-1}H_kt}^{-1}\circ(1\otimes F_{t,k})^\vee[1/X]\colon (F_t^*M)^\vee_\infty[1/X] \to  \Lambda_\ell(N_0)\otimes_{\varphi_t}M^\vee_\infty[1/X]\ .\label{limFtk}
\end{equation}
This map is indeed $\Gamma$- and $\varphi$-equivariant because we compute
\begin{align*}
F_k\circ F_{t,k}=\Tr_{H_k/sH_ks^{-1}}\circ(s\cdot)\circ\Tr_{H_k/tH_kt^{-1}}\circ(t\cdot)=\\
=\Tr_{H_k/s^ktH_kt^{-1}s^{-k}}\circ(s^kt\cdot)=\\
=\Tr_{H_k/tH_kt^{-1}}\circ(t\cdot)\circ\Tr_{H_k/sH_ks^{-1}}\circ(s\cdot)=F_{t,k}\circ F_k\ .
\end{align*}
Now we have two maps \eqref{Ftinfty} and \eqref{limFtk} between $(F_t^*M)^\vee_\infty[1/X]$ and $\Lambda_\ell(N_0)\otimes_{\varphi_t}M^\vee_\infty[1/X]$ that agree after taking $H_0$-coinvariants by definition. Hence they are equal by the equivalence of categories in Thm.\ 8.20 in \cite{SVZ}.

We obtain in particular that the map \eqref{Ftk} has finite kernel and cokernel as it becomes an isomorphism after taking Pontryagin duals and inverting $X$. Hence there exists a finite $\Lambda(N_0/H_k)/\varpi^h$-submodule $M_{t,k,\ast}$ of $M_k$ such that the kernel of $1\otimes F_{t,k}$ is contained in the image of $\Lambda(N_0/H_k)/\varpi^h\otimes_{\varphi}M_{t,k,\ast}$ in $\Lambda(N_0/H_k)/\varpi^h\otimes_{\varphi}M_k$. We denote by $M_{t,k}^\ast\leq F_{t,k}^\ast M_k$ the image of $1\otimes F_{t,k}$. We conclude that as in Proposition \ref{M_ketale}, we can describe the $\varphi_t$-action in the following way:
\begin{eqnarray}\label{varphi_tk}
\varphi_t \colon M_k^\vee[1/X]&\to& (F_{t,k}^*M_k)^\vee[1/X]\notag\\
f&\mapsto &(\Tr_{t^{-1}H_kt/H_k}^{-1}\circ(1\otimes F_{t,k})^\vee[1/X])^{-1}(1\otimes f)
\end{eqnarray}

Being an \'etale $T_+$-module over $\Lambda_\ell(N_0)$ we equip $D^{\vee}_{\xi,\ell,\infty}(\pi)$ with the $\psi$-action of $T_+$: $\psi_t$ is the canonical left inverse of $\varphi_t$ for all $t\in T_+$.

\begin{pro}\label{prpsiT+}
The map $\pr\colon D_{SV}(\pi)\to D^{\vee}_{\xi,\ell,\infty}(\pi)$ is $\psi$-equivariant for the $\psi$-actions of $T_+$ on both sides.
\end{pro}
\begin{proof}
We proceed as in the proofs of Proposition \ref{limit} and Lemma \ref{prpsigamma}. We fix $t\in T_+$, $W\in\mathcal{B}_+(\pi)$ and $M\in\mathcal{M}(\pi^{H_0})$ and show that $\pr_{W,M}$ is $\psi_t$-equivariant. Fix $k$ such that $F_{t,k}^*M_k\leq W$ and $tH_0t^{-1}\geq H_k$.

At first we compute the formula analogous to \eqref{varphik}. Let $f$ be in $M_k^\vee$ such that its restriction to $M_{t,k,\ast}$ is zero and $m\in M_{t,k}^\ast\leq F_{t,k}^*M_k$ be in the form 
\begin{equation*}
m=\sum_{u\in J(N_0/tN_0t^{-1})}u F_{t,k}(m_u)
\end{equation*}
with elements $m_u\in M_k$ for $u\in J(N_0/tN_0t^{-1})$. $M_{t,k}^\ast$ is a finite index submodule of $F_{t,k}^*M_k$. Note that the elements $m_u$ are unique upto $M_{t,k,\ast}+\Ker(F_{t,k})$. Therefore $\varphi_t(f)\in (M_{t,k}^\ast)^\vee$ is well-defined by our assumption that $f_{\mid M_{t,k,\ast}}=0$ noting that the kernel of $F_{t,k}$ equals the kernel of $\Tr_{t^{-1}H_kt/H_k}$ since the multiplication by $t$ is injective and we have $F_{t,k}=t\circ\Tr_{t^{-1}H_kt/H_k}$. So we compute
\begin{align}
\varphi_t(f)(m)=((1\otimes F_{t,k})^\vee)^{-1}(\Tr_{t^{-1}H_kt/H_k}(1\otimes f))(m)=\notag\\
=((1\otimes F_{t,k})^\vee)^{-1}(1\otimes \Tr_{t^{-1}H_kt/H_k}(f))(\sum_{u\in J((N_0/H_k)/t(N_0/H_k)t^{-1})}u F_{t,k}(m_u))=\notag\\
=\Tr_{t^{-1}H_kt/H_k}(f)(F_{t,k}^{-1}(u_0F_{t,k}(m_{u_0})))=f(\Tr_{t^{-1}H_kt/H_k}((t^{-1}u_0t)m_{u_0}))\label{varphit}
\end{align}
where $u_0$ is the single element in $J(N_0/tN_0t^{-1})$ corresponding to the coset of $1$.

Now let $f$ be in $W^\vee$ such that the restriction $f_{\mid N_0tM_{t,k,\ast}}=0$. By definition we have $\psi_t(f)(w)=f(tw)$ for any $w\in W$. Choose an element $m\in M_{t,k}^\ast\in F_{t,k}^*M_k$ written in the form

$$m=\sum_{u\in J(N_0/tN_0t^{-1})}u F_{t,k}(m_u)=\sum_{u\in J(N_0/tN_0t^{-1})}ut\Tr_{t^{-1}H_kt/H_k}(m_u)\ .$$

Then we compute
\begin{align*}
f_{\mid F_{t,k}^*M_k}(m)=\sum_{u\in J(N_0/tN_0t^{-1})}f(ut\Tr_{t^{-1}H_kt/H_k}(m_u))=\\
=\sum_{u\in J(N_0/tN_0t^{-1})}\psi_t(u^{-1}f)(\Tr_{t^{-1}H_kt/H_k}(m_u))=\\
\overset{\eqref{varphit}}{=}\sum_{u\in J(N_0/tN_0t^{-1})}\varphi_t(\psi_t(u^{-1}f)_{\mid F_{t,k}^*M_k})(F_{t,k}(m_u))=\\
=\sum_{u\in J(N_0/tN_0t^{-1})}u\varphi_t(\psi_t(u^{-1}f)_{\mid M_k})(uF_{t,k}(m_u))=\\
=\sum_{u\in J(N_0/tN_0t^{-1})}u\varphi_t(\psi_t(u^{-1}f)_{\mid M_k})(m)
\end{align*}
as for distinct $u,v\in J(N_0/tN_0t^{-1})$ we have $u\varphi_t(f_0)(vF_{t,k}(m_v))=0$ for any $f_0\in (M_{t,k}^\ast)^\vee$. So by inverting $X$ and taking projective limits with respect to $k$ we obtain
\begin{equation*}
\pr_{W,F_t^*M}(f)=\sum_{u\in J(N_0/tN_0t^{-1})}u\varphi_t(\pr_{W,M}(\psi_t(u^{-1}f)))
\end{equation*}
as we have $(M_{t,k}^\ast)^\vee[1/X]\cong (F_{t,k}^*M)^\vee[1/X]$. Since the map \eqref{Ftinfty} is an isomorphism we may decompose $\pr_{W,F_t^*M}(f)$ uniquely as
\begin{equation*}
\pr_{W,F_t^*M}(f)=\sum_{u\in J(N_0/tN_0t^{-1})}u\varphi_t(\psi_t(u^{-1}\pr_{W,F_t^*M}(f)))
\end{equation*}
so we must have $\psi_t(\pr_{W,F_t^*M}(f))=\pr_{W,M}(\psi_t(f))$. For general $f\in W^\vee$ note that $N_0sM_{t,k,*}$ is killed by $\varphi_t(X^r)$ for $r\geq 0$ big enough, so we have 
\begin{align*}
X^r\psi_t(\pr_{W,F_t^*M}(f))=\psi_t(\pr_{W,F_t^*M}(\varphi_t(X^r)f))=\\
=\pr_{W,M}(\psi_t(\varphi_t(X^r)f))=X^r\pr_{W,M}(\psi_t(f))\ . 
\end{align*}
Since $X^r$ is invertible in $\Lambda_\ell(N_0)$, we obtain $$\psi_t(\pr_{W,F_t^*M}(f))=\pr_{W,M}(\psi_t(f))$$ for any $f\in W^\vee$. The statement follows taking the projective limit with respect to $M\in\mathcal{M}(\pi^{H_0})$ and the inductive limit with respect to $W\in\mathcal{B}_+(\pi)$.
\qed\end{proof}

We end this section by proving a Lemma that will be needed several times later on.
\begin{lem}\label{MT'}
For any $M\in \mathcal{M}(\pi^{H_0})$ there exists an open subgroup $T'=T'(M)\leq T$ such that $M$ is $T'$-stable.
\end{lem}
\begin{proof}
Choose $m_1,\dots,m_a\in M$ ($a\geq1$)  generating $M$ as a module over $o/\varpi^h\bs X\js[F]$. Since $\pi$ is smooth, there exists an open subgroup $T'\leq T_0$ stabilizing all $m_1,\dots,m_a$. Now $T'$ normalizes $N_0$ and all the elements $t\in T'$ commute with $F$ we deduce that $T'$ acts on $M$.
\qed\end{proof}

\section{Compatibility with a reverse functor}

Assume $\ell=\ell_\alpha$ for some simple root $\alpha\in \Delta$ so we may apply the results of section \ref{nongeneric}.

\subsection{A $G$-equivariant sheaf $\mathfrak{Y}$ on $G/B$ attached to $D^\vee_{\xi,\ell,\infty}(\pi)$}\label{sheaf}

Let $D$ be an \'etale $(\varphi,\Gamma)$-module over the ring $\Lambda_\ell(N_0)/\varpi^h$. Recall that the $\Lambda(N_0)$-submodule $D^{bd}$ of bounded elements in $D$ is defined \cite{SVZ} as
\begin{equation*}
D^{bd}=\{x\in D\mid \{\ell_D(\psi_s^k(u^{-1}x))\mid k\geq 0,u\in N_0\}\subseteq D_{H_0} \text{ is bounded}\}\ .
\end{equation*}
where $\ell_D$ denotes the natural map $D\to D_{H_0}$. Note that $D_{H_0}$ is an \'etale $(\varphi,\Gamma)$-module over $o/\varpi^h\bg X\jg$, so the bounded subsets of $D_{H_0}$ are exactly those contained in a compact $o/\varpi^h\bs X\js$-submodule of $D_{H_0}$.

\begin{lem}
Assume that $D$ is a finitely generated \'etale $(\varphi,\Gamma)$-module over $\Lambda_\ell(N_0)/\varpi^h$. Then $d\in D$ lies in $D^{bd}$ if and only if $d$ is contained in a compact $\psi_s$-invariant $\Lambda(N_0)$-submodule of $D$.
\end{lem}
\begin{proof}
If $d$ is in $D^{bd}$ then it is contained in $$D^{bd}(D_0)=\{x\in D\mid \ell_D(\psi_{s}^k(u^{-1}x))\subseteq D_0\}$$ for some treillis $D_0\subset D_{H_0}$ where $D^{bd}(D_0)$ is a compact $\psi_s$-stable $\Lambda(N_0)$-submodule of $D$ by Prop.\ 9.10 in \cite{SVZ}. On the other hand if $x\in D_1$ for some compact $\psi_s$-invariant $\Lambda(N_0)$-submodule $D_1\subset D$ then we have $$\{\ell_D(\psi_s^k(u^{-1}x))\mid k\geq 0,u\in N_0\}\subseteq \ell_D(D_1)$$ where $\ell_D(D_1)$ is bounded as $D_1$ is compact and $\ell_D$ is continuous.
\qed\end{proof}

We call a pseudocompact $\Lambda_\ell(N_0)$-module together with a $\varphi$-action of the monoid $T_+$ (resp.\ $\Zp\setminus\{0\}$) a pseudocompact \'etale $T_+$-module (resp.\ $(\varphi,\Gamma)$-module) over $\Lambda_\ell(N_0)$ if  it is a topologically \'etale $o[B_+]$-module in the sense of section 4.1 in \cite{SVZ}. Recall that a pseudocompact module over the pseudocompact ring $\Lambda_\ell(N_0)$ is the projective limit of finitely generated $\Lambda_\ell(N_0)$-modules. As for $D=D^\vee_{\xi,\ell,\infty}(\pi)$ in section \ref{multvarbr} we equip the pseudocompact $\Lambda_\ell(N_0)$-modules $D$ with the weak topology, ie.\ with the projective limit  topology of the weak topologies of these finitely generated quotients of $D$. Recall from section 4.1 in \cite{SVZ} that the condition for $D$ to be topologically \'etale means in this case that the map
\begin{eqnarray}
B_+\times D&\to& D\notag\\
(b,x)&\mapsto&\varphi_b(x)\label{multbyB+}
\end{eqnarray}
is continuous and $\psi=\psi_s\colon D\to D$ is continuous (Lemma 4.1 in \cite{SVZ}).

\begin{lem}\label{B+cont}
$D^\vee_{\xi,\ell,\infty}(\pi)$ is a pseudocompact \'etale $T_+$-module over $\Lambda_\ell(N_0)$. 
\end{lem}
\begin{proof}
At first we show that the map \eqref{multbyB+} is continuous in the weak topology of $D=D^\vee_{\xi,\ell,\infty}(\pi)$. Let $b=ut\in B_+$ ($u\in N_0$, $t\in T_+$), $x,y\in D^\vee_{\xi,\ell,\infty}(\pi)$ be such that $u\varphi_t(y)=x$ and let $M\in \mathcal{M}(\pi^{H_0})$, $l,l'\geq 0$ be arbitrary. Recall from \eqref{OMll'} that the sets
\begin{equation*}
O(M,l,l'):=f_{M,l}^{-1}(\Lambda(N_0/H_l)\otimes_{u_\alpha}X^{l'}M^\vee[1/X]^{++})
\end{equation*}
form a system of neighbourhoods of $0$ in the weak topology of $D^\vee_{\xi,\ell,\infty}(\pi)$. We need to verify that the preimage of $x+O(M,l,l')$ under \eqref{multbyB+} contains a neighbourhood of $(b,y)$. By Lemma \ref{MT'} there exists an open subgroup $T'\leq T_0\leq T$ acting on $M$ therefore also on $M_l^\vee[1/X]$ as $T_0$ normalizes $H_l$ for all $l\geq 0$ by the assumption $\ell=\ell_\alpha$. Moreover, this action is continuous in the weak topology of $M_l^\vee[1/X]$, so there exists an open subgroup $T_1\leq T'$ such that we have $(T_1-1)x\subset O(M,l,l')$.
Moreover, since we have $D^\vee_{\xi,\ell,\infty}(\pi)/O(M,l,l')\cong M_l^\vee[1/X]/(\Lambda(N_0/H_l)\otimes_{u_\alpha}X^{l'}M^\vee[1/X]^{++})$ is a smooth representation of $N_0$, we have an open subgroup $N_1\leq N_0$ with $(N_1-1)x\subset O(M,l,l')$. Moreover, we may assume that $T_1$ normalizes $N_1$ so that $B_1:=N_1T_1$ is an open subgroup in $B_0\leq B_+$ for which we have $(B_1-1)x\subset O(M,l,l')$ as $O(M,l,l')$ is $N_0$-invariant. Choose an elemet $t'\in T_+$ such that $tt'=s^r$ for some $r\geq 0$. Note that the composite map $D^\vee_{\xi,\ell,\infty}(\pi)\overset{\varphi_{t}}{\rightarrow}D^\vee_{\xi,\ell,\infty}\twoheadrightarrow M^\vee[1/X]$ factors through the $\varphi_s$-equivariant map $$((1\otimes F_{t})^\vee[1/X])^{-1}\colon (F_{t'}^*M)^\vee[1/X]\to M^\vee[1/X]$$ mapping $X^{l'}(F_{t'}^*M)^\vee[1/X]^{++}$ into $X^{l'}M^\vee[1/X]^{++}$.
Since $X^{l'}M^\vee[1/X]^{++}$ is $B_1$-invariant (as each $\varphi_{t_1}$ for $t_1\in T_1$ commutes with $\varphi_s$), so is $O(M,l,l')$. We deduce that $$B_1b\times (y+O(F_{t'}^*M,l,l'))\subset B_+\times D^\vee_{\xi,\ell,\infty}(\pi)$$ maps into $x+O(M,l,l')$ via \eqref{multbyB+}.

The continuity of $\psi_s$ follows from Proposition 8.22 in \cite{SVZ} since $\psi_s$ on  $D^\vee_{\xi,\ell,\infty}(\pi)$ is the projective limit of the maps $\psi_s\colon M_\infty^\vee[1/X]\to M_\infty^\vee[1/X]$ for $M\in\mathcal{M}(\pi^{H_0})$.
\qed\end{proof}

In view of the above Lemmata we define $D^{bd}$ for a pseudocompact \'etale $(\varphi,\Gamma)$-module $D$ over $\Lambda_\ell(N_0)$ as 
\begin{equation*}
D^{bd}=\bigcup_{D_c\in\mathfrak{C}_0(D)}D_c
\end{equation*}
where we denote the set of $\psi_s$-invariant compact $\Lambda(N_0)$-submodules $D_c\subset D$ by $\mathfrak{C}_0=\mathfrak{C}_0(D)$.

The following is a generalization of Prop.\ 9.5 in \cite{SVZ}.
\begin{pro}\label{bdetale}
Let $D$ be a pseudocompact \'etale $(\varphi,\Gamma)$-module over $\Lambda_\ell(N_0)$. Then $D^{bd}$ is an \'etale $(\varphi,\Gamma)$-module over $\Lambda(N_0)$. If we assume in addition that $D$ is an \'etale $T_+$-module over $\Lambda_\ell(N_0)$ (for a $\varphi$-action of the monoid $T_+$ extending that of $\xi(\Zp\setminus\{0\})$) then $D^{bd}$ is an \'etale $T_+$-module over $\Lambda(N_0)$ (with respect to the action of $T_+$ restricted from $D$).
\end{pro}
\begin{proof}
We prove the second statement assuming that $D$ is an \'etale $T_+$-module. The first statement follows easily the same way. 

At first note that $D^{bd}$ is $\psi_t$-invariant for all $t\in T_+$ as for $D_c\in\mathfrak{C}_0$ we also have $\psi_t(D_c)\in\mathfrak{C}_0$. So it suffices to show that it is also stable under the $\varphi$-action of $T_+$ since these two actions are clearly compatible (as they are compatible on $D$). At first we show that we have $\varphi_s(D^{bd})\subset D^{bd}$. Let $D_c\in \mathfrak{C}_0$ be arbitrary. Then the $\psi$-action of the monoid $p^{\mathbb{Z}}$ (ie.\ the action of $\psi_s$) is nondegenerate on $D_c$ as $D_c$ is a $\psi_s$-invariant submodule of an \'etale module $D$. So by the remark after Proposition \ref{etalehull} and by Corollary \ref{tildeinjective} we obtain an injective $\psi_s$ and $\varphi_s$-equivariant homomorphism $i\colon\widetilde{D}_c\hookrightarrow D$.
However, each $\varphi_{s^k}^*D_c\subseteq \widetilde{D}_c$ is compact and $\psi$-equivariant therefore the image of $\widetilde{D}_c$ is contained in $D^{bd}$ showing that $\varphi_s(D_c)\subset N_0\varphi_s(D_c)=i(\varphi_s^*D_c)\subseteq D^{bd}$. However, for each $t\in T_+$ there exists a $t'\in T_+$ with $tt'=s^k$ for some $k\geq 0$, so $\varphi_t(D_c)=\psi_{t'}(\varphi_{s^k}(D_c))\subseteq D^{bd}$ showing that $D^{bd}$ is $\varphi_t$-invariant for all $t\in T_+$.
\qed\end{proof}

\begin{cor}\label{prbd}
The image of the map $\widetilde{\pr}\colon\widetilde{D_{SV}}(\pi)\to D^\vee_{\xi,\ell,\infty}(\pi)$ is contained in $D^\vee_{\xi,\ell,\infty}(\pi)^{bd}$.
\end{cor}
\begin{proof}
By Propositions \ref{etalehull} and \ref{bdetale} it suffices to show that the image of $\pr\colon D_{SV}(\pi)\to D^\vee_{\xi,\ell,\infty}(\pi)$ lies in $D^\vee_{\xi,\ell,\infty}(\pi)^{bd}$. However, this is clear since $\pr(D_{SV}(\pi))$ is a $\psi_s$-invariant compact $\Lambda(N_0)$-submodule of $D^\vee_{\xi,\ell,\infty}(\pi)$.
\qed\end{proof}

Let $\mathfrak{C}$ be the set of all compact subsets $C$ of $D^\vee_{\xi,\ell,\infty}(\pi)$ contained in one of the compact subsets $D_c\in\mathfrak{C}_0=\mathfrak{C}_0(D^\vee_{\xi,\ell,\infty}(\pi))$. Recall from Definition 6.1 in \cite{SVZ} that the family $\mathfrak{C}$ is said to be special if it satisfies the following axioms:
\begin{enumerate}[$\mathfrak{C}(1)$]
\item Any  compact subset of  a compact set in
$ \mathfrak{C}$  also lies in $\mathfrak{C}$.
\item If $C_1,C_2,\dots,C_n\in\mathfrak{C}$ then $\bigcup_{i=1}^nC_i$ is in
  $\mathfrak{C}$, as well.
\item For all $C\in\mathfrak{C}$ we have $N_0C\in\mathfrak{C}$.
\item $D(\mathfrak{C}):=\bigcup_{C\in\mathfrak{C}}C$ is an \'etale $T_+$-submodule
of $D$.
\end{enumerate}

\begin{lem}
The set $\mathfrak{C}$ is a special family of compact sets in $D^\vee_{\xi,\ell,\infty}(\pi)$ in the sense of Definition 6.1 in \cite{SVZ}.
\end{lem}
\begin{proof}
$\mathfrak{C}(1)$ is satisfied by construction. So is $\mathfrak{C}(3)$ by noting that any $C\in\mathfrak{C}$ is contained in a $D_c\in \mathfrak{C}_0$ which is $N_0$-stable. For $\mathfrak{C}(2)$ note that for any $D_{c,1},\dots,D_{c,r}\in\mathfrak{C}_0$ we have $\sum_{i=1}^rD_{c,i}\in\mathfrak{C}_0$. Finally, $\mathfrak{C}(4)$ is just Proposition \ref{bdetale}.
\qed\end{proof}

Our next goal is to construct a $G$-equivariant sheaf $\mathfrak{Y}=\mathfrak{Y}_{\alpha,\pi}$ on $G/B$ in \cite{SVZ} with sections $\mathfrak{Y}(\mathcal{C}_0)$ on $\mathcal{C}_0:=N_0w_0B/B$ isomorphic to $D^\vee_{\xi,\ell,\infty}(\pi)^{bd}$ as a $B_+$-module. Here $w_0\in N_G(T)$ is a representative of an element in the Weyl group $N_G(T)/C_G(T)$ of \emph{maximal length}. For this we identify $D^\vee_{\xi,\ell,\infty}(\pi)^{bd}$ with the global sections of a $B_+$-equivariant sheaf on $N_0$ as in \cite{SVZ}. The restriction maps $\res_{us^kN_0s^{-k}}^{N_0}$ are defined as $u\circ\varphi_s^k\circ\psi_s^k\circ u^{-1}$. The open sets $us^kN_0s^{-k}$ form a basis of the topology on $N_0$, so it suffices to give these restriction maps. Indeed,  any open compact subset $\mathcal{U}\subseteq N_0$ is the disjoint union of cosets of the form $us^kN_0s^{-k}$ for $k\geq k'(\mathcal{U})$ large enough. For a fixed $k\geq k'(\mathcal{U})$ we put
\begin{align*}
\res_{\mathcal{U}}=\res^{N_0}_{\mathcal{U}}:=\sum_{u\in J(N_0/s^kN_0s^{-k})\cap \mathcal{U}}u\varphi_{s^k}\circ\psi_s^k\circ(u^{-1}\cdot)\ .
\end{align*}

This is independent of the choice of $k\geq k'(\mathcal{U})$ by Prop.\ 3.16 in \cite{SVZ}. Note that the map $$u\mapsto x_u:=uw_0B/B\in\mathcal{C}_0$$ is a $B_+$-equivariant homeomorphism from $N_0$ to $\mathcal{C}_0$ therefore we may view $D^\vee_{\xi,\ell,\infty}(\pi)^{bd}$ as the global sections of a sheaf on $\mathcal{C}_0$. For an open subset $U\subseteq N_0$ we denote the image of $U$ by $x_U\subseteq \mathcal{C}_0$ under the above map $u\mapsto x_u$. Moreover,  we regard $\res$ as an $\End^{cont}_o(D^\vee_{\xi,\ell,\infty}(\pi))$-valued measure on $\mathcal{C}_0$, ie.\ a ring homomorphism $\res: C ^\infty(\mathcal{C}_0,o)\to \End^{cont}_o(D^\vee_{\xi,\ell,\infty}(\pi))$. We restrict $\res$ to a map $\res\colon C ^\infty(\mathcal{C}_0,o)\to \Hom^{cont}_o(D^\vee_{\xi,\ell,\infty}(\pi)^{bd},D^\vee_{\xi,\ell,\infty}(\pi))$.
Put $\mathcal{C}:=Nw_0B/B\supset \mathcal{C}_0$. By the discussion in section 5 of \cite{SVZ} in order to construct a $G$-equivariant sheaf on $G/B$ with the required properties we need to integrate the map
\begin{eqnarray*}
\alpha_{g}\colon \mathcal{C}_0&\to &\Hom_o^{cont}(D^\vee_{\xi,\ell,\infty}(\pi)^{bd},D^\vee_{\xi,\ell,\infty}(\pi))  \\
x_u&\mapsto& \alpha(g,u) \circ \res(1_{\alpha(g,u)^{-1}\mathcal{C}_0\cap\mathcal C_0})
\end{eqnarray*}
with respect to the measure $\res$ where for $x_u\in g^{-1}\mathcal{C}_0\cap\mathcal{C}_0\subset g^{-1}\mathcal{C}\cap \mathcal{C}$ we take $\alpha(g,u)$ to be the unique element in $B$ with the property
\begin{equation*}
guw_0N=\alpha(g,u)uw_0N\ .
\end{equation*}
Note that since $x_u$ lies in $g^{-1}\mathcal{C}_0\cap\mathcal{C}_0$ we also have $x_u\in \alpha(g,u)^{-1}\mathcal{C}_0\cap \mathcal{C}_0$ so the latter set is nonempty and open in $G/B$. Recall from section 6.1 in \cite{SVZ} that a map $F\colon \mathcal{C}_0\to  \Hom_o^{cont}(D^\vee_{\xi,\ell,\infty}(\pi)^{bd},D^\vee_{\xi,\ell,\infty}(\pi))$ is called integrable with respect to ($s$, $\res$, $\mathfrak{C}$)  if the limit
\begin{equation*}
    \int_{\mathcal{C}_0} Fd\res := \lim_{k \rightarrow \infty} \sum_{u \in J(N_0/s^kNs^{-k})} F(x_u) \circ \res(1_{x_{us^kNs^{-k}}}) 
\end{equation*}
exists in $ \Hom_o^{cont} (D^\vee_{\xi,\ell,\infty}(\pi)^{bd},D^\vee_{\xi,\ell,\infty}(\pi))$ and does not depend on the choice of the sets of representatives $J(N_0/s^kNs^{-k})$.

\begin{pro}
The map $\alpha_g$ is $(s,\res,\mathfrak{C})$-integrable for any $g\in G$.
\end{pro}
\begin{proof}
By Proposition 6.8 in \cite{SVZ} it suffices to show that $\mathfrak{C}$ satisfies:
\begin{enumerate}
\item[$\mathfrak{C}(5)$]   For any $C\in\mathfrak{C}$ the compact subset $\psi_{s}(C)\subseteq D^\vee_{\xi,\ell,\infty}(\pi)$ also lies in $\mathfrak{C}$.

\item[$\mathfrak{T}(1)$]   For any $C\in\mathfrak{C}$ such that $C=N_{0}C$, any open $o[N_0]$-submodule $\mathcal{D}$ of $D^\vee_{\xi,\ell,\infty}(\pi)$, and any
compact subset $C_+ \subseteq T_+$ there exists a compact open subgroup $B_{1}=B_1(C,\mathcal{D},C_+)
\subseteq B_0$ and an integer $k(C,\mathcal{D},C_+) \geq 0$ such that
\begin{equation*}
\varphi_{s}^{k }\circ(1-B_1)C_+ \psi_s^{k}(C) \subseteq \mathcal{D} \qquad\text{for any $k \geq k(C,\mathcal{D},C_+)$} \ .
\end{equation*}
Here the multiplication by $C_+$ is via the $\varphi$-action of $T_+$ on $D^\vee_{\xi,\ell,\infty}(\pi)$.
\end{enumerate}

The condition $\mathfrak{C}(5)$ is clearly satisfied as for any $D_c\in \mathfrak{C}_0$ we have $\psi_{s}(D_c)\in\mathfrak{C}_0$, as well. For the condition $\mathfrak{T}(1)$ choose a $C\in \mathfrak{C}$ with $C=N_0C$, a compact subset $C_+\subset T_+$, and an open $o[N_0]$-submodule $\mathcal{D}\subseteq D^\vee_{\xi,\ell,\infty}(\pi)$. As $D^\vee_{\xi,\ell,\infty}(\pi)$ is the topological projective limit $\varprojlim_{M\in\mathcal{M}(\pi^{H_0}),n\geq 0}M_n^\vee[1/X]$ we may assume without loss of generality that $\mathcal{D}$ is the preimage of a compact $\Lambda(N_0)$-submodule $D_{n}\leq M_n^\vee[1/X]$ with $D_{n}[1/X]=M_n^\vee[1/X]$ under the natural surjective map $f_{M,n}\colon D^\vee_{\xi,\ell,\infty}(\pi)\twoheadrightarrow M_n^\vee[1/X]$ for some $M\in \mathcal{M}(\pi^{H_0})$ and $n\geq 0$.
Moreover, since $B_0=T_0N_0$ is compact and normailizes $H_0$, the $T_0$-orbit of any element $m\in M\leq \pi^{H_0}$ is finite and contained in $\pi^{H_0}$. Therefore we also have $B_0M=T_0M\in \mathcal{M}(\pi^{H_0})$. So we may assume without loss of generality that $M$ is $B_0$-invariant whence we have an action of $B_0$ on $M_n^\vee[1/X]$. Choose a $D_c\in \mathfrak{C}_0$ with $C\subseteq D_c$. Since $D_c$ is $\psi_{s}$-invariant, we have $C_+\psi_{s}^k(C)\subseteq C_+\psi_{s}^k(D_c)\subseteq C_+D_c$. Moreover, $C_+D_c$ is compact as both $C_+$ and $D_c$ are compact, so $f_{M,n}(C_+\psi_{s}^k(C))\subset M_n^\vee[1/X]$ is bounded. In particular, we have a compact $\Lambda(N_0)$-submodule $D'$ of $M_n^\vee[1/X]$ containing $f_{M,n}(C_+\psi_{s}^k(C))$. So by the continuity of the action of $B_0$ on $M_n^\vee[1/X]$ there exists an open subgroup $B_1\leq B_0$ such that we have
\begin{align*}
(1-B_1)f_{M,n}(C_+\psi_{s}^k(C))\subset \Lambda(N_0/H_n)\otimes_{\Lambda(N_{\alpha,0})}(M^\vee[1/X]^{++})\leq\\
\leq \Lambda(N_0/H_n)\otimes_{\Lambda(N_{\alpha,0})}M^\vee[1/X]\cong M_n^\vee[1/X]
\end{align*}
for any $k\geq 0$. Here $M^\vee[1/X]^{++}$ denotes the treillis in $M^\vee[1/X]$ consisting of those elements $d\in M^\vee[1/X]$ such that $\varphi_s^n(d)\to 0$ in $M^\vee[1/X]$ as $n\to \infty$ (cf.\ section I.3.2 in \cite{C}). Finally, since $D_n$ is open and $M^\vee[1/X]^{++}$ is finitely generated over $\Lambda(N_{\alpha,0})\cong o\bs X\js$ there exists an integer $k_1\geq 0$ such that $\varphi_s^k(\Lambda(N_0/H_n)\otimes_{\Lambda(N_{\alpha,0})}(M^\vee[1/X]^{++}))$ is contained in $D_n$ for all $k\geq k_1$. In particular, we have
\begin{align*}
f_{M,n}(\varphi_{s}^{k }\circ(1-B_1)C_+ \psi_s^{k}(C))=\varphi_{s}^{k }\circ(1-B_1)(f_{M,n}(C_+ \psi_s^{k}(C)))\subseteq \\
\subseteq \varphi_{s}^{k }\circ(1-B_1)(M^\vee[1/X]^{++})\subseteq D_n
\end{align*}
showing that $\varphi_{s}^{k }\circ(1-B_1)C_+ \psi_s^{k}(C)$ is contained in $\mathcal{D}$.
\qed\end{proof}

For all $g\in G$ we denote by $\mathcal{H}_g\in\Hom_o^{cont}(D^\vee_{\xi,\ell,\infty}(\pi)^{bd},D^\vee_{\xi,\ell,\infty}(\pi))$ the integral 
\begin{equation*}
   \mathcal{H}_g:= \int_{\mathcal{C}_0} \alpha_g d\res = \lim_{k \rightarrow \infty} \sum_{u \in J(N_0/s^kNs^{-k})} \alpha_g(x_u)u\circ\varphi_{s}^k\circ\psi_{s}^k\circ u^{-1}
\end{equation*}
we have just proven to converge. We denote the $k$th term of the above sequence by 
\begin{equation}
\mathcal{H}_g^{(k)}=\mathcal{H}_{g,J(N_0/s^kN_0s^{-k})}:=\sum_{u \in J(N_0/s^kNs^{-k})} \alpha_g(x_u)u\circ\varphi_{s}^k\circ\psi_{s}^k\circ u^{-1}\ .\label{Hgkdef}
\end{equation}

Our main result in this section is the following

\begin{pro}\label{sheafGB}
The image of the map $\mathcal{H}_g\colon D^\vee_{\xi,\ell,\infty}(\pi)^{bd}\to D^\vee_{\xi,\ell,\infty}(\pi)$ is contained in $D^\vee_{\xi,\ell,\infty}(\pi)^{bd}$. There exists a $G$-equivariant sheaf $\mathfrak{Y}=\mathfrak{Y}_{\alpha,\pi}$ on $G/B$ with sections $\mathfrak{Y}(\mathcal{C}_0)$ on $\mathcal{C}_0$ isomorphic $B_+$-equivariantly to $D^\vee_{\xi,\ell,\infty}(\pi)^{bd}$ such that we have $\mathcal{H}_g=\res^{G/B}_{\mathcal{C}_0}\circ(g\cdot)\circ\res^{G/B}_{\mathcal{C}_0}$ as maps on $D^\vee_{\xi,\ell,\infty}(\pi)^{bd}=\mathfrak{Y}(\mathcal{C}_0)$.
\end{pro}
\begin{proof}
By Prop.\ 5.14 and 6.9 in \cite{SVZ} it suffices to check the following conditions:
\begin{enumerate}
\item[$\mathfrak{C}(6)$] For any $C\in\mathfrak{C}$ the compact subset
  $ \varphi_{s}(C)\subseteq M$ also lies in $\mathfrak{C}$.
\item[$\mathfrak{T}(2)$] Given a  set $J(N_0/s^kN_0s^{-k})\subset N_{0}$ of representatives for all $k\geq 1$, for any $x\in D^\vee_{\xi,\ell,\infty}(\pi)^{bd}$ and $g\in G$ there exists a compact $\psi_{s}$-invariant $\Lambda(N_0)$-submodule $D_{x,g}\in\mathfrak{C}$ and a positive integer $k_{x,g}$ such that
  $\mathcal{H}_g^{(k)}(x)\subseteq D_{x,g}$ for any $k\geq  k_{x,g}$.
 \end{enumerate}
The condition $\mathfrak{C}(6)$ follows from (the proof of) Prop.\ \ref{bdetale} as for $C\subseteq D_c\in\mathfrak{C}_0$ we have $\varphi_{s}(C)\subseteq \varphi_{s}(D_c)\subseteq i(\varphi_s^*D_c)\in\mathfrak{C}_0$.

The proof of $\mathfrak{T}(2)$ is very similar to the proof of Corollary 9.15 in \cite{SVZ}. However, it is not a direct consequence of that as $D^\vee_{\xi,\ell,\infty}(\pi)$ is not necessarily finitely generated over $\Lambda_\ell(N_0)$, so we recall the details. For any $x$ in $D^\vee_{\xi,\ell,\infty}(\pi)^{bd}$, the element $\mathcal{H}_g^{(k)}(x)$ also lies in $D^\vee_{\xi,\ell,\infty}(\pi)^{bd}$ for any fixed $k$ since the set of bounded elements form an \'etale $T_+$-submodule (by axiom $\mathfrak{C}(4)$) whence they are closed under the operations ($\varphi$-, $\psi$-, and $N_0$-actions) defining the map $\mathcal{H}_g^{(k)}$. So by axiom $\mathfrak{C}(2)$ we only need to show that for $k$ large enough the difference $$s_g^{(k)}(x):=\mathcal{H}_g^{(k)}(x)-\mathcal{H}_{g}^{(k+1)}(x)$$ lies in a compact submodule $D_{x,g}\leq D^\vee_{\xi,\ell,\infty}(\pi)^{bd}$ in $\mathfrak{C}_0$ independent of $k$. In order to do so we proceed in four steps. In steps $1$, $2$, and $3$ the goal is to show that for a fixed choice $M\in \mathcal{M}(\pi^{H_0})$ the image of $s_g^{(k)}(x)$ lies in a compact $\psi$-invariant $\Lambda(N_0)$-submodule of $M_\infty^\vee[1/X]$ under the projection map $D^\vee_{\xi,\ell,\infty}(\pi)\twoheadrightarrow M_\infty^\vee[1/X]$ for $k$ large enough \emph{not depending} on $M$. This compact submodule of $M_\infty^\vee[1/X]$ will be of the form
$$\{m\in M^\vee_\infty[1/X]\mid \ell_M(\psi_s^r(u^{-1}m))\text{ is in }\\ D_0\text{ for all }r\geq 0,u\in N_0\}$$
for some treillis $D_0\subset M^\vee[1/X]$ where $\ell_M\colon M_\infty^\vee[1/X]\to M^\vee[1/X]$ is the natural projection map. Step $1$ is devoted to showing this for smaller $r$ (compared to $k$) with some choice of a treillis and in Step $2$ we take care of all larger $r$ (using a different treillis in $M^\vee[1/X]$). In both of these steps $k\geq k(M)$ is large enough \emph{depending} on $M$. In Step $3$ we eliminate this dependence on $M$ of the lower bound for $k$ by choosing a third treillis so that the sum $D_0$ of these three different choices of a treillis will do. In Step $4$ we take the projective limit of these compact sets for all possible choices of $M$ to obtain a compact subset of $D^\vee_{\xi,\ell,\infty}(\pi)$.

\emph{Step 1.} Equation $(43)$ in \cite{SVZ} shows that for any compact open subgroup $B_1\leq B_0$ there exist integers $0\leq k_g^{(1)}\leq k_g^{(2)}(B_1)$ and a compact subset $\Lambda_g\subset T_+$ such that for $k\geq k_g^{(2)}(B_1)$ we have
\begin{equation}\label{sgk}
 s_{g}^{(k)} \in \  \langle N_0 s^{k-k_{g}^{(1)}} (1- B_1) \Lambda_g s \psi_{s}^{k+1} N_0\rangle_{o}\ ,
\end{equation}
where we denote by $\langle\cdot\rangle_o$ the generated $o$-submodule. Here $k_g^{(1)}$ is chosen so that $\{\alpha(g,u)us^{k_g^{(1)}}\mid x_u\in g^{-1}\mathcal{C}_0\cap\mathcal{C}_0\}$ is contained in $B_+=N_0T_+$. There exists such an integer $k_g^{(1)}$ since $\{\alpha(g,u)u\mid x_u\in g^{-1}\mathcal{C}_0\cap\mathcal{C}_0\}$ is a compact subset in $N_0T$.
Choose a compact $\psi_{s}$-invariant $\Lambda(N_0)$-submodule $D_c\in \mathfrak{C}_0$ containing the element $x\in D^\vee_{\xi,\ell,\infty}(\pi)^{bd}$ and pick an $M$ in $\mathcal{M}(\pi^{H_0})$. Applying $\mathfrak{T}(1)$ in the situation $C=D_c$, $C_+=\Lambda_gs$, and $\mathcal{D}=f_{M,0}^{-1}(M^\vee[1/X]^{++})$ we find an integer $k_1\geq 0$ and a compact open subgroup $B_1\leq B_0$ such that $\varphi_s^k\circ(1-B_1)\Lambda_gsD_c\subseteq \mathcal{D}$ for all $k\geq k_1$. Noting that $D_c$ is $\psi_{s}$-stable and $\mathcal{D}$ is a $\Lambda(N_0)$-submodule we obtain $s_g^{(k)}(D_c)\subseteq N_0\varphi_s^r(\mathcal{D})$ for $k\geq r+k_1+k_g^{(2)}(B_1)$. Applying $\psi_s^r$ to this using \eqref{sgk} and putting $k_g(M):=k_1+k_g^{(2)}(B_1)$ we deduce 
\begin{equation}\label{smallr}
\psi_s^r(\Lambda(N_0)s_g^{(k)}(D_c))\subseteq \mathcal{D}\qquad\text{for all $k\geq k_g(M)$ and $r\leq k-k_g(M)$ .} 
\end{equation}
Note that the subgroup $B_1$ depends on $M$ therefore so do $k_g^{(2)}(B_1)$ and $k_g(M)$, but not $k_g^{(1)}$. 

\emph{Step 2.} We are going to find another treillis $D_1\leq M^\vee[1/X]$ such that for all $k\geq k_g(M)$ and $r\geq k-k_g(M)$ we have 
\begin{equation}\label{bigr}
\psi_s^r(\Lambda(N_0)\mathcal{H}_g^{(k)}(D_c))\subseteq \mathcal{D}_1:=f_{M,0}^{-1}(D_1)\ .
\end{equation}
For $x_u\in g^{-1}\mathcal{C}_0\cap\mathcal{C}_0$ write $\alpha(g,u)u$ in the form $\alpha(g,u)u=n(g,u)t(g,u)$ with $n(g,u)\in N_0$ and $t(g,u)\in T$. Since $g^{-1}\mathcal{C}_0\cap \mathcal{C}_0$ is compact, $t(g,\cdot)$ is continuous, and $k_g(M)\geq k_g^{(1)}$ the set $C'_+:=\{t(g,u)s^{k_g(M)}\mid x_u\in g^{-1}\mathcal{C}_0\cap\mathcal{C}_0\}\subset T$ is compact and contained in $T_+$. So we compute
\begin{align*}
\psi_s^r(\Lambda(N_0)\mathcal{H}_g^{(k)}(D_c))=\\
=\psi_s^r(\Lambda(N_0)\sum_{u\in J(N_0/s^kN_0s^{-k})}n(g,u)\varphi_{t(g,u)s^k}\circ\psi_s^k(u^{-1}D_c))\subseteq \\
\subseteq \psi_s^r(\Lambda(N_0)\varphi_{s}^{k-k_g(M)}\circ\varphi_{t(g,u)s^{k_g(M)}}(D_c))\subseteq \psi_{s}^{r-k+k_g(M)}(\Lambda(N_0)C'_+(D_c))\ .
\end{align*}
Since $C'_+\subset T_+$ is compact , there exists an integer $k(C'_+)$ such that $s^kt^{-1}$ lies in $T_+$ for all $t\in C'_+$. So we have $C'_+(D_c)\subseteq i(\varphi_{s^{k(C'_+)}}^*D^\vee_{\xi,\ell,\infty}(\pi)^{bd})\in\mathfrak{C}_0$ showing that $$D_1:=f_{M,0}(i(\varphi_{s^{k(C'_+)}}^*D^\vee_{\xi,\ell,\infty}(\pi)^{bd}))$$ is a good choice as $i(\varphi_{s^{k(C'_+)}}^*D^\vee_{\xi,\ell,\infty}(\pi)^{bd})$ is a $\psi_s$-stable $\Lambda(N_0)$ submodule. 

\emph{Step 3.} For each fixed $k\geq k_g^{(1)}$ there exists a compact $\psi_{s}$-invariant $\Lambda(N_0)$-submodule $D_{c,k}\in\mathfrak{C}_0$ containing $\mathcal{H}_g^{(k)}(D_c)$. In particular, we may choose a treillis $D_2\leq M^\vee[1/X]$ containing
\begin{equation*}
f_{M,0}(\psi_{s}^r(\Lambda(N_0)\mathcal{H}_g^{(k)}(D_c)))
\end{equation*}
for all $k_g^{(1)}\leq k\leq k_g(M)$ and $r\geq 0$. Putting $\mathcal{D}_2:=f_{M,0}^{-1}(D_2)$ and combining this with \eqref{smallr} and \eqref{bigr} we obtain
\begin{equation}\label{D12}
\psi_{s}^r(\Lambda(N_0)\mathcal{H}_g^{(k)}(D_c))\subseteq \mathcal{D}+\mathcal{D}_1+\mathcal{D}_2
\end{equation}
for all $k\geq k_{x,g}:=k_g^{(1)}$ and $r\geq 0$. Denote by $f_{M,\infty}$ the natural surjective map $f_{M,\infty}\colon D^\vee_{\xi,\ell,\infty}\to M^\vee_\infty[1/X]$. Note that $f_{M,0}$ factors through $f_{M,\infty}$. The equation \eqref{D12} implies (in fact, is equivalent to) that 
\begin{equation*}
f_{M,\infty}\left(\bigcup_{k\geq k_{x,g}}\mathcal{H}_g^{(k)}(D_c)\right)\subseteq M^\vee_\infty[1/X]^{bd}(D_0)
\end{equation*}
where 
\begin{align*}
M^\vee_\infty[1/X]^{bd}(D_0)=\{m\in M^\vee_\infty[1/X]\mid \ell_M(\psi_s^r(u^{-1}m))\text{ is in }\\ D_0:=M^\vee[1/X]^{++}+D_1+D_2\text{ for all }r\geq 0,u\in N_0\}
\end{align*}
is a compact $\psi_s$-invariant $\Lambda(N_0)$-submodule in $M^\vee_\infty[1/X]$ (Prop.\ 9.10 in \cite{SVZ}). 

\emph{Step 4.} We put $D_{x,g}(M):=\bigcap \mathfrak{D}$ where $\mathfrak{D}$ runs through all the $\psi_s$-invariant compact $\Lambda(N_0)$-submodules of $M^\vee_\infty[1/X]$ containing $f_{M,\infty}(\bigcup_{k\geq k_{x,g}}\mathcal{H}_g^{(k)}(D_c))$. Therefore 
\begin{equation*}
D_{x,g}:=\varprojlim_{M\in\mathcal{M}(\pi^{H_0})}D_{x,g}(M)
\end{equation*}
is a $\psi_{s}$-invariant compact $\Lambda(N_0)$-submodule of $D^\vee_{\xi,\ell,\infty}(\pi)$ (ie.\ we have $D_{x,g}\in\mathfrak{C}_0$) containing $\bigcup_{k\geq k_{x,g}}\mathcal{H}_g^{(k)}(D_c)$.
\qed\end{proof}

We end this section by putting a natural topology (called the weak topology) on the global sections $\mathfrak{Y}(G/B)$ that will be needed in the next section. At first we equip $D^\vee_{\xi,\ell,\infty}(\pi)^{bd}$ with the inductive limit topology of the compact topologies of each $D_c\in\mathfrak{C}_0$. This makes sense as the inclusion maps $D_c\hookrightarrow D'_c$ for $D_c\subseteq D'_c\in\mathfrak{C}_0$ are continuous as these compact topologies are obtained as the subspace topologies in the weak topology of $D^\vee_{\xi,\ell,\infty}(\pi)$. We call this topology the weak topology on $D^\vee_{\xi,\ell,\infty}(\pi)^{bd}$.

\begin{lem}\label{Hgcont}
The operators $\mathcal{H}_g$ and $\res_{\mathcal{U}}$ on $D^\vee_{\xi,\ell,\infty}(\pi)^{bd}$ are continuous in the weak topology of $D^\vee_{\xi,\ell,\infty}(\pi)^{bd}$ for all $g\in G$ and $\mathcal{U}\subseteq N_0$ compact open. In particular, $D^\vee_{\xi,\ell,\infty}(\pi)^{bd}$ is the topological direct sum of $\res_{\mathcal{U}}(D^\vee_{\xi,\ell,\infty}(\pi)^{bd})$ and $\res_{N_0\setminus\mathcal{U}}(D^\vee_{\xi,\ell,\infty}(\pi)^{bd})$.
\end{lem}
\begin{proof}
By the property $\mathfrak{T}(2)$ the restriction of $\mathcal{H}_g^{(k)}$ to a compact subset $D_c$ in $\mathfrak{C}_0$ has image in a compact set $D_{c,g}\in\mathfrak{C}_0$ for all large enough $k$. Moreover, each $\mathcal{H}_g^{(k)}$ is continuous by Lemma \ref{B+cont}. On the other hand, the limit $\mathcal{H}_g=\lim_{k\to\infty}\mathcal{H}_g^{(k)}$ is uniform on each compact subset $D_c\in\mathfrak{C}_0$ by Proposition 6.3 in \cite{SVZ}, so the limit $\mathcal{H}_g\colon D_c\to D_{c,g}$ is also continuous. Taking the inductive limit on both sides we deduce that $\mathcal{H}_g\colon D^\vee_{\xi,\ell,\infty}(\pi)\to D^\vee_{\xi,\ell,\infty}(\pi)$ is also continuous. The continuity of $\res_{\mathcal{U}}$ follows in a similar but easier way.
\qed\end{proof}

So far we have put a topology on $D^\vee_{\xi,\ell,\infty}(\pi)^{bd}=\mathfrak{Y}(\mathcal{C}_0)$. The multiplication by an element $g\in G$ gives an $o$-linear bijection $g\colon \mathfrak{Y}(\mathcal{C}_0)\to \mathfrak{Y}(g\mathcal{C}_0)$. We define the weak topology on $\mathfrak{Y}(g\mathcal{C}_0)$ so that this is a homeomorphism. Now we equip $\mathfrak{Y}(G/B)$ with the coarsest topology such that the restriction maps $\res^{G/B}_{g\mathcal{C}_0}\colon \mathfrak{Y}(G/B)\to\mathfrak{Y}(g\mathcal{C}_0)$ are continuous for all $g\in G$. We call this the weak topology on $\mathfrak{Y}(G/B)$ making $\mathfrak{Y}(G/B)$ a linear-topological $o$-module.

\begin{lem}
\begin{enumerate}[$a)$]
\item The multiplication by $g$ on $\mathfrak{Y}(G/B)$ is continuous (in fact a homeomorphism) for each $g\in G$. 
\item The weak topology on $\mathfrak{Y}(G/B)$ is Hausdorff.
\end{enumerate}
\end{lem}
\begin{proof}
For $a)$ we need to check that the composite of the function $$(g\cdot)_{G/B}\colon \mathfrak{Y}(G/B)\to \mathfrak{Y}(G/B)$$ with the projections $\res_{h\mathcal{C}_0}^{G/B}$ is continuous for all $h\in G$. However, $\res_{h\mathcal{C}_0}^{G/B}\circ (g\cdot)_{G/B}=(g\cdot)_{g^{-1}h\mathcal{C}_0}\circ \res_{g^{-1}h\mathcal{C}_0}^{G/B}$ is the composite of two continuous maps hence also continuous.

For $b)$ note that the weak topology on $D^\vee_{\xi,\ell,\infty}(\pi)^{bd}$ is finer than the subspace topology inherited from $D^\vee_{\xi,\ell,\infty}(\pi)$ therefore it is Hausdorff. To see this we need to show that the inclusion $D^\vee_{\xi,\ell,\infty}(\pi)^{bd}\hookrightarrow D^\vee_{\xi,\ell,\infty}(\pi)$ is continuous. As the weak topology on $D^\vee_{\xi,\ell,\infty}(\pi)^{bd}$ is defined as a direct limit, it suffices to check this on the defining compact sets $D_c\in\mathfrak{C}_0$. However, on these compact sets the inclusion map is even a homeomorphism by definition.

So the topology on $\mathfrak{Y}(G/B)$ is also Hausdorff as for any two different global sections $x\neq y\in \mathfrak{Y}(G/B)$ there exists an element $g\in G$ such that $\res_{g\mathcal{C}_0}^{G/B}(x)\neq \res_{g\mathcal{C}_0}^{G/B}(y)$.
\qed\end{proof}

\subsection{A $G$-equivariant map $\pi^\vee\to \mathfrak{Y}(G/B)$}\label{Gmap}

Here we generalize Thm.\ IV.4.7 in \cite{C} to $\mathbb{Q}_p$-split reductive groups $G$ over $\mathbb{Q}_p$ with connected centre. Assume in this section that $\pi$ is an \emph{admissible} smooth $o/\varpi^h$-representation of $G$ of \emph{finite length}.

By Corollary \ref{prbd} we have the composite maps
\begin{equation*}
\beta_{g\mathcal{C}_0}\colon \pi^\vee\overset{g^{-1}\cdot}{\longrightarrow} \pi^\vee\overset{\pr_{SV}}{\longrightarrow} D_{SV}(\pi) \overset{\pr}{\longrightarrow}D^\vee_{\xi,\ell,\infty}(\pi)^{bd}\overset{\sim}{\longrightarrow}\mathfrak{Y}(\mathcal{C}_0)\overset{g\cdot}{\longrightarrow} \mathfrak{Y}(g\mathcal{C}_0)
\end{equation*}
for each $g\in G$. By definition we have $\beta_{g\mathcal{C}_0}(\mu)=g\beta_{\mathcal{C}_0}(g^{-1}\mu)$ for all $\mu\in\pi^\vee$ and $g\in G$. Our goal is to show that these maps glue together to a $G$-equivariant map $\beta_{G/B}\colon \pi^\vee\to \mathfrak{Y}(G/B)$.

Let $n_0=n_0(G)\in\mathbb{N}$ be the maximum of the degrees of the algebraic characters $\beta\circ\xi\colon \mathbb{G}_m\to\mathbb{G}_m$ for all $\beta$ in $\Phi^+$ and put $U^{(k)}:=\Ker (G_0\twoheadrightarrow G(\Zp/p^k\Zp))$ where $G_0=\mathbf{G}(\Zp)$.

\begin{lem}\label{smallt}
For any fixed $r_0\geq 1$ we have $t^{-1}U^{(k)}t\leq U^{(k-n_0r_0)}$ for all $t\leq s^{r_0}$ in $T_+$ and $k\geq r_0n_0$.
\end{lem}
\begin{proof}
The condition $t\leq s^{r_0}$ implies that $v_p(\beta(t))\leq v_p(\beta(s^{r_0}))=v_p(\beta\circ\xi(p^{r_0}))\leq r_0n_0$ for all $\beta\in \Phi^+$ by the maximality of $n_0$. On the other hand, by the Iwahori factorization we have $U^{(k)}=(U^{(k)}\cap\overline{N})(U^{(k)}\cap T)(U^{(k)}\cap N)$. Since $t$ is in $T_+$ we deduce
\begin{eqnarray*}
t^{-1}(U^{(k)}\cap\overline{N})t\leq (U^{(k)}\cap\overline{N})&\leq& (U^{(k-r_0n_0)}\cap\overline{N})\\
t^{-1}(U^{(k)}\cap T)t= (U^{(k)}\cap T)&\leq& (U^{(k-r_0n_0)}\cap T)\\
t^{-1}(U^{(k)}\cap N)t=&&\\
\prod_{\beta\in \Phi^+}t^{-1}(U^{(k)}\cap N_\beta)t &\leq&  \prod_{\beta\in \Phi^+}(U^{(k-r_0n_0)}\cap N_\beta)\\
&&= (U^{(k-r_0n_0)}\cap N)\ .\ \qed
\end{eqnarray*}
\end{proof}

\begin{lem}\label{W0}
Assume that $\pi$ is an admissible representation of $G$ of finite length. Then there exists a finitely generated $o$-submodule $W_0\leq \pi$ such that $\pi=BW_0$.
\end{lem}
\begin{proof}
Since $\pi$ has finite length, by induction we may assume it is irreducible (hence killed by $\varpi$). In this case we may take $W_0=\pi^{U^{(1)}}$ which is $G_0$-stable as $U^{(1)}$ is normal in $G_0$. It is nonzero since $\pi$ is smooth, and finitely generated over $o$ as $\pi$ is admissible. By the Iwasawa decomposition we have $\pi=GW_0=BG_0W_0=BW_0$.
\qed\end{proof}

Let $W_0$ be as in Lemma \ref{W0} and put $W:=B_+W_0$, $W_r:=\bigcup_{t\leq s^r}N_0tW_0$ so we have 
\begin{equation}\label{Wr}
W=\varinjlim_r W_r=\bigcup_{r\geq 0}W_r
\end{equation}
where $W_r$ is finitely generated over $o$ for all $r\geq 0$. By construction $W$ is a generating $B_+$-subrepresentation of $\pi$. So the map $\pr_{SV}$ factors through the natural projection map $\pr_W\colon \pi^\vee\to W^\vee$. Here the Pontryagin dual $W^\vee$ is a compact $\Lambda(N_0)$-module with a $\psi$-action of $T_+$ coming from the multiplication by $T_+$ on $W$. By Proposition \ref{etalehull} we may form the \'etale hull $\widetilde{W^\vee}$ of $W^\vee$ which is an \'etale $T_+$-module over $\Lambda(N_0)$. Since $D^\vee_{\xi,\ell,\infty}(\pi)$ is an \'etale $T_+$-module over $\Lambda(N_0)$ and the composite map $W^\vee\to D_{SV}(\pi)\to D^\vee_{\xi,\ell,\infty}(\pi)$ is $\psi$-equivariant, it factors through $\widetilde{W^\vee}$. All in all we have factored the map $\pr\circ\pr_{SV}$ as
\begin{equation*}
\pr\circ\pr_{SV}\colon \pi^\vee\overset{\widetilde{\pr_{W}}}{\longrightarrow}\widetilde{W^\vee}\overset{\pr^{\widetilde{W^\vee}}_{D}}{\longrightarrow}D^\vee_{\xi,\ell,\infty}(\pi)\ .
\end{equation*}
The advantage of considering $\widetilde{W^\vee}$ is that the operators $\mathcal{H}_g^{(k)}$ make sense as maps $\widetilde{W^\vee}\to\widetilde{W^\vee}$ and the map $\widetilde{W^\vee}\to D^\vee_{\xi,\ell,\infty}(\pi)$ is $\mathcal{H}_g^{(k)}$-equivariant as it is a morphism of \'etale $T_+$-modules over $\Lambda(N_0)$. More precisely, let $g$ be in $G$ and put $\mathcal{U}_g:=\{u\in N_0\mid x_u\in g^{-1}\mathcal{C}_0\cap\mathcal{C}_0\}$, $\mathcal{U}_g^{(k)}:=J(N_0/s^kN_0s^{-k})\cap \mathcal{U}_g$. For any $u\in \mathcal{U}_g$ we write $gu$ in the form $gu=n(g,u)t(g,u)\overline{n}(g,u)$ for some unique $n(g,u)\in N_0$, $t(g,u)\in T$, $\overline{n}(g,u)\in \overline{N}$.

\begin{lem}\label{k0g}
There exists an integer $k_0=k_0(g)$ such that for all $k\geq k_0$ and $u\in \mathcal{U}_g$ we have $us^kN_0s^{-k}\subseteq \mathcal{U}_g$, $s^kt(g,u)\in T_+$, and $s^{-k}\overline{n}(g,u)s^k\in \overline{N}_0=G_0\cap \overline{N}$. In particular, for any set $J(N_0/s^kN_0s^{-k})$ of representatives of the cosets in $N_0/s^kN_0s^{-k}$ we have $\mathcal{U}_g=\bigcup_{u\in \mathcal{U}_g^{(k)}}us^kN_0s^{-k}$.
\end{lem}
\begin{proof}
Since $\mathcal{U}_g$ is compact and open in $N_0$, it is a union of finitely many cosets of the form $us^kN_0s^{-k}$ for $k$ large enough. Moreover, the maps $t(g,\cdot)$ and $\overline{n}(g,\cdot)$ are continuous in the $p$-adic topology. So the image of $t(g,\cdot)$ is contained in finitely many cosets of $T/T_0$ as $T_0$ is open. For the statement regarding $\overline{n}(g,u)$ note that we have $\overline{N}=\bigcup_{k\geq 0}s^k\overline{N}_0s^{-k}$.
\qed\end{proof}

For $k\geq k_0=k_0(g)$ let $J(N_0/s^kN_0s^{-k})\subset N_0$ be an arbitrary set of representatives of $N_0/s^kN_0s^{-k}$. Recall from the proof of Prop.\ \ref{sheafGB} Step $2$ (see also \cite{SVZ}) that for fixed $g\in G$ and all $u\in N_0$ we may write $\alpha_g(x_u)u$ in the form $n(g,u)t(g,U)$ for some $n(g,u)\in N_0$ and $t(g,U)\in s^{-k_0}T_+$. In particular the equation \eqref{Hgkdef} defining $\mathcal{H}_g^{(k)}$ reads
\begin{align*}
\mathcal{H}_g^{(k)}=\mathcal{H}_{g,J(N_0/s^kN_0s^{-k})}:=\sum_{u\in \mathcal{U}_g^{(k)}}n(g,u)\varphi_{t(g,u)s^k}\circ\psi_s^k\circ(u^{-1}\cdot)
\end{align*}
where $t(g,u)s^k$ lies in $T_+$. Further, any open compact subset $\mathcal{U}\subseteq N_0$ is the disjoint union of cosets of the form $us^kN_0s^{-k}$ for $k\geq k'(\mathcal{U})$ large enough. For a fixed $k\geq k'(\mathcal{U})$ we put
\begin{align*}
\res_{\mathcal{U}}:=\sum_{u\in J(N_0/s^kN_0s^{-k})\cap \mathcal{U}}u\varphi_{s^k}\circ\psi_s^k\circ(u^{-1}\cdot)\ .
\end{align*}
The operators $\mathcal{H}_g^{(k)}$ and $\res_{\mathcal{U}}$ make sense in any \'etale $T_+$-module over $\Lambda(N_0)$, in particular also in $\widetilde{W^\vee}$ and $D^\vee_{\xi,\ell,\infty}(\pi)$. Moreover, $\res_{\mathcal{U}}$ is independent of the choice of $k\geq k'(\mathcal{U})$. Further, any morphism between \'etale $T_+$-modules over $\Lambda(N_0)$ is $\mathcal{H}_g^{(k)}$- and $\res_{\mathcal{U}}$-equivariant.

\begin{lem}\label{ngubij}
Let $g$ be in $G$, $u$ be in $\mathcal{U}_g$, and $k\geq k_0+1$ be an integer. Then the map 
\begin{equation}\label{uskN0}
n(g,\cdot)\colon us^kN_0s^{-k}\to n(g,u)t(g,u)s^kN_0s^{-k}t(g,u)^{-1} 
\end{equation}
is a bijection. In particular, for any set $J(N_0/s^kN_0s^{-k})$ of representatives of the cosets in $N_0/s^kN_0s^{-k}$ the set $\mathcal{U}_{g^{-1}}$ is the disjoint union of the cosets $n(g,u)t(g,u)s^kN_0s^{-k}t(g,u)^{-1}$ for $u\in \mathcal{U}_g^{(k)}$.
\end{lem}
\begin{proof}
By our assumption $k\geq k_0+1$, $s^{-k}\overline{n}(g,u)s^k$ lies in $s^{-1}\overline{N}_0s\subseteq U^{(1)}$. So for any $v\in N_0$ we have $s^{-k}\overline{n}(g,u)s^kv=vv_1$ for some $v_1$ in $v^{-1}U^{(1)}v=U^{(1)}$. Further, by the Iwahori factorization we have $U^{(1)}=(N\cap U^{(1)})(T\cap U^{(1)})(\overline{N}\cap U^{(1)})$. So we obtain that $s^{-k}\overline{n}(g,u)s^kvw_0B\subset \mathcal{C}_0$ for all $v\in N_0$, whence we deduce $s^{-k}\overline{n}(g,u)s^k\mathcal{C}_0\subseteq \mathcal{C}_0$. Similarly we have $s^{-k}\overline{n}(g,u)^{-1}s^k\mathcal{C}_0\subseteq \mathcal{C}_0$ showing that in fact $s^{-k}\overline{n}(g,u)s^k\mathcal{C}_0= \mathcal{C}_0$. We compute
\begin{align*}
g(us^kN_0s^{-k})w_0B=gus^kN_0w_0B=n(g,u)t(g,u)s^k(s^{-k}\overline{n}(g,u)s^k)\mathcal{C}_0=\\
=n(g,u)t(g,u)s^k\mathcal{C}_0=n(g,u)(t(g,u)s^kN_0s^{-k}t(g,u)^{-1})w_0B\ .
\end{align*}
Since the map $n(g,\cdot)$ is induced by the multiplication by $g$ on $g^{-1}\mathcal{C}_0\cap \mathcal{C}_0$ (identified with $\mathcal{U}_g$), we deduce that the map \eqref{uskN0} is a bijection. The second statement follows as $n(g,\cdot)\colon \mathcal{U}_g\to\mathcal{U}_{g^{-1}}$ is a bijection and we have a partition of $\mathcal{U}_g$ into cosets $us^kN_0s^{-k}$ for $u\in \mathcal{U}_g^{(k)}$ by Lemma \ref{k0g}.
\qed\end{proof}

\begin{lem}\label{WrzeroM++}
Let $M$ be arbitrary in $\mathcal{M}(\pi^{H_0})$ and $l,l'\geq 0$ be integers. There exists an integer $k_1=k_1(M,W_0,l,l')\geq 0$ such that for all $r\geq k_1$ the image of the natural composite map
\begin{equation*}
(W/W_r)^\vee\hookrightarrow W^\vee\to D^\vee_{\xi,\ell,\infty}(\pi)\overset{f_{M,l}}{\twoheadrightarrow} M_l^\vee[1/X]
\end{equation*}
lies in $\Lambda(N_0/H_l)\otimes_{u_\alpha}X^{l'}M^\vee[1/X]^{++}\subset \Lambda(N_0/H_l)\otimes_{u_\alpha}M^\vee[1/X]\cong M_l^\vee[1/X]$. Here $M^\vee[1/X]^{++}$ denotes the $o/\varpi^h\bs X\js$-submodule of $M^\vee[1/X]$ consisting of elements $d\in M^\vee[1/X]$ with $\varphi_{s}^n(d)\to 0$ as $n\to\infty$.
\end{lem}
\begin{proof}
By \eqref{Wr} the $\Lambda(N_0)$-submodules $(W/W_r)^\vee$ form a system of neighbourhoods of $0$ in $W^\vee$. On the other hand, $X^{l'}M^\vee[1/X]^{++}$ being a treillis in $M^\vee[1/X]$ (Prop.\ II.2.2 in \cite{Mira}), $\Lambda(N_0/H_l)\otimes_{u_\alpha}X^{l'}M^\vee[1/X]^{++}$ is open in the weak topology of $M_l^\vee[1/X]$. Therefore its preimage in $W^\vee$ contains $(W/W_r)^\vee$ for $r$ large enough.
\qed\end{proof}

Since $t(g,\cdot)$ is continuous and $\mathcal{U}_g$ is compact, there exists an integer $c\geq 0$ such that for all $u\in \mathcal{U}_g$ there is an element $t'(g,u)\in T_+$ such that $t(g,u)s^{k_0}t'(g,u)=s^c$.
\begin{lem}\label{finFt'}
For any fixed $M\in\mathcal{M}(\pi^{H_0})$ there are finitely many different values of $F_{t'(g,u)}^*M$ where $g\in G$ is fixed and $u$ runs on $\mathcal{U}_g$.
\end{lem}
\begin{proof}
By Lemma \ref{MT'} there exists an open subgroup $T'\leq T$ acting on $M$. In particular, $F_{t'(g,u)}^*M$ only depends on the coset $t'(g,u)T'$. Now $t'(g,\cdot)=s^{c-k_0}t(g,\cdot)^{-1}$ is continuous and $\mathcal{U}_g$ is compact therefore there are only finitely many cosets of the form $t'(g,u)T'$.
\qed\end{proof}

Our key proposition is the following:

\begin{pro}\label{betag}
For all $g\in G$ we have $\res_{g\mathcal{C}_0\cap\mathcal{C}_0}^{\mathcal{C}_0}\circ\beta_{\mathcal{C}_0}=\res_{g\mathcal{C}_0\cap\mathcal{C}_0}^{g\mathcal{C}_0}\circ\beta_{g\mathcal{C}_0}$.
\end{pro}
\begin{proof}
Note that since $G/B$ is totally disconnected in the $p$-adic topology, in particular $g\mathcal{C}_0\cap\mathcal{C}_0$ is both open and closed in $\mathcal{C}_0$, we have $\mathfrak{Y}(\mathcal{C}_0)=\mathfrak{Y}(g\mathcal{C}_0\cap\mathcal{C}_0)\oplus \mathfrak{Y}(\mathcal{C}_0\setminus g\mathcal{C}_0)$. By Prop.\ \ref{sheafGB} $\mathcal{H}_g$ is the composite map $$D^\vee_{\xi,\ell,\infty}(\pi)^{bd}=\mathfrak{Y}(\mathcal{C}_0)\overset{g\cdot}{\to}\mathfrak{Y}(g\mathcal{C}_0)\overset{\res^{g\mathcal{C}_0}_{g\mathcal{C}_0\cap \mathcal{C}_0}}{\twoheadrightarrow} \mathfrak{Y}(g\mathcal{C}_0\cap\mathcal{C}_0)\hookrightarrow \mathfrak{Y}(\mathcal{C}_0)=D^\vee_{\xi,\ell,\infty}(\pi)^{bd}\ ,$$ ie.\ we obtain $\res_{g\mathcal{C}_0\cap\mathcal{C}_0}^{g\mathcal{C}_0}\circ(g\cdot)=\mathcal{H}_g$ as maps on $\mathfrak{Y}(\mathcal{C}_0)$ once we identify $\mathfrak{Y}(g\mathcal{C}_0\cap \mathcal{C}_0)$ with a subspace in $\mathfrak{Y}(\mathcal{C}_0)$ via the above direct sum decomposition. On the other hand, by definition $\beta_{\mathcal{C}_0}=\pr\circ\pr_{SV}\colon \pi^\vee\to D^\vee_{\xi,\ell,\infty}(\pi)^{bd}\subset D^\vee_{\xi,\ell,\infty}(\pi)$ is the natural map and we have $\beta_{g\mathcal{C}_0}(\mu)=g\beta_{\mathcal{C}_0}(g^{-1}\mu)$ for any $g\in G$ and $\mu\in \pi^\vee$. Further, as maps on $D^\vee_{\xi,\ell,\infty}(\pi)^{bd}=\mathfrak{Y}(\mathcal{C}_0)$ we have $\res_{\mathcal{U}_{g^{-1}}}=\res^{\mathcal{C}_0}_{g\mathcal{C}_0\cap \mathcal{C}_0}$. Putting these together our equation to show reads $$\res_{\mathcal{U}_{g^{-1}}}\circ\pr\circ\pr_{SV}(\mu)=\mathcal{H}_g(\pr\circ\pr_{SV}(g^{-1}\mu))\ .$$ We want to write $\mathcal{H}_g$ as the limit of the maps $\mathcal{H}_g^{(k)}$, so we set $\mathcal{U}_g^{(k)}:= \{u\in J(N_0/s^kN_0s^{-k})\mid x_u\in g^{-1}\mathcal{C}_0\cap\mathcal{C}_0\}$ and compute
\begin{align}
\mathcal{H}_g^{(k)}\circ\widetilde{\pr_W}(g^{-1}\mu)=\notag\\
=\sum_{u\in \mathcal{U}_g^{(k)}}n(g,u)\varphi_{t(g,u)s^k}\circ\psi_s^k(u^{-1}\widetilde{\pr_W}(g^{-1}\mu))=\notag\\
=\sum_{u\in \mathcal{U}_g^{(k)}}n(g,u)\varphi_{t(g,u)s^k}\circ\widetilde{\pr_W}(s^{-k}u^{-1}g^{-1}\mu)=\notag\\
=\sum_{u\in \mathcal{U}_g^{(k)}}\iota_{t(g,u)s^k,\infty}(n(g,u)\otimes_{s^k}\pr_W(s^{-k}u^{-1}g^{-1}\mu))=\notag\\
=\sum_{u\in \mathcal{U}_g^{(k)}}\iota_{t(g,u)s^k,\infty}(n(g,u)\otimes_{s^k}\pr_W(s^{-k}\overline{n}(g,u)^{-1}t(g,u)^{-1}n(g,u)^{-1}\mu))\notag\\
=\sum_{u\in \mathcal{U}_g^{(k)}}\iota_{t(g,u)s^k,\infty}(n(g,u)\otimes_{s^k}\pr_W((s^{-k}\overline{n}(g,u)^{-1}s^k)t(g,u)^{-1}s^{-k}n(g,u)^{-1}\mu))\label{HgtildeW}
\end{align}
where $\iota_{t(g,u)s^k,\infty}\colon \varphi_{t(g,u)s^k}^*W^\vee\to \varinjlim_t\varphi_t^*W^\vee=\widetilde{W^\vee}$ is the natural map. By Lemma \ref{k0g} we have $$s^{-k}\overline{n}(g,u)^{-1}s^k\in s^{-k+k_0}(G_0\cap \overline{N})s^{k-k_0}\leq U^{(k-k_0)}\ .$$ As $\pi$ is a smooth representation of $G$ and $W_0$ is finite, there exists an integer $k_2=k_2(W_0)$ such that for all $k'\geq k_2$ the subgroup $U^{(k')}$ acts trivially on $W_0$. By Lemma \ref{smallt} we deduce 
\begin{align*}
\pr_W(s^{-k}\overline{n}(g,u)^{-1}t(g,u)^{-1}n(g,u)^{-1}\mu)\mid_{W_r}= \pr_W(s^{-k}t(g,u)^{-1}n(g,u)^{-1}\mu)\mid_{W_r}
\end{align*}
for all $r\leq \frac{k-k_2-k_0}{n_0}$ since $N_0$ normalizes $U^{(k-k_0)}$. Therefore by Lemma \ref{ngubij} and \eqref{HgtildeW} we obtain
\begin{align*}
\mathcal{H}_g^{(k)}\circ\widetilde{\pr_W}(g^{-1}\mu)-\res_{\mathcal{U}_{g^{-1}}}\circ\widetilde{\pr_W}(\mu)=\\
=\mathcal{H}_g^{(k)}\circ\widetilde{\pr_W}(g^{-1}\mu)-\sum_{u\in \mathcal{U}_g^{(k)}}n(g,u)\varphi_{t(g,u)s^k}\circ\psi_{t(g,u)s^{k}}(n(g,u)^{-1}\widetilde{\pr_W}(\mu))=\\
=\sum_{u\in \mathcal{U}_g^{(k)}}\iota(n(g,u)\otimes\pr_W((s^{-k}\overline{n}(g,u)^{-1}s^k-1)s^{-k}t(g,u)^{-1}n(g,u)^{-1}\mu))\\
\in \sum_{u\in \mathcal{U}_g^{(k)}}\iota(\Lambda(N_0)\otimes_{\Lambda(N_0),\varphi_{t(g,u)s^k}}(W/W_r)^\vee)
\end{align*}
where $\iota=\iota_{t(g,u)s^k,\infty}$.

Finally, the sets $O(M,l,l')\subset D^\vee_{\xi,\ell,\infty}(\pi)$ in \eqref{OMll'} form a system of open neighbourhoods of $0$ in $D^\vee_{\xi,\ell,\infty}(\pi)$. Moreover, for any fixed choice $l,l'\geq 0$ and $M\in \mathcal{M}(\pi^{H_0})$ there exists an integer $k_1\geq 0$ such that for all $r\geq k_1$ and $u\in\mathcal{U}_g$ we have $$\pr_{W,F_{t'(g,u)}^*M_l}((W/W_r)^\vee)\subseteq \Lambda(N_0/H_l)\otimes_{u_\alpha}X^{l'}(F_{t'(g,u)}^*M)^\vee[1/X]^{++}$$ 
(see Lemmata \ref{WrzeroM++} and \ref{finFt'}). Note that the composite map $D^\vee_{\xi,\ell,\infty}(\pi)\overset{\varphi_{t(g,u)s^k}}{\rightarrow}D^\vee_{\xi,\ell,\infty}(\pi)\overset{f_{M,0}}{\twoheadrightarrow} M^\vee[1/X]$ factors through the $\varphi_s$-equivariant map $$((1\otimes F_{t(g,u)s^k})^\vee[1/X])^{-1}\colon (F_{t'(g,u)}^*M)^\vee[1/X]\to M^\vee[1/X]$$ mapping $X^{l'}(F_{t'(g,u)}^*M)^\vee[1/X]^{++}$ into $X^{l'}M^\vee[1/X]^{++}$. So we deduce that
\begin{equation*}
\mathcal{H}_g^{(k)}\circ\pr\circ\pr_{SV}(g^{-1}\mu)-\res_{\mathcal{U}_{g^{-1}}}\circ\pr\circ\pr_{SV}(\mu)
\end{equation*}
lies in $O(M,l,l')$ for all $k\geq k_0+k_2+n_0k_1$ and any choice of $J(N_0/s^kN_0s^{-k})$. The result follows by taking the limit $\mathcal{H}_g=\lim_{k\to\infty}\mathcal{H}_g^{(k)}$.
\qed\end{proof}

Now for any fixed $\mu\in\pi^\vee$ consider the the elements $\beta_{g\mathcal{C}_0}(\mu)\in\mathfrak{Y}(g\mathcal{C}_0)$ for $g\in G$. By Proposition \ref{betag} we also deduce 
\begin{align*}
\res^{g\mathcal{C}_0}_{g\mathcal{C}_0\cap h\mathcal{C}_0}\circ\beta_{g\mathcal{C}_0}(\mu)
=\res_{g\mathcal{C}_0\cap h\mathcal{C}_0}^{g\mathcal{C}_0}(g\beta_{\mathcal{C}_0}(g^{-1}\mu))=\\
=g\res^{\mathcal{C}_0}_{\mathcal{C}_0\cap g^{-1}h\mathcal{C}_0}\circ\beta_{\mathcal{C}_0}(g^{-1}\mu)
\overset{\ref{betag}}{=}g\res^{g^{-1}h\mathcal{C}_0}_{\mathcal{C}_0\cap g^{-1}h\mathcal{C}_0}\circ\beta_{g^{-1}h\mathcal{C}_0}(g^{-1}\mu)=\\
=\res^{h\mathcal{C}_0}_{g\mathcal{C}_0\cap h\mathcal{C}_0}(g(g^{-1}h)\beta_{\mathcal{C}_0}((g^{-1}h)^{-1}g^{-1}\mu))
=\res^{h\mathcal{C}_0}_{g\mathcal{C}_0\cap h\mathcal{C}_0}(h\beta_{\mathcal{C}_0}(h^{-1}\mu))=\\
=\res_{g\mathcal{C}_0\cap h\mathcal{C}_0}^{h\mathcal{C}_0}\circ\beta_{h\mathcal{C}_0}(\mu)
\end{align*}
for all $g,h\in G$. Since $\mathfrak{Y}$ is a sheaf and we have $\bigcup_{g\in G}g\mathcal{C}_0=G/B$, there exists a unique element $\beta_{G/B}(\mu)$ in the global sections $\mathfrak{Y}(G/B)$ with $$\res_{g\mathcal{C}_0}^{G/B}(\beta_{G/B}(\mu))=\beta_{g\mathcal{C}_0}(\mu)$$ for all $g\in G_0$. So we obtained a map $\beta_{G/B}\colon \pi^\vee\to \mathfrak{Y}(G/B)$. Our main result in this section is the following

\begin{thm}\label{main}
The family of morphisms $\beta_{G/B,\pi}$ for smooth, admissible $o$-torsion representations $\pi$ of $G$ of finite length form a natural transformation between the functors $(\cdot)^\vee$ and $\mathfrak{Y}_{\alpha,\cdot}(G/B)$. Whenever $D^\vee_{\xi,\ell}(\pi)$ is nonzero, the map $\beta_{G/B,\pi}$ is nonzero either. In particular, if we further assume that $\pi$ is irreducible then $\beta_{G/B}$ is injective.
\end{thm}
\begin{proof}
At first we need to check that $\beta_{G/B,\pi}\colon \pi^\vee\to \mathfrak{Y}_{\alpha,\pi}(G/B)$ is $G$-equivari-ant and continuous for all $\pi$. For $g,h\in G$ and $\mu\in\pi^\vee$ we compute
\begin{align*}
\res_{g\mathcal{C}_0}^{G/B}(\beta_{G/B}(h\mu))=\beta_{g\mathcal{C}_0}(h\mu)=g\beta_{\mathcal{C}_0}(g^{-1}h\mu)=\\
=h\beta_{h^{-1}g\mathcal{C}_0}(\mu)=h\res^{G/B}_{h^{-1}g\mathcal{C}_0}\circ\beta_{G/B}(\mu)=\res^{G/B}_{g\mathcal{C}_0}(h\beta_{G/B}(\mu))
\end{align*}
showing that $\beta_{G/B}(h\mu)$ and $h\beta_{G/B}(\mu)$ are equal locally everywhere, so they are equal globally, too. The continuity follows from the fact that $\beta_{g\mathcal{C}_0}$ is continuous for each $g\in G$. 

By Thm.\ 9.24 in \cite{SVZ} the assignment $\pi\mapsto\mathfrak{Y}_{\alpha,\pi}$ is functorial. Moreover, by definition we have $\beta_{g\mathcal{C}_0,\pi}=(g\cdot)\circ\beta_{\mathcal{C}_0,\pi}\circ(g^{-1}\cdot)$ so we are reduced to showing the naturality of $\beta_{\mathcal{C}_0,\cdot}$. This follows from the fact that for any morphism $f\colon \pi\to\pi'$ of smooth, admissible $o$-torsion representations of $G$ of finite length and $M_k\in\mathcal{M}_k(\pi^{H_k})$ for any $k\geq 0$ we have $f(M_k)\in\mathcal{M}_k(\pi'^{H_k})$.
\qed\end{proof}

\section*{Acknowledgements}
Our debt to the works of Christophe Breuil \cite{B}, Pierre Colmez \cite{Mira} \cite{C}, Peter Schneider, and Marie-France Vigneras \cite{SVig} \cite{SVZ} will be obvious to the reader. We would especially like to thank Breuil for discussions on the exactness properties of his functor and its dependence on the choice of $\ell$. We would also like to thank P.\ Schneider for discussions on the topic.

\end{document}